\documentclass[a4paper,11pt]{amsart}
\usepackage{amsmath,amssymb,amsthm,mathtools,mathrsfs}
\usepackage{graphicx}
\usepackage{stmaryrd}
\usepackage{microtype}
\usepackage{enumitem}
\usepackage{booktabs,longtable,array}
\usepackage{xcolor}
\usepackage{hyperref}
\usepackage{cleveref}
\usepackage[T1]{fontenc}
\usepackage{lmodern}
\hypersetup{colorlinks=true,linkcolor=blue!55!black,citecolor=green!45!black,urlcolor=blue!65!black}

\allowdisplaybreaks
\newtheorem{theorem}{Theorem}[section]
\newtheorem{proposition}[theorem]{Proposition}
\newtheorem{propositiondefinition}[theorem]{Proposition--Definition}
\newtheorem{lemma}[theorem]{Lemma}
\newtheorem{corollary}[theorem]{Corollary}

\newtheorem{definition}[theorem]{Definition}
\newtheorem{remark}[theorem]{Remark}

\newtheorem{example}[theorem]{Example}

\newcommand{\g}{\mathfrak g}
\newcommand{\dloop}{\mathfrak l}
\newcommand{\dplus}{\mathfrak l_+}
\newcommand{\dminus}{\mathfrak l_-}

\newcommand{\cO}{\mathcal O}
\newcommand{\cA}{\mathcal A}
\newcommand{\cT}{\mathcal T}

\newcommand{\Res}{\operatorname{Res}}
\newcommand{\Tr}{\operatorname{Tr}}
\newcommand{\ad}{\operatorname{ad}}
\newcommand{\Ad}{\operatorname{Ad}}
\newcommand{\Der}{\operatorname{Der}}
\newcommand{\Sym}{\operatorname{Sym}}
\newcommand{\Zhu}{\operatorname{Zhu}}
\newcommand{\gr}{\operatorname{gr}}

\newcommand{\id}{\mathrm{id}}
\newcommand{\dd}{\mathrm d}
\newcommand{\Uenv}{U}
\newcommand{\normal}[1]{\mathopen{:}#1\mathclose{:}}
\newcommand{\Lcur}{\mathsf L}

\newcommand{\Tu}{\partial_u}

\newcommand{\bbC}{\mathbb C}
\newcommand{\kk}{\Bbbk}
\newcommand{\Ych}{\mathcal Y^{\mathrm{ch}}}
\newcommand{\whkappa}{%
  \mathchoice
    {\widehat{\kappa}}%
    {\widehat{\kappa}}%
    {\scalebox{0.72}{$\widehat{\kappa}$}}%
    {\scalebox{0.52}{$\widehat{\kappa}$}}%
}

\title{Chiral Extensions of Cotangent Yangian} 
\author{Keyou Zeng}
\address{Center of Mathematical Sciences and Applications, Harvard University, Cambridge, MA 02138, USA}
\email{kzeng@cmsa.fas.harvard.edu}
\date{}
\begin{document}

\begin{abstract}
We construct chiral extensions of the cotangent Yangian using vertex algebroids. At zero level, the construction yields a vertex algebra with a compatible coproduct, whose Zhu reduction recovers the cotangent Yangian of Abedin and Niu. We also construct vertex algebra extensions at non-zero levels and study level-additive coproducts. The zero-level vertex algebra admits two natural quasi-conformal structures. We compute the corresponding sphere one-point conformal blocks and equip each block space with an associative product induced by the coproduct. For the first quasi-conformal structure, the block space can be identified with the dual of a certain Lie algebroid homology; for the second, we describe the block algebra and its classical limit in terms of mini-twistor geometry.
\end{abstract}
\maketitle

\tableofcontents

\section{Introduction}
\label{sec:introduction}
Vertex algebras provide an algebraic axiomatization of two-dimensional conformal field theory, establishing fundamental connections between infinite-dimensional Lie theory, representation theory, and quantum groups. Their relationship with quantum groups takes several forms. In one direction, certain vertex algebras give rise to Yangian-type algebras through Zhu reduction: the (H-twisted) Zhu algebra of an affine \(\mathcal W\)-algebra is the corresponding finite \(\mathcal W\)-algebra \cite{DeSoleKacZhu,Arakawa2007}, which in type \(A\) is isomorphic to a truncation of a shifted Yangian \cite{BrundanKleshchev2006}, while matrix-extended \(\mathcal W_\infty\) algebras have Zhu algebras isomorphic to Yangians \cite{GaiottoRapcakZhou}. In another direction, Etingof and Kazhdan construct quantum vertex algebras from spectral \(R\)-matrices \cite{EtingofKazhdanV}, replacing ordinary locality with a braided exchange relation and incorporating spectral parameters into the state--field correspondence. 

In this paper, we construct chiral extensions of the cotangent Yangian following the former approach, albeit through a distinct geometric method. Our starting point is the realization of the cotangent Yangian via polynomial currents acting by dressing transformations on a negative-loop group. This dressing action naturally defines an action Lie algebroid. By applying the theory of vertex algebroids, we chiralize this Lie algebroid to obtain an ordinary vertex algebra in an independent coordinate $z$, keeping the separate Yangian spectral parameter $u$; in particular, unlike in quantum vertex-algebra constructions, vertex translation is not identified with spectral translation $\partial_u$. At zero level, we equip this algebra with a compatible coproduct and counit. Zhu reduction then recovers the dense evaluated cotangent Yangian of Abedin and Niu \cite{AbedinNiu}. 

The coproduct has substantial algebraic consequences beyond Zhu reduction. When equipped with a quasi-conformal structure, the vertex algebra globalizes on an algebraic curve, allowing us to define the spaces of conformal blocks and coinvariants. Under compatibility conditions with the coproduct, the space of sphere conformal blocks acquires an associative convolution product. We establish two compatible quasi-conformal structures on the zero-level chiral cotangent Yangian and determine the resulting block algebras. In the first quasi-conformal structure, its coinvariants can be described in terms of a Lie algebroid homology, and detects only the center of $\mathfrak g$. For the second structure, which incorporates spectral dilation, we obtain a noncommutative convolution algebra governed by mini-twistor geometry.

We now describe the construction more precisely. Let $\mathfrak g$ be a finite-dimensional complex Lie algebra. The formal loop algebra admits the standard rational splitting
\[
  \mathfrak g((u^{-1})) = \mathfrak g[u] \oplus u^{-1}\mathfrak g[[u^{-1}]].
\]
The negative loop subalgebra integrates to the prounipotent group $G_- = \exp\bigl(u^{-1}\mathfrak g[[u^{-1}]]\bigr)$, which can be identified with the first congruence subgroup $G_1[[u^{-1}]]$. Its coordinate ring $A = \mathcal O(G_-)$ consists of polynomial functions depending on finitely many loop coefficients. The splitting determines an infinitesimal dressing action $\rho: \mathfrak g[u] \to \operatorname{Der}(A)$, and hence an action Lie algebroid $A \otimes \mathfrak g[u]$. When $\mathfrak g$ is equipped with a nondegenerate invariant form, it identifies with the cotangent Lie algebroid of the rational Poisson--Lie structure on $G_-$.

The theory of vertex algebroids provides the bridge from this classical geometric data to a vertex algebra. For each invariant symmetric bilinear form $\widehat\kappa$ on $\mathfrak g[u]$, the extension theorem of Gorbounov, Malikov, and Schechtman \cite{GMS} produces a vertex algebroid $\mathcal A_{\whkappa}(A,\mathfrak g[u])$. We define the corresponding chiral cotangent Yangian as its universal vertex envelope:
\[
  \Ych_{\whkappa}(T^*\mathfrak g) := \Uenv\bigl(\mathcal A_{\whkappa}(A,\mathfrak g[u])\bigr).
\]
It is generated by functions $f \in A$ and currents $\Lcur_\xi$ ($\xi \in \mathfrak g[u]$), with their mixed operator product expansions determined by the dressing action.

While the extension theorem constructs the vertex algebra, it does not by itself provide a coalgebraic structure. At level zero, however, group multiplication $m$ on $G_-$ lifts naturally to a morphism of vertex algebroids. Setting $\Ych(T^*\mathfrak g) := \Ych_0(T^*\mathfrak g)$, the coproduct $\Delta: \Ych(T^*\mathfrak g) \to \Ych(T^*\mathfrak g) \otimes \Ych(T^*\mathfrak g)$ and counit $\epsilon: \Ych(T^*\mathfrak g) \to \kk$ are determined on generators by
\[
  \Delta(f) = m^*f,
  \qquad
  \Delta(\Lcur_\xi) = \Lcur_\xi \otimes 1 + \Lcur_{M_{g_1}\xi}^{(2)},
  \qquad
  \epsilon(f) = f(e),
  \qquad
  \epsilon(\Lcur_\xi) = 0,
\]
for $f \in A$ and $\xi \in \mathfrak g[u]$, where $M_g\xi = (\operatorname{Ad}_{g^{-1}}\xi)_+$ is the projection of the adjoint action, and $\Lcur_{M_{g_1}\xi}^{(2)}$ denotes the current in the second tensor factor with coefficients evaluated at $g_1 \in G_-$ in the first tensor factor.  Together, these equip the zero-level chiral cotangent Yangian with the structure of a graded local vertex bialgebra (Thm.~\ref{thm:zero-algebroid-coprod} and Cor.~\ref{cor:zero-bialg}).

Zhu reduction connects these vertex algebras directly to associative algebras. While the general theorem of Arakawa, Chebotarov, and Malikov \cite{ArakawaChebotarovMalikov} yields an associative algebra isomorphism $\operatorname{Zhu}\bigl(\Ych_{\whkappa}(T^*\mathfrak g)\bigr) \cong \mathcal O(G_-) \#_\rho U(\mathfrak g[u])$ at any level, at level zero this map preserves the coalgebra structure, and we have an isomorphism of bialgebras:
\[
  \operatorname{Zhu}\bigl(\Ych(T^*\mathfrak g)\bigr) \cong Y^\circ_1(T^*\mathfrak g)
\]
recovering the dense evaluated cotangent Yangian of Abedin and Niu \cite{AbedinNiu} (Cor.~\ref{cor:Zhu-dense-bialg-comp}).

At nonzero levels, lifting group multiplication requires level additivity $\widehat\kappa_0 = \widehat\kappa_1 + \widehat\kappa_2$. This yields level-additive coproducts $\Delta_{\whkappa_1,\whkappa_2}: \Ych_{\whkappa_1+\whkappa_2}(T^*\mathfrak g) \to \Ych_{\whkappa_1}(T^*\mathfrak g) \otimes \Ych_{\whkappa_2}(T^*\mathfrak g)$, and in particular endows each vertex algebra $\Ych_{\whkappa}(T^*\mathfrak g)$ with a comodule structure of the zero-level vertex bialgebra $\Ych(T^*\mathfrak g)$.

We next turn to the global consequences of the coproduct. A (quasi-)conformal structure specifies the transformation behavior of a vertex algebra under formal coordinate changes, allowing it to globalize to an algebraic curve \cite{FrenkelBenZvi}. Although $\Ych(T^*\mathfrak g)$ possesses no conformal vector, it admits two natural quasi-conformal structures:
\begin{enumerate}
  \item The \emph{standard nonnegative structure} $\{\mathscr{L}_n\}_{n\geq -1}$, where functions in $A$ have conformal weight zero and currents $\Lcur_\xi$ have conformal weight one.
  \item The \emph{minitwistor structure} $\{\mathscr{L}^{\mathrm{MT}}_n\}_{n\geq -1}$, defined by combining the preceding standard grading with spectral dilation $E_u$:
  \[
    \mathscr L_0^{\mathrm{MT}} = \mathscr L_0 - E_u,
  \]
\end{enumerate}

Using a cocycle in $Z^1(\g[u],A)$ arising from spectral translation, we further extend these structures to two one-parameter families $\{\mathscr L_n^\theta\}_{n\geq-1}$ and $\{\mathscr L_n^{\mathrm{MT},\theta}\}_{n\geq-1}$ for $\theta\in\kk$ (Sec.~\ref{sec:twisted-qc}). These twists preserve the respective grading operators, while the generators $\Lcur_{\xi}$ need no longer be primary. For both families of quasi-conformal structures, we show that the coproduct descends to the sphere one-point coinvariants, endowing them with a coalgebra structure, while the spaces of conformal blocks $\mathscr{B}$ acquire an associative convolution product:
\[
  (\lambda \star \mu)(v) = (\lambda \otimes \mu)\Delta(v), \quad \lambda,\mu \in \mathscr{B},\; v \in \Ych(T^\ast\mathfrak g).
\]

For the standard quasi-conformal structure, the coinvariants are canonically described by the degree-zero Lie algebroid homology of the cotangent Lie algebroid of $G_-$ valued in certain module. For the case of reductive $\mathfrak g$ and $\theta = 0$, this homology depends only on the center $\mathfrak z(\mathfrak g)$: writing $Z_- = \exp\bigl(u^{-1}\mathfrak z(\mathfrak g)[[u^{-1}]]\bigr) \subset G_-$ for the central negative-loop subgroup, the coinvariants identify with $\mathcal O(Z_-)$, and the block algebra is isomorphic to the commutative convolution algebra $\mathcal O(Z_-)^\vee$.

The mini-twistor quasi-conformal structure yields a substantially richer, noncommutative block algebra at level zero. Its coordinate-change rule identifies $u$ with the fibre coordinate of $\mathcal O_{\mathbb P^1}(2)$, combining the vertex and spectral coordinates geometrically into the mini-twistor space $\mathbb{MT} = \operatorname{Tot}_{\mathbb P^1}\mathcal O(2)$. Cover $\mathbb P^1$ by the standard affine charts $U_0, U_\infty$. Over the chart overlap, the spaces of $\mathfrak g$-valued regular functions and principal parts along the zero section form a matched pair of Lie algebras $(\mathfrak D_+, \mathfrak D_-)$, concretely realized by
\[
  \mathfrak D_+ = \mathfrak g[u, t, t^{-1}], \qquad \mathfrak D_- = u^{-1}\mathfrak g[u^{-1}, t, t^{-1}].
\]
We show that every conformal block $\lambda \in \mathscr B^{\mathrm{MT}}$ is represented by a linear map $F_\lambda: U(\mathfrak D_+) \to \widehat{U(\mathfrak D_-)}$ satisfying a two-chart invariance condition under the chart subalgebras $\mathfrak D_+^0 = \mathfrak g[u,t]$ and $\mathfrak D_+^\infty = \mathfrak g[w,z]$ (with $w = u t^{-2}, z = t^{-1}$). The convolution of such maps is governed by the Koppinen smash product \cite{BrzezinskiEntwiningCohomology,SchauenburgDoiKoppinen}:
\[
  (F \diamond G)(X) = \sum_{(X)} Y_{(2)}\, G\bigl(X_{(2)} \triangleleft Y_{(1)}\bigr), \qquad Y = F(X_{(1)}),
\]
where $\triangleleft$ denotes the right action of $U(\mathfrak D_-)$ on $U(\mathfrak D_+)$ induced by the matched pair decomposition. Thm.~\ref{thm:mt-kop-comp} proves that $F_{\lambda \star \mu} = F_\lambda \diamond F_\mu$, identifying $\mathscr B^{\mathrm{MT}}$ with the closed subalgebra of two-chart invariant maps in $\operatorname{Hom}(U(\mathfrak D_+), \widehat{U(\mathfrak D_-)})$.

The polynomial degree of $A$ induces an almost-commutative filtration $P$ on $\mathscr B^{\mathrm{MT}}$. Geometrically, the classical limit is described by the sheaf of Lie algebras
\[
  \mathscr{G}_{\mathbb{MT}} = (\mathfrak g \otimes \mathcal O_{\mathbb{MT}}) \ltimes (\mathfrak g^\vee \otimes K_{\mathbb{MT}}),
\]
where $\mathfrak g \otimes \mathcal O_{\mathbb{MT}}$ acts coadjointly on $\mathfrak g^\vee \otimes K_{\mathbb{MT}}$. Via the two-chart \v{C}ech cover, the quotient of overlap currents by the two chart subalgebras identifies with the cohomology $H^1(\mathbb{MT}, \mathscr{G}_{\mathbb{MT}})$. Taking into account the residual zero-mode symmetry under $H^0(\mathbb{MT}, \mathscr{G}_{\mathbb{MT}}) = \mathfrak g$, this yields an isomorphism of (completed) commutative algebras (Thm.~\ref{thm:mt-classical-limit} and Cor.~\ref{cor:mt-cohomology}):
\[
  \widehat{\operatorname{gr}}_P \mathscr B^{\mathrm{MT}} \cong \left( \operatorname{Sym}\bigl(H^1(\mathbb{MT},\mathscr{G}_{\mathbb{MT}})\bigr)^\vee \right)^{\mathfrak g}.
\]

The mini-twistor geometry reflects the physical motivation of our construction, inspired by the work of Costello and Paquette \cite{CostelloPaquetteCelestial} on a six-dimensional field theory on twistor space. Here, the chiral cotangent Yangian is expected to arise from a five-dimensional field theory on the mini-twistor correspondence space---an $\mathbb R$-fibration over the mini-twistor space $\mathbb{MT}$ locally modeled by $\mathbb R \times \mathbb C_z \times \mathbb C_u$. The two holomorphic directions $z,u$ provide the vertex and spectral coordinates, while defect fusion along the $\mathbb R$-fibres induces the coproduct. Although the underlying chiral algebras differ, both frameworks share the central feature that twistor geometry governs the global structure of conformal blocks. While motivated by these field-theoretic origins, all constructions and arguments in the present paper are purely algebraic.

\section{Vertex algebroid}
\label{sec:vertexextensions}

In this section, we review the theory of vertex algebroids and related notions following Gorbounov, Malikov, and Schechtman \cite{GMS}. Readers already familiar with these materials may skip this section.

\subsection{Extended Lie algebroids}
Throughout the algebraic constructions, let $\kk$ be a field of characteristic zero. Let $A$ be a commutative $\kk$-algebra. We first review the notion of an extended Lie $A$-algebroid introduced in \cite{GMS}.

\begin{definition}	
	An \emph{extended Lie $A$-algebroid} $\mathcal{T}$ is a quintuple
	\[
	(A,T,\Omega,\partial,\langle-,-\rangle)
	\]
	consisting of a Lie $A$-algebroid $T$, an $A$-module $\Omega$ with a module structure over the Lie algebra $T$, a derivation $\partial:A\to\Omega$ which is also a morphism of $T$-modules, and an $A$-bilinear pairing $\langle-,-\rangle:T\times\Omega\to A$.
	These data must satisfy the following identities:
	\begin{align*}
		\langle\tau,\partial a\rangle&=\tau(a),
\\
		\tau(a\omega)&=\tau(a)\omega+a\tau(\omega),
\\
		(a\tau)(\omega)&=a\tau(\omega)+\langle\tau,\omega\rangle\partial a,
\\
		\tau\langle\nu,\omega\rangle
		&=\langle[\tau,\nu],\omega\rangle+\langle\nu,\tau(\omega)\rangle,
	\end{align*}
	with $a \in A$, $\tau,\nu \in T$ and $\omega \in \Omega$
\end{definition}

\begin{example}
\label{ex:action-ext-alg}
Let a Lie algebra $\mathfrak a$ act on $A$ by derivations through $\rho:\mathfrak a\to\Der(A)$. The associated \emph{action Lie algebroid} is $T=A\otimes_{\kk}\mathfrak a$, with anchor induced by $\rho$. There is also a canonical extended Lie algebroid associated to it. Set $\Omega=\Omega^1_{A}$ the module of K\"ahler differentials (over $\kk$) and let $\partial = \dd:A\to\Omega$ be the universal derivation. The $T$-action on $\Omega$ and the pairing $\langle-,-\rangle$ are induced by the standard $\Der(A)$ action on $\Omega^1_{A}$ and pairing $\Der(A)\times\Omega^1_{A}\to A$. These data satisfy the axioms of an extended Lie algebroid. We will denote this extended Lie algebroid by $\cT(A, \mathfrak a)$.
\end{example}

Let $\mathcal T=(A,T,\Omega,\partial,\langle-,-\rangle)$ be an extended Lie algebroid.  Following \cite[Sec.~1.3]{GMS}, we define its \emph{de Rham--Chevalley complex} $\Omega^{\bullet}_{\mathrm{DR}}(\mathcal T)$ as follows. In degrees zero and one, $\Omega^0_{\mathrm{DR}}(\mathcal T)=A$, $\Omega^1_{\mathrm{DR}}(\mathcal T)=\Omega$. For $n\geq2$, let $\Omega^n_{\mathrm{DR}}(\mathcal T)$ be the $A$-module of $A$-linear maps $h:\Lambda_A^{n-1}T\longrightarrow\Omega$ for which
\[
(\tau_1,\ldots,\tau_n) \mapsto \langle\tau_1,h(\tau_2,\ldots,\tau_n)\rangle
\] is alternating in all $n$ arguments.

For $n\geq1$ and $h\in\Omega^n_{\mathrm{DR}}(\mathcal T)$, let $d_{\mathrm{Lie}}h$ denote the Chevalley-Eilenberg differential for the underlying Lie algebra $T$ acting on $\Omega$: 
\begin{align*}
(d_{\mathrm{Lie}}h)(\tau_1,\ldots,\tau_n)
={}& \sum_{i=1}^{n}(-1)^{i+1} \tau_i\!\left(
h(\tau_1,\ldots,\widehat{\tau_i},\ldots,\tau_n) \right)\notag\\
&+ \sum_{1\leq i<j\leq n}(-1)^{i+j} h\!\left( [\tau_i,\tau_j], \tau_1,\ldots,\widehat{\tau_i},\ldots, \widehat{\tau_j},\ldots,\tau_n
\right).
\end{align*}
The differential of the complex $\Omega^{\bullet}_{\mathrm{DR}}(\mathcal T)$ is simply
\begin{align*}
d_{\mathrm{DR}}a&=-\partial a,
\\
(d_{\mathrm{DR}}h)(\tau_1,\ldots,\tau_n)
&=(d_{\mathrm{Lie}}h)(\tau_1,\ldots,\tau_n)
-\partial\langle\tau_1,h(\tau_2,\ldots,\tau_n)\rangle,
\qquad n\geq1.
\end{align*}
For example,
$(d_{\mathrm{DR}}\omega)(\tau)
=\tau(\omega)-\partial\langle\tau,\omega\rangle$ for
$\omega\in\Omega$.  One has $d_{\mathrm{DR}}^2=0$.

There is a natural cochain map $K^\bullet :\Omega_A^\bullet \to \Omega^{\bullet}_{\mathrm{DR}}(\mathcal T)$, where $(\Omega_A^{\bullet} = \bigwedge^{\bullet}\Omega_A^1, \dd)$ is the algebraic de Rham complex. We define $K^0=-\id_A$ and $K^1: \Omega_A^1\longrightarrow\Omega$ to be the map $a\,\dd b\longmapsto a\,\partial b$ induced by $\partial$.  For $n\geq1$,
\begin{equation*}
K^n(\alpha)(\tau_1,\ldots,\tau_{n-1}):=
K^1\!\left(\iota_{\rho(\tau_{n-1})}\cdots
\iota_{\rho(\tau_1)}\alpha\right),
\qquad
\alpha\in\Omega_A^n.
\end{equation*}
The map $K^1$ intertwines the $T$-actions, so Cartan calculus gives the cochain identity
\begin{equation*}
d_{\mathrm{DR}}K^n=K^{n+1}\dd .
\end{equation*}
Consequently, $K^\bullet$ induces a map $H^\bullet_{\mathrm{dR}}(A) \to H^\bullet_{\mathrm{DR}}(\mathcal{T})$.

\subsection{Vertex algebroid}
\label{sec:valg-def}

\begin{definition}
	A \emph{vertex $A$-algebroid} consists of an extended Lie $A$-algebroid $\mathcal T=(A,T,\Omega,\partial,\langle-,-\rangle)$ equipped with additional data $(\gamma,\chi,c)$, denoted
\[
\cA=(\mathcal T,\gamma,\chi,c),
\]
where:
\begin{enumerate}[label=(\roman*)]
\item $\gamma:A\times T\to\Omega$ is $\kk$-bilinear;
\item $\chi:T\times T\to A$ is a symmetric $\kk$-bilinear pairing;
\item $c:T\times T\to\Omega$ is skew-symmetric and $\kk$-bilinear.
\end{enumerate}
These data satisfy the five identities (A1)--(A5) below. \footnote{Following \cite{GMS}, the pairings $\langle-,-\rangle$ on $T\times\Omega$ and $\chi$ on $T\times T$ may equivalently be combined into a single symmetric bilinear pairing on $T\oplus\Omega$ vanishing on $\Omega\times\Omega$. We keep $\chi$ notationally distinct to avoid ambiguity.}
\end{definition}

For $a,b\in A$ and $\tau,\tau_i\in T$ the identities are
\begin{align}
\gamma(a,b\tau)
&=\gamma(ab,\tau)-a\gamma(b,\tau)
-\tau(a)\partial b-\tau(b)\partial a,
\tag{A1}\label{eq:GMS-A1}\\
\chi(a\tau_1,\tau_2)
&=a\chi(\tau_1,\tau_2)
+\langle\tau_2,\gamma(a,\tau_1)\rangle
-\tau_1\tau_2(a),
\tag{A2}\label{eq:GMS-A2}
\end{align}
and
\begin{align}
 c(a\tau_1,\tau_2)
={}&a c(\tau_1,\tau_2)
+\gamma(a,[\tau_1,\tau_2])
-\gamma(\tau_2(a),\tau_1)
+\tau_2\bigl(\gamma(a,\tau_1)\bigr)
\notag\\
&-\frac12\chi(\tau_1,\tau_2)\partial a
+\frac12\partial\bigl(\tau_1\tau_2(a)\bigr)
-\frac12\partial\langle\tau_2,\gamma(a,\tau_1)\rangle,
\tag{A3}\label{eq:GMS-A3}
\end{align}
while covariance of the $T$--$T$ pairing is
\begin{align}
\chi([\tau_1,\tau_2],\tau_3)
+\chi(\tau_2,[\tau_1,\tau_3])
={}&\tau_1\chi(\tau_2,\tau_3)
-\frac12\tau_2\chi(\tau_1,\tau_3)
-\frac12\tau_3\chi(\tau_1,\tau_2)
\notag\\
&+\langle\tau_2,c(\tau_1,\tau_3)\rangle
+\langle\tau_3,c(\tau_1,\tau_2)\rangle,
\tag{A4}\label{eq:GMS-A4}
\end{align}
and the final Jacobi condition is
\begin{align}
(d_{\mathrm{Lie}}c)(\tau_1,\tau_2,\tau_3)
={}&-\frac12\partial\Bigl(
\chi([\tau_1,\tau_2],\tau_3)
+\chi([\tau_1,\tau_3],\tau_2)
-\chi([\tau_2,\tau_3],\tau_1)
\notag\\
&\hspace{18mm}
-\tau_1\chi(\tau_2,\tau_3)
+\tau_2\chi(\tau_1,\tau_3)
-2\langle\tau_3,c(\tau_1,\tau_2)\rangle
\Bigr).
\tag{A5}\label{eq:GMS-A5}
\end{align}

The following extension theorem from \cite{GMS} will be important for our later constructions.
\begin{theorem}
\label{lem:frame-reconstruct}
Let $(A,T,\Omega,\partial,\langle-,-\rangle)$ be an extended Lie algebroid and let $T_0\subset T$ be a $\kk$-submodule generating $T$ over $A$.  Suppose we prescribe $\kk$-bilinear maps
\begin{equation*}
\gamma_0:A\times T_0\longrightarrow\Omega,
\qquad
\chi_0:\Sym^2T_0\longrightarrow A,
\qquad
c_0:\Lambda^2T_0\longrightarrow\Omega,
\end{equation*}
Then these maps admit at most one extension $(\gamma,\chi,c)$ satisfying (A1)--(A3). For $a,b\in A$ and $\tau,\tau_1,\tau_2\in T_0$, any such extension is necessarily given by 
\begin{align}
\gamma(a,b\tau)
={}&\gamma_0(ab,\tau)-a\gamma_0(b,\tau)
-\tau(a)\partial b-\tau(b)\partial a,
\label{eq:frame-gamma}\\
\chi(a\tau_1,b\tau_2)
={}&ab\chi_0(\tau_1,\tau_2)
+a\langle\tau_1,\gamma_0(b,\tau_2)\rangle
+b\langle\tau_2,\gamma_0(a,\tau_1)\rangle
\notag\\
&-a\tau_2\tau_1(b)-b\tau_1\tau_2(a)
-\tau_1(b)\tau_2(a),
\label{eq:frame-pairing}
\end{align}
and
\begin{equation}\label{eq:frame-c}
  \begin{aligned}
    c&(a\tau_1,b\tau_2) =  ab c_0(\tau_1,\tau_2) +\gamma(ab,[\tau_1,\tau_2]) +\gamma_0(a\tau_1(b),\tau_2) -\gamma_0(b\tau_2(a),\tau_1) \\
 &-a\tau_1\!\left(\gamma_0(b,\tau_2)\right) +b\tau_2\!\left(\gamma_0(a,\tau_1)\right)+\frac12\chi_0(\tau_1,\tau_2)(a\partial b-b\partial a) \\
 &-[\tau_1,\tau_2](a)\partial b -[\tau_1,\tau_2](b)\partial a +\frac12\bigl( \tau_1(b)\partial\tau_2(a) -\tau_2(a)\partial\tau_1(b) \bigr)\\
 &+\frac12\Bigl( -\langle\tau_1,\gamma_0(b,\tau_2)\rangle\partial a +\langle\tau_2,\gamma_0(a,\tau_1)\rangle\partial b +a\partial\langle\tau_1,\gamma_0(b,\tau_2)\rangle -b\partial\langle\tau_2,\gamma_0(a,\tau_1)\rangle \Bigr) \\
&+\frac12\partial\bigl(
b\tau_1\tau_2(a)-a\tau_2\tau_1(b)
\bigr).
  \end{aligned}
\end{equation}
Assume that the displayed formulas give well-defined maps on $T$. If these data satisfy (A4) and (A5) for all $\tau_i\in T_0$, then the resulting maps define a vertex algebroid $\cA = (\mathcal T,\gamma,\chi,c)$ \cite[Thm.~1.9]{GMS}.
\end{theorem}

\begin{example}
\label{ex:affine-level-action-ext}
For the extended action Lie algebroid of Ex.~\ref{ex:action-ext-alg}, and given an invariant symmetric bilinear form $\kappa$ on $\mathfrak a$, prescribe on $T_0=1\otimes\mathfrak a$ the data
$$
\gamma(a,\xi)=0, \qquad \chi(\xi,\eta)=\kappa(\xi,\eta), \qquad c(\xi,\eta)=0,\quad \text{ for } a\in A,\;\xi,\eta\in\mathfrak a.
$$
Here (A4) is the invariance identity $\kappa([\xi,\eta],\zeta)+\kappa(\eta,[\xi,\zeta])=0$, and (A5) is automatic. Thm.~\ref{lem:frame-reconstruct} therefore gives a vertex algebroid. For $a,b\in A$ and $\xi,\eta\in\mathfrak a$, formulas~\eqref{eq:frame-gamma}--\eqref{eq:frame-c} simplify to
\begin{equation}\label{eq:action-ext}
  \begin{aligned}
    \gamma(a,b\xi)&=-\xi(a)\partial b-\xi(b)\partial a,\\
\chi(a\xi,b\eta)&=ab\,\kappa(\xi,\eta)-a\,\eta\xi(b)-b\,\xi\eta(a)-\xi(b)\eta(a),\\
2c(a\xi,b\eta)&=\kappa(\xi,\eta)(a\partial b-b\partial a)-2\bigl([\xi,\eta](a)\partial b+[\xi,\eta](b)\partial a\bigr)\\
&+\xi(b)\partial\eta(a)-\eta(a)\partial\xi(b)
+\partial\bigl(b\,\xi\eta(a)-a\,\eta\xi(b)\bigr).
\end{aligned}
\end{equation}
\end{example}

\subsection{Truncation and universal vertex envelopes}
\label{sec:valg-truncation}
We assume the reader to be familiar with the basic theory of ($\mathbb Z_{\geq 0}$-graded) vertex algebras. We use the notation $V=\bigoplus_{n\geq0}V_n$ for a nonnegatively graded vertex algebra, with vacuum $\mathbf1\in V_0$, translation $\partial$ of degree one, and state--field map $Y(a,z)b=\sum_{n\in\mathbb Z}a_{(n)}b\,z^{-n-1}$. We refer to \cite{FrenkelBenZvi} for detailed introduction to vertex algebras.

For vertex algebroids satisfying $\Omega=A\partial A$, the following construction identifies their data with the data of a degree-$\leq1$ truncated vertex algebra, together with a splitting of the quotient $V_1/\Omega$. More precisely, let $ V=\bigoplus_{n\ge0}V_n$ be a $\mathbb{Z}_{\geq 0}$-graded vertex algebra over $\kk$ and put $A=V_0$. Let $\Omega\subset V_1$ denote the $\kk$-submodule spanned by the
elements $a_{(-1)}\partial b$.  The vertex operation $~\!_{(-1)}$ make this subspace $\Omega$ an $A$-module, even though $V_1$ itself is not an $A$-module in general. Set
\[
T:=V_1/\Omega.
\]
Then $T$ is a Lie $A$-algebroid: the $~\!_{(-1)}$-product gives the $A$-module structure, the $~\!_{(0)}$-product descends to the Lie bracket and the action on $A$, and $\partial:A\to\Omega$ is a derivation.

The vertex algebra $V$ is called splittable if there exists a $\kk$-linear splitting $s:T\to V_1$ of the canonical projection $V_1 \to T$. A \emph{split} vertex algebra is a pair $(V,s)$ in which such a splitting has been chosen. We define a skew symmetric operation $[-,-]: V_1\times V_1 \to V_1$ by
\[
[x,y]:=\frac12\bigl(x_{(0)}y-y_{(0)}x\bigr).
\]
By \cite[Thm.~2.3]{GMS}, these operations define a vertex algebroid structure
\begin{align*}
\gamma(a,\tau)=s(a\tau)-a_{(-1)}s(\tau), \;\; \chi(\tau,\nu)=s(\tau)_{(1)}s(\nu),\;\; c(\tau,\nu)=s([\tau,\nu])-[s(\tau),s(\nu)].
\end{align*}
We write $t(V):=V_0\oplus V_1$ for the $1$-truncated vertex algebra, with the products $~_{(-1)},~_{(0)},~_{(1)}$ inherited from $V$, and $\mathcal A(t(V),s)$ for the vertex algebroid obtained from a chosen splitting $s$.

Conversely, from a vertex algebroid we set
\[
V_0=A,
\qquad
V_1=T\oplus\Omega,
\]
We obtain a split 1-truncated vertex algebra by defining the vertex operations as
\begin{equation}
	a_{(-1)}b=ab, \qquad a_{(-1)}\omega=a\omega, \qquad a_{(-1)}\tau=a\tau-\gamma(a,\tau),
\label{eq:GMS-reconstruct-m1}
\end{equation}
and
\begin{align}
\tau_{(0)}a&=\tau(a),\quad  \tau_{(0)}\omega=\tau(\omega),\quad  \tau_{(0)}\nu =[\tau,\nu]-c(\tau,\nu)+\frac12\partial\chi(\tau,\nu),
\label{eq:GMS-reconstruct-z2}\\
\tau_{(1)}\nu&=\chi(\tau,\nu), \qquad \tau_{(1)}\omega=\langle\tau,\omega\rangle, \qquad \omega_{(1)}\omega'=0,
\label{eq:GMS-reconstruct-1}
\end{align}
for $\tau,\nu\in T$ and $\omega,\omega'\in\Omega$.
The remaining products follow from truncated skew-symmetry. These operations define a $1$-truncated vertex algebra, denoted by $u(\cA)$, with the displayed splitting. This construction is inverse to the construction of $\mathcal{A}(t(V),s)$ when $\Omega=A\partial A$ \cite[Sec.~2.4]{GMS}.

The universal vertex envelope $\Uenv(\cA)$ of a vertex algebroid $\cA$ is the $\mathbb{Z}_{\geq 0}$-graded vertex algebra generated by $A$ in weight zero and
$T\oplus\Omega$ in weight one with the truncated operations \eqref{eq:GMS-reconstruct-m1}--\eqref{eq:GMS-reconstruct-1}. It has the following universal property. For a $\mathbb{Z}_{\geq 0}$-graded vertex algebra $W$, one has
\begin{equation*}
\operatorname{Hom}_{\mathrm{VA}}
\bigl(\Uenv(\cA),W\bigr)
\cong
\operatorname{Hom}_{\mathrm{VA}_{\leq 1}}
\bigl(u(\cA),t(W)\bigr).
\end{equation*}
Here $\mathrm{VA}$ denotes $\mathbb{Z}_{\ge 0}$-graded vertex algebras with degree-preserving morphisms, while $\mathrm{VA}_{\leq 1}$ denotes $1$-truncated vertex algebras with morphisms preserving all truncated operations.  In other words, $\Uenv$ is left adjoint to truncation.

We briefly recall the PBW theorem of vertex envelopes. There is a standard filtration $F$ on $\Uenv(\cA)$ such that $F^r\Uenv(\cA)$ is the span of the monomials containing at most $r$ letters translated from $T$, where letters translated from $A$ or $\Omega$ have order zero. \cite{GMS} constructed a filtration $H$ on $\Uenv(\cA)$ that furter refines $F$ and proved the following. 
\begin{theorem}
\label{thm:vertex-env-PBW}
Assume that $T$ and $\Omega$ are free $A$-modules.  The filtration $H$ is exhaustive, compatible with conformal weight, and functorial in $\cA$.  The generator classes induce an isomorphism of commutative differential $A$-algebras
\begin{equation*}
\Sym_A\!\left(
\bigoplus_{i\geq1}T^{(i)}
\oplus
\bigoplus_{i\geq1}\Omega^{(i)}
\right)
\xrightarrow{\ \sim\ }
\gr_H\Uenv(\cA),
\end{equation*}
where $T^{(i)}$ and $\Omega^{(i)}$ are copies of $T$ and $\Omega$ in conformal weight $i$, represented by their $(i-1)$st translates.
\end{theorem}

Inside the universal envelope $\Uenv(\cA)$, the truncated products determine the operator product expansions among the generators. The algebra $A$ generates a commutative vertex subalgebra with regular OPEs: $a(z)b(w) \sim 0$, $a,b \in A$. The weight-one fields $T$ act on $A$ via the anchor map,
\begin{equation*}
\tau(z)a(w) \sim \frac{\tau(a)(w)}{z-w}, \qquad \tau \in T,\ a \in A,
\end{equation*}
Finally, for sections $\tau,\nu \in T$, their OPE can be expressed in terms of the vertex algebroid operations as:
\begin{equation}
\tau(z)\nu(w) \sim \frac{\chi(\tau,\nu)(w)}{(z-w)^2} + \frac{[\tau,\nu](w) - c(\tau,\nu)(w) + \frac{1}{2}\partial\chi(\tau,\nu)(w)}{z-w}.
\label{eq:GMS-OPE-dictionary}
\end{equation}

\subsection{Morphisms and changes of splitting}
\label{sec:mor-splittings}
The preceding subsection relates split $1$-truncated vertex algebras and vertex algebroids.  We now describe their morphisms.  A morphism of vertex algebroids
\[
g:\cA\longrightarrow\cA'
\]
is described by a quadruple $(g_A,g_T,g_\Omega,h)$, where $(g_A,g_T,g_\Omega)$ is a morphism of the underlying extended Lie algebroids and $h: T\longrightarrow\Omega'$ satisfies
\begin{align}
h(a\tau)
&=g_A(a)h(\tau)-\gamma'(g_A(a),g_T\tau)+g_\Omega\gamma(a,\tau),
\label{eq:GMS-mor-gamma}\\
g_A\chi(\tau_1,\tau_2)
&=\chi'(g_T\tau_1,g_T\tau_2)
+\langle g_T\tau_1,h(\tau_2)\rangle
+\langle g_T\tau_2,h(\tau_1)\rangle,
\label{eq:GMS-mor-pair}
\end{align}
and
\begin{align}
g_\Omega c(\tau_1,\tau_2)
={}&c'(g_T\tau_1,g_T\tau_2)
+\frac12\partial'\langle g_T\tau_1,h(\tau_2)\rangle
-\frac12\partial'\langle g_T\tau_2,h(\tau_1)\rangle
\notag\\
&-g_T\tau_1\bigl(h(\tau_2)\bigr)
+g_T\tau_2\bigl(h(\tau_1)\bigr)
+h([\tau_1,\tau_2]).
\label{eq:GMS-mor-c}
\end{align}
Equivalently, if we regard $(\tau, \eta)$ as elements of $T\oplus \Omega = u(\cA)_1$, the above conditions are precisely the conditions for
\begin{equation*}
(\tau,\eta)
\longmapsto
\bigl(g_T\tau,g_\Omega\eta+h(\tau)\bigr)
\end{equation*}
together with $g_A$, to preserve the split $1$-truncated vertex-algebra operations.   

If $g':\cA'\to\cA''$ has components $(g'_A,g'_T,g'_\Omega,h')$, composition is given by
\begin{equation}\label{eq:v_algebroid_comp}
g'\circ g = \bigl(g'_Ag_A,g'_Tg_T,g'_\Omega g_\Omega,g'_\Omega h+h'g_T\bigr),
\end{equation}
In this way, $u$ provides an equivalence between the category of vertex algebroids and the category of split $1$-truncated vertex algebras (where morphisms preserve the one-form subspaces, but are not required to preserve the chosen splittings).

The following extension theorem from \cite[Thm.~3.8]{GMS} for morphisms is the counterpart of Thm.~\ref{lem:frame-reconstruct}.
\begin{theorem}
\label{thm:GMS-ext-morphisms}
Let $\cA$ and $\cA'$ be vertex algebroids, let $T_0\subset T$ be a $\kk$-submodule generating $T$ as an $A$-module, and let $(g_A,g_T,g_\Omega)$ be a morphism of their underlying extended Lie algebroids.  For a $\kk$-linear map $h:T\to\Omega'$, the quadruple $(g_A,g_T,g_\Omega,h)$ defines a unique vertex-algebroid morphism if and only if \eqref{eq:GMS-mor-gamma} holds for $a\in A$ and $\xi\in T_0$, while \eqref{eq:GMS-mor-pair} and \eqref{eq:GMS-mor-c} hold for $\xi,\eta\in T_0$.
\end{theorem}

We will mostly deal with the specialization of the above theorem for action Lie algebroids. 
\begin{corollary}
\label{cor:frame-mor-crit}
In the setting of the above theorem, suppose that $T=A\otimes_{\kk}\mathfrak a$. Given a $\kk$-linear map $h_0:\mathfrak a\to\Omega'$, there is a unique map $h:T\to\Omega'$ extending $h_0$ and satisfying
\eqref{eq:GMS-mor-gamma}, given by
\begin{equation}
h(a\xi)
=
g_A(a)h_0(\xi)
-\gamma'\bigl(g_A(a),g_T\xi\bigr)
+g_\Omega\gamma(a,\xi).
\label{eq:frame-h-ext}
\end{equation}
Then $(g_A,g_T,g_\Omega,h)$ is a vertex-algebroid morphism if and only if \eqref{eq:GMS-mor-pair} and \eqref{eq:GMS-mor-c} hold for $\xi,\eta\in\mathfrak a$.
\end{corollary}

Moreover, for the extended action Lie algebroids, the underlying morphism $(g_A,g_T,g_{\Omega})$ can also be determined by their data on $T_0 = 1\otimes\mathfrak a$.  Namely, given an algebra map $g_A:A\to A'$ and a $\kk$-linear map $g_{T_0} :\mathfrak a\to T'$ satisfying
\begin{equation*}
\rho'_{g_{T_0}(\xi)}\bigl(g_A(a)\bigr)=g_A\bigl(\rho_\xi(a)\bigr),
\qquad
[g_{T_0}(\xi),g_{T_0}(\eta)]=g_{T_0}([\xi,\eta]),
\end{equation*}
define
\begin{equation*}
g_T(a\xi)=g_A(a)g_{T_0}(\xi),
\qquad
g_\Omega(a\,\dd b)=g_A(a)\,\dd g_A(b).
\end{equation*}
Then a vertex-algebroid lift is determined by $h_0:\mathfrak a\to\Omega'$, subject to the identities \eqref{eq:GMS-mor-pair} and \eqref{eq:GMS-mor-c} for $\xi,\eta\in\mathfrak a$.

The case used very often later is the following morphism associated with a change of splitting data.
\begin{proposition}
\label{prop:change-splitting}
Let $\cA$ be a vertex algebroid and let $\Theta:T\to\Omega$ be a $\kk$-linear map. We keep the underlying extended Lie algebroid fixed and define new vertex algebroid data by
\begin{equation}\label{eq:GMS-change}
  \begin{aligned}
    \gamma^\Theta(a,\tau) &=\gamma(a,\tau)+\Theta(a\tau)-a\Theta(\tau),\\
\chi^\Theta(\tau,\nu) &=\chi(\tau,\nu) +\langle\tau,\Theta(\nu)\rangle +\langle\nu,\Theta(\tau)\rangle,\\
c^\Theta(\tau,\nu) &=c(\tau,\nu)+\Theta([\tau,\nu]) -\tau\Theta(\nu)+\nu\Theta(\tau) +\frac12\partial\langle\tau,\Theta(\nu)\rangle -\frac12\partial\langle\nu,\Theta(\tau)\rangle.
  \end{aligned}
\end{equation}
These formulas define a vertex algebroid $\cA^\Theta$. Moreover,
\begin{equation*}
\Phi_\Theta
:=
(\id_A,\id_T,\id_\Omega,-\Theta):
\cA\xrightarrow{\sim}\cA^\Theta
\end{equation*}
is a vertex-algebroid isomorphism, and it extends uniquely to an isomorphism of universal vertex envelopes.
For successive changes by $\Theta_1$ and $\Theta_2$,
\begin{equation*}
(\cA^{\Theta_1})^{\Theta_2}=\cA^{\Theta_1+\Theta_2},
\qquad
\Phi_{\Theta_2}\circ\Phi_{\Theta_1}
=\Phi_{\Theta_1+\Theta_2},
\qquad
\Phi_\Theta^{-1}=\Phi_{-\Theta}.
\end{equation*}
\end{proposition}

\begin{proof}
Let $\cA^\Theta$ denote the vertex algebroid $(\mathcal T, \gamma^\Theta, \chi^\Theta, c^\Theta)$ defined by \eqref{eq:GMS-change}. A straightforward calculation using the axioms for $\cA$ shows that $\cA^\Theta$ satisfies the vertex-algebroid axioms (A1)--(A5).

By definition, the quadruple $\Phi_\Theta = (\id_A, \id_T, \id_\Omega, -\Theta)$ defines a morphism of vertex algebroids $\cA \to \cA^\Theta$ if and only if the morphism conditions \eqref{eq:GMS-mor-gamma}--\eqref{eq:GMS-mor-c} are satisfied. Substituting the components yields
\begin{align*}
-\Theta(a\tau) &= -a\Theta(\tau) - \gamma^\Theta(a,\tau) + \gamma(a,\tau),\\
\chi(\tau_1,\tau_2) &= \chi^\Theta(\tau_1,\tau_2) - \langle\tau_1,\Theta(\tau_2)\rangle - \langle\tau_2,\Theta(\tau_1)\rangle,\\
c(\tau_1,\tau_2) &= c^\Theta(\tau_1,\tau_2) - \frac12\partial\langle\tau_1,\Theta(\tau_2)\rangle + \frac12\partial\langle\tau_2,\Theta(\tau_1)\rangle + \tau_1\bigl(\Theta(\tau_2)\bigr) - \tau_2\bigl(\Theta(\tau_1)\bigr) - \Theta([\tau_1,\tau_2]).
\end{align*}
Rearranging these equations yields precisely the defining identities \eqref{eq:GMS-change}. Thus $\Phi_\Theta$ is a morphism of vertex algebroids. 

The composition identity follows immediately from the morphism composition formula~\eqref{eq:v_algebroid_comp}, which also shows $(\cA^{\Theta_1})^{\Theta_2}=\cA^{\Theta_1+\Theta_2}$. In particular, $\Phi_\Theta$ is invertible with inverse $\Phi_\Theta^{-1} = \Phi_{-\Theta}$.
\end{proof}

\begin{remark}
If an ambient splittable vertex algebra is given and a chosen splitting $s:T\to V_1$ is replaced by $s^\Theta(\tau)=s(\tau)+\Theta(\tau)$, then the two extracted vertex-algebroid presentations are $\cA$ and $\cA^\Theta$.  Indeed, the same element of $V_1$ is written as 
\begin{equation*}
s(\tau)+\eta
=
s^\Theta(\tau)+\bigl(\eta-\Theta(\tau)\bigr).
\end{equation*}
\end{remark}

\section{Cotangent Yangian Lie Algebroid}
\label{sec:coYangian-alg}

\subsection{The prounipotent negative-loop group}
\label{sec:formal-neg-loop-group}

Throughout this section, $\g$ is a finite-dimensional Lie algebra over $\kk$ and
\begin{equation*}
\dloop=\g((u^{-1})), \qquad \dplus=\g[u],\qquad \dminus=u^{-1}\g[[u^{-1}]].
\end{equation*}

We recall the construction of the pro-unipotent BCH group $G_-$ integrating $\dminus$. Consider the filtration on $\dminus$ by
\begin{equation*}
F^r\dminus=u^{-r}\g[[u^{-1}]],
\qquad r\geq1.
\end{equation*}
The quotient
\begin{equation*}
\dminus^{(N)}
:=
\dminus/F^{N+1}\dminus
\cong
\bigoplus_{r=1}^{N}\g u^{-r}
\end{equation*}
is a finite-dimensional nilpotent Lie algebra. We therefore define
\begin{equation*}
G_-^{(N)} := \exp(\dminus^{(N)})
\end{equation*}
via the BCH formula. Then we define
\begin{equation*}
G_-=\varprojlim_N G_-^{(N)}
\end{equation*}
Thus $G_-$ is a pro-unipotent pro-affine group.  Its algebra of regular
functions is the direct limit
\begin{equation*}
A:=\cO(G_-) = \varinjlim_N\cO(G_-^{(N)}) \cong \Sym(\g^*[u]).
\end{equation*}

We write $g(u)$, or simply $g$, for an element of $G_-$. As usual for a pro-affine group, formulas involving $g$ are understood functorially over commutative k-algebras and coefficientwise in $u^{-1}$. For every commutative $\kk$-algebra $R$, we take $g\in G_-(R)=\varprojlim_N G_-^{(N)}(R)$; equality in $G_-(R)$ means equality after projection to $G_-^{(N)}(R)$ for every $N$. We usually suppress $R$ from the notation.

There is also a loop-group description.  Suppose we have a smooth affine algebraic group $G$ over $\kk$ with $\operatorname{Lie}(G)=\g$. Its first congruence subgroup is
\begin{equation*}
G_1[[u^{-1}]] := \ker\bigl(G[[u^{-1}]]\overset{|_{u^{-1} = 0}}{\longrightarrow} G \bigr).
\end{equation*}
At order $N$, exponential and logarithm identify $\ker\bigl( G[u^{-1}]/(u^{-N-1})\longrightarrow G \bigr)$ with the BCH group of $u^{-1}\g[u^{-1}]/(u^{-N-1})$.  These identifications commute with truncation and yield an isomorphism
\begin{equation*}
G_-\xrightarrow{\sim}G_1[[u^{-1}]].
\end{equation*}
For complex simple $\g$, this is the first congruence subgroup appearing in the commutative filtration of the Yangian \cite{KamnitzerWebsterWeekesYacobi,IlinRybnikov}. That filtration identifies Yangian as a quantization of the associated Poisson algebra of $\mathcal{O}(G_-)$. This Poisson structure will also be important for our later construction.

We record the vector fields induced by translation and dilation of the spectral parameter.
\begin{lemma}
\label{lem:spectral-affine-vf}
Spectral translation $\sigma_s:g(u)\mapsto g(u+s)$ and dilation $\delta_\lambda:g(u)\mapsto g(\lambda^{-1}u)$ define pro-algebraic actions of $\mathbb G_a$ and $\mathbb G_m$ on $G_-$. Differentiating their pullbacks at the identity gives regular vector fields $D_u,E_u\in\Der(A)$, respectively, with $E_u$ the Euler derivation for spectral degree.
\end{lemma}

\begin{proof}
The substitutions preserve the filtration of $\dminus$ and induce algebraic actions on its finite BCH quotients, compatibly with truncation. Their pullbacks therefore act on $A$ and differentiate to the asserted derivations.
\end{proof}

\subsection{Action Lie algebroid from bicrossed product}
\label{sec:bicrossed-action-algs}
Let $\mathfrak p_+$ and $\mathfrak p_-$ be finite-dimensional Lie algebras.  A \emph{matched pair} consists of bilinear maps
\begin{equation*}
\phi_+:\mathfrak p_+\otimes\mathfrak p_-\longrightarrow\mathfrak p_-,
\qquad
\phi_-:\mathfrak p_-\otimes\mathfrak p_+\longrightarrow\mathfrak p_+
\end{equation*}
such that the bracket
\begin{equation*}
[(\xi,v),(\eta,w)]
=\bigl([\xi,\eta]+\phi_-(v)\eta-\phi_-(w)\xi [v,w]+\phi_+(\xi)w-\phi_+(\eta)v\bigr),
\end{equation*}
for $\xi,\eta\in\mathfrak p_+$ and $v,w\in\mathfrak p_-$ satisfies the Jacobi identity. The resulting Lie algebra on $\mathfrak p_+\oplus\mathfrak p_-$ is also called the \emph{bicrossed product} $\mathfrak p_+\bowtie\mathfrak p_-$.

Equivalently, suppose that a Lie algebra $\mathfrak p$ has a vector-space decomposition $\mathfrak p=\mathfrak p_+\oplus\mathfrak p_-$ into Lie subalgebras. If $p_\pm:\mathfrak p\to\mathfrak p_\pm$ are the projections, then
\begin{equation*}
\phi_+(\xi)v=p_-[\xi,v], \qquad \phi_-(v)\xi=p_+[v,\xi]
\end{equation*}
defines a matched pair, and identifies $\mathfrak p_+\bowtie\mathfrak p_-$ with $\mathfrak p$. Therefore, we can use these two descriptions interchangeably.  

Recall that a Manin triple $(\mathfrak p,\mathfrak p_+,\mathfrak p_-)$ consists of a finite-dimensional Lie algebra $\mathfrak p$ with a nondegenerate invariant symmetric form $\kappa_{\mathfrak p}$, together with two complementary $\kappa_{\mathfrak p}$-isotropic Lie subalgebras $\mathfrak p_\pm$. Therefore, every Manin triple gives rise to a Lie algebra matched pair.

We assume that the Lie algebra $\mathfrak p_-$ integrates to a group $\mathsf P_-$ and that the Lie algebra action $\ad|_{\mathfrak p_-}:\mathfrak p_-\longrightarrow\Der(\mathfrak p)$ integrates uniquely to an action:
\begin{equation*}
\Ad:\mathsf P_-\longrightarrow \operatorname{Aut}(\mathfrak p).
\end{equation*}
We can use either smooth or algebraic setting. In the smooth setting we take $\kk=\mathbb R$, and let $\mathsf P_-$ be the simply connected Lie group integrating $\mathfrak p_-$. The above condition is automatic. For later applications to vertex algebroid/algebra, it will also be convenient to work in the algebraic setting. We can instead assume a compatible algebraic integration: an algebraic group $\mathsf P$ with Lie algebra $\mathfrak p$ and a smooth closed algebraic subgroup
$\mathsf P_-\subset\mathsf P$ with Lie algebra $\mathfrak p_-$.  Then $\Ad:\mathsf P_-\to \operatorname{Aut}(\mathfrak p)$ is the restriction of the ambient adjoint action. We use $\cO_{\mathsf P_-}$ in both settings: it denotes the sheaf of smooth real functions in the smooth setting and the algebraic structure sheaf in the algebraic setting.  Similarly, $\Omega^1_{\mathsf P_-}$ denotes smooth one-forms or algebraic K\"ahler one-forms, respectively.

\begin{remark}
A common alternative is to use formal group schemes, where every finite-dimensional matched pair (or Manin triple) gives rise to a matched pair of formal groups $\widehat{\mathsf P}_+\bowtie\widehat{\mathsf P}_-$ (or a formal double group $(\widehat{\mathsf P},\widehat{\mathsf P}_+,\widehat{\mathsf P}_-)$). However, this setup introduces additional completions in the base algebra and in the universal vertex envelope, so we do not pursue it here.
\end{remark}

For $g\in\mathsf P_-$ and $\xi\in\mathfrak p_+$, put
\begin{equation*}
\rho_\xi(g):=-(L_g)_*p_-\Ad_{g^{-1}}\xi.
\end{equation*}
The smooth case of the following proposition is the standard infinitesimal matched-pair action constructed by Majid \cite[Theorem~3.1]{Majid1990}. The same infinitesimal computation applies in the algebraic setting.

\begin{proposition}
\label{prop:finite-factor-action}
The fields $\rho_\xi$ are global smooth vector fields in the smooth
setting and global regular vector fields in the algebraic setting.  They
satisfy
\begin{equation*}
[\rho_\xi,\rho_\eta]=\rho_{[\xi,\eta]}.
\end{equation*}
Consequently, $\rho$ defines an infinitesimal action
\begin{equation*}
\rho:\mathfrak p_+ \longrightarrow
\Gamma(\mathsf P_-,\mathcal T_{\mathsf P_-}).
\end{equation*}
\end{proposition}

\begin{definition}
\label{def:matched-pair-action-alg}
For an integrated matched pair as above, its \emph{matched-pair action Lie
algebroid} on $\mathsf P_-$ is the locally free sheaf
\begin{equation*}
T(\mathsf P_-, \mathfrak p_+)
:=
\cO_{\mathsf P_-}\otimes_{\kk}\mathfrak p_+.
\end{equation*}
Its anchor is $a\xi\mapsto a\rho_\xi$, and its bracket on local sections
is
\begin{equation}
[a\xi,b\eta] = ab[\xi,\eta] +a\rho_\xi(b)\eta -b\rho_\eta(a)\xi.
\label{eq:action-alg-bracket}
\end{equation}
\end{definition}
As in Ex.~\ref{ex:action-ext-alg}, the matched-pair action Lie algebroid has a canonical extended structure, denoted by $\cT(\mathsf P_-, \mathfrak p_+)$. We take $\Omega=\Omega_{\mathsf P_-}^1$ and let $\partial = \dd :\cO_{\mathsf P_-}\to\Omega_{\mathsf P_-}^1$ be the de Rham differential in the smooth setting and the K\"ahler differential in the algebraic setting.

We now specialize to the matched pair of the rational loop splitting. We use the group $G_-$ of $\dminus$ as considered in Sec.~\ref{sec:formal-neg-loop-group}. The projections $p_{\pm}$ will be denoted as $(-)_{\pm}$. For $\xi\in\dplus$ we define
\begin{align}
M_g(\xi) = \bigl(\Ad_{g^{-1}}\xi\bigr)_+, \quad \rho_\xi(g) = -(L_g)_*
\bigl(\Ad_{g^{-1}}\xi\bigr)_-.
\label{eq:rat-factor-action}
\end{align}

We do not invoke a general integration theorem for infinite-dimensional matched pairs, but can check directly the following:

\begin{lemma}
\label{lem:rat-factor-poly-closure}
For every $\xi\in\dplus$, the function $g\mapsto M_g(\xi)$ is polynomial with values in $\dplus$; equivalently, $M:\dplus\longrightarrow A\otimes\dplus$ defined by $M(\xi)(g):=M_g(\xi)$ is well defined.  

The map $\rho_\xi$ preserves $A=\cO(G_-)$ and hence defines a derivation of $A$.  Moreover,
\begin{equation*}
\rho:\dplus\longrightarrow\Der(A)
\end{equation*}
is a Lie algebra homomorphism.
\end{lemma}

\begin{proof}
Put $G_-^{(0)}:=\operatorname{Spec}\kk$ and $A_{\leq N}:=\cO(G_-^{(N)})$, so that $A=\varinjlim_NA_{\leq N}$.  Let $\xi(u)=\sum_{j=0}^{d}\xi_ju^j$ and set $g = \exp(B)$, where $B = \sum_{r \geq 1} B_r u^{-r}$. Then
\begin{equation*}
\Ad_{g^{-1}}\xi
=
\sum_{n\geq0}\frac{(-1)^n}{n!}
\sum_{j=0}^{d}
\sum_{r_1,\ldots,r_n\geq1}
\ad_{B_{r_1}}\cdots\ad_{B_{r_n}}(\xi_j)
u^{j-r_1-\cdots-r_n}.
\end{equation*}
A summand contributes to the coefficient of $u^k$, for $k\geq0$, only if $r_1+\cdots+r_n=j-k$. Thus $0\leq k\leq j\leq d$, and only finitely many compositions of $j-k$ occur.  It follows that $(\Ad_{g^{-1}}\xi)_+$ has degree at most
$d$ and that all of its coefficients are polynomials in
$B_1,\ldots,B_d$.  Hence $M(\xi)$ factors through $G_-^{(d)}$ and is a
regular $\dplus$-valued function, proving the first assertion.

Similarly, a summand contributes to the coefficient of $u^{-m}$ only if $r_1+\cdots+r_n=j+m$. For $m\leq N$, this coefficient is therefore a polynomial in $B_1,\ldots,B_{N+d}$.  Thus $\rho_\xi(A_{\leq N})\subset A_{\leq N+d}$. Taking the direct limit proves $\rho_\xi(A)\subset A$.  Finally, the
infinitesimal calculation underlying Prop.~\ref{prop:finite-factor-action} applies here, then passing to the direct limit gives $[\rho_\xi,\rho_\eta]=\rho_{[\xi,\eta]}$.
\end{proof}

\begin{definition} \label{def:rat-yangian-action-alg}  The \emph{Yangian action Lie algebroid} is the Lie algebroid
\begin{equation*}
T(G_-, \dplus) = \cO(G_-)\otimes\dplus,
\end{equation*}
with anchor $a\xi\mapsto a\rho_\xi$ and bracket \eqref{eq:action-alg-bracket}. We also define its associated extended Lie algebroid, denoted by $\cT(G_-, \dplus)$, with the module $\Omega_{G_-}^1$ of Kähler differentials.
\end{definition}

\subsection{Poisson cotangent algebroids and Manin triples}
\label{sec:poisson-cotangent-comp}
The preceding construction has an equivalent geometric description in terms of Poisson Lie structure.  Let $X$ be a smooth Poisson manifold over $\mathbb R$, or a smooth algebraic Poisson variety over $\kk$. If $\pi\in\Gamma(X,\wedge^2T_X)$ is the Poisson bivector, its cotangent Lie algebroid is
\begin{equation*}
T(X,\pi) := \bigl(\Omega_X^1,[\, ,\,]_\pi,\pi^\sharp \bigr), \qquad \pi^\sharp(\alpha):=\pi(\alpha,-),
\end{equation*}
where
\begin{equation*}
[\alpha,\beta]_\pi
=
\mathcal L_{\pi^\sharp\alpha}\beta
-\mathcal L_{\pi^\sharp\beta}\alpha
-\dd\bigl(\pi(\alpha,\beta)\bigr).
\end{equation*}
In particular,
$[\dd f,\dd g]_\pi=\dd\{f,g\}$. This construction also naturally produce an extended Lie algebroid that we denote by $\cT(X,\pi)$. When $X$ is moreover a Poisson Lie group, we can trivialize the cotangent bundle, and this construction can be shown to be equivalent to the construction of proceeding section.

Let $(\mathfrak p,\mathfrak p_+,\mathfrak p_-)$ be a Manin triple. Suppose its underlying matched pair satisfies the integration settings of Sec.~\ref{sec:bicrossed-action-algs}. $\kappa_{\mathfrak p}$ identifies $\mathfrak p_+\simeq\mathfrak p_-^*$.  For $\xi\in\mathfrak p_+$, let $\theta_\xi^R$ be the corresponding
right-invariant one-form on $\mathsf P_-$:
\begin{equation*}
(\theta_\xi^R)_g(v)
=
\kappa_{\mathfrak p}\left(
\xi,(R_{g^{-1}})_*v
\right).
\end{equation*}

Choose $\kappa_{\mathfrak p}$-dual bases $\{e_i\}\subset\mathfrak p_+$ and $\{e^i\}\subset\mathfrak p_-$. We can write the associated classical $r$-matrix as 
\begin{equation*}
r_{\mathfrak p}
=
\frac12\sum_i
\bigl(e_i\otimes e^i-e^i\otimes e_i\bigr)
\in \wedge^2\mathfrak p.
\end{equation*}
The standard multiplicative Poisson tensor on the double is
\begin{equation*}
\Pi_{\mathsf P}=r_{\mathfrak p}^L-r_{\mathfrak p}^R.
\end{equation*}
It induces the multiplicative Poisson tensor $\pi_{\mathfrak{p}}$ on $\mathsf P_-$
\cite{SemenovTianShansky,LuWeinstein}. 

The following identification is stated in \cite[Remark~4.9(2)]{EvensLuWeinstein1996}. We record it in our Manin-triple conventions, together with its extension to the canonical extended Lie algebroids.

\begin{theorem}
\label{thm:dressing-cotangent}
Under the preceding hypotheses, the map
\begin{equation*}
\jmath:T(\mathsf P_-, \mathfrak p_+)
\xrightarrow{\sim}T(\mathsf P_-, \pi_{\mathfrak p}),
\qquad
\jmath(a\xi)=-a\theta^R_\xi,
\end{equation*}
is an isomorphism of Lie algebroids. It also induces an isomorphism of their canonical (K\"{a}hler) extended Lie algebroids,
\begin{equation*}
\bigl( \mathrm{id}_{\mathcal O_{\mathsf P_-}}, \jmath, \mathrm{id}_{\Omega^1_{\mathsf P_-}} \bigr): \cT(\mathsf P_-, \mathfrak p_+) \xrightarrow{\sim} \cT(\mathsf P_-, \pi_{\mathfrak p}).
\end{equation*}
\end{theorem}

\begin{proof}
Right translation and the Manin pairing identify the underlying
$\mathcal O_{\mathsf P_-}$-modules. In our conventions, the anchor and
invariant-form identities are
\begin{equation*}
\pi_{\mathfrak p}^{\sharp}(\theta^R_\xi)(g)
=
(L_g)_*p_-\operatorname{Ad}_{g^{-1}}\xi
=
-\rho_\xi(g).
\end{equation*}
\begin{equation}
\label{eq:manin-bracket}
[\theta^R_\xi,\theta^R_\eta]_{\pi_{\mathfrak p}}
=
-\theta^R_{[\xi,\eta]}.
\end{equation}
Thus $\jmath(\xi)=-\theta^R_\xi$ preserves anchors and brackets, and the Leibniz rule extends this to all sections. The same identities apply to regular forms in the algebraic setting. The canonical
extended structures use the same differential and the contraction and Lie derivative along the anchor, giving the isomorphism of extended Lie algebroids.
\end{proof}

Assume now that $(\g,\kappa)$ is quadratic: $\kappa$ is a nondegenerate invariant symmetric bilinear form. The residue pairing
\begin{equation*}
\langle X,Y\rangle = \Res_u\kappa(X(u),Y(u))
\end{equation*}
makes $(\dloop,\dplus,\dminus)$ the rational topological Manin triple. Its standard Yang's $r$-matrix is
\begin{equation*}
r_{\mathrm{Yang}}(u,v) = \iota_{u,v}\frac{\Omega_\g}{u-v},
\end{equation*}
where $\Omega_\g\in\g\otimes\g$ is the inverse tensor of $\kappa$ and $\iota_{u,v}$ denotes expansion in $u^{-1}$ with polynomial coefficients in $v$.   Denote by $\pi_{\mathrm{rat}}$ the multiplicative Poisson bivector on $G_-$ determined by $r_\mathrm{Yang}$. This can also be identified the standard Poisson structure on the first congruence subgroup, see \cite{KamnitzerWebsterWeekesYacobi} and \cite{IlinRybnikov}. For $\xi\in\dplus$, the right-invariant form and our sign convention are
\begin{equation*}
\theta_\xi^R = \Res_u\kappa\bigl(\xi(u),(\dd g)g^{-1}\bigr),
\qquad
\pi_{\mathrm{rat}}^\sharp(\theta_\xi^R)=-\rho_\xi,
\end{equation*}
We have
\begin{corollary}
\label{thm:action-cotangent-comp}
The map
\begin{equation*}
\jmath: \mathcal{O}(G_-)\otimes\dplus \xrightarrow{\ \sim\ } \Omega_{G_-}^1,
\qquad
\jmath(a\xi)=-a\theta_\xi^R
\end{equation*}
is an isomorphism from the Yangian action Lie algebroid to the rational Poisson cotangent Lie algebroid, which also identifies their canonical (K\"{a}hler) extended Lie algebroids.
\end{corollary}
\begin{remark}
\label{rem:dep-quadratic-form}
One may notice that the rational Poisson cotangent Lie algebroid depends on the choice of the invariant form $\kappa$ on $\g$, while the action Lie algebroid depends only on the loop splitting. In fact, the isomorphism $\jmath$ itself depends on $\kappa$ via the residue pairing, so this dependence compensates on both sides, allowing the action Lie algebroid to be defined without selecting a quadratic form.
\end{remark}

\subsection{The matrix realization for \texorpdfstring{$G=GL_N$}{G = GL(N)}}
\label{sec:matrix-rat-realization}

When specializing the Lie algebroid construction to $\g=\mathfrak{gl}_N$ equipped with the standard trace form $\kappa(x,y)=\Tr(xy)$, all abstract structures admit concrete matrix presentations that connect directly with the classical Yangian formalism.

The formal negative loop group $G_-$ consists of monic matrix series in $u^{-1}$:
\begin{equation*}
G_- = \left\{ T(u)=I+\sum_{r\geq1}T^{(r)}u^{-r} \ \middle|\ T^{(r)}\in\operatorname{Mat}_N(\kk) \right\}.
\end{equation*}
Writing $T(u)=\bigl(t_{ij}(u)\bigr)_{i,j=1}^N$ with $t_{ij}(u) = \delta_{ij}+\sum_{r\geq1}t_{ij}^{(r)}u^{-r}$, and setting $t_{ij}^{(0)}=\delta_{ij}$, the coordinate ring of $G_-$ is the free polynomial algebra
\begin{equation*}
A = \mathcal O(G_-) = \kk\bigl[t_{ij}^{(r)} \bigm| 1\leq i,j\leq N,\ r\geq1\bigr].
\end{equation*}
For a polynomial current $\xi(u)\in\mathfrak{gl}_N[u]$, the dressing map $M_T(\xi)$ and the infinitesimal dressing action $\rho_\xi(T)$ from \eqref{eq:rat-factor-action} take the matrix form
\begin{equation*}
	M_T(\xi) = \bigl(T^{-1}\xi T\bigr)_+,\quad
\rho_\xi(T) = T\bigl(T^{-1}\xi T\bigr)_+-\xi T.
\end{equation*}

Geometrically, the cotangent bundle of $G_-$ carries two natrual trivializations: as a pro-unipotent Lie group, its cotangent bundle $T^*G_-$ is trivialized by right-invariant one-forms; as a pro-affine space with coordinate functions $\{t_{ij}^{(r)}\}_{i,j,r}$, $T^*G_-$ is also trivialized by coordinate differentials $\dd t_{ij}^{(r)}$. The interplay between these two frames connects the dressing action to the Yangian Poisson bracket.

The right Maurer--Cartan form on $G_-$ and the corresponding right-invariant one-forms are
\begin{equation*}
\Theta^R=(\dd T)T^{-1},
\qquad
\theta_\xi^R
=
\Res_u\Tr\bigl(\xi(u)\Theta^R(u)\bigr).
\end{equation*}

We also have the flat coordinate one-forms $\mathsf{K}_\xi$ defined by
\begin{equation*}
(\mathsf{K}_\xi)_T(\delta T) = \Res_u\Tr\bigl(\xi(u)\delta T(u)\bigr).
\end{equation*}
Choosing the standard basis $\xi(u)=E_{ji}u^{r-1}$ selects precisely the one-form $\mathsf{K}_{E_{ji}u^{r-1}} = \dd t_{ij}^{(r)}$. The coordinate forms are related to the right-invariant forms as follows. Writing $T\xi = (T\xi)_+ + (T\xi)_-$, we have $\Tr\bigl((T\xi)_+\Theta^R\bigr) =  \Tr(\xi\dd T) - \Tr\bigl(T^{-1}(T\xi)_-\dd T\bigr) $. Since $(T\xi)_-$ and $\dd T$ each have order at least $u^{-1}$, the second term have its residue vanishes. Hence
\begin{equation*}
\mathsf{K}_\xi=\theta_{(T\xi)_+}^R,
\end{equation*}
where the map $\xi\mapsto(T\xi)_+$ has the inverse $\eta\mapsto(T^{-1}\eta)_+$.

Evaluating the anchor map on these coordinate one-forms yields the standard matrix expression for the Yangian Poisson structure. Since $\mathsf K_\xi=\dd f_\xi$ for the linear coordinate function $f_\xi(T)=\Res_u\Tr(\xi(u)T(u))$, the anchor map sends $\mathsf K_\xi$ to the Hamiltonian vector field of $f_\xi$. Via the relation $\mathsf K_\xi=\theta_{(T\xi)_+}^R$, this evaluates the Poisson bracket of $f_\xi$ with $T$ as
\begin{equation*}
\{f_\xi, T\}_{\pi_{\mathrm{rat}}}
= \pi_{\mathrm{rat},T}^{\sharp}(\mathsf K_\xi)
= -\rho_{(T\xi)_+}(T)
= (T\xi)_+T - T(\xi T)_+.
\end{equation*}
Specializing to the basis currents $\xi(u)=E_{ji}u^{r-1}$ (for which $f_\xi=t_{ij}^{(r)}$), and summing over $r\geq1$ directly yields
\begin{equation}
\label{eq:rat-RTT}
\{t_{ij}(u),t_{kl}(v)\}_{\pi_{\mathrm{rat}}}
= \frac{1}{u-v}(t_{kj}(u)t_{il}(v)-t_{kj}(v)t_{il}(u)).
\end{equation}
This is the Sklyanin Poisson--Lie bracket associated with the rational loop splitting \cite{SemenovTianShansky}, written in classical RTT form \cite{Shapiro}.

\section{Vertex extensions of cotangent Yangian}
\label{sec:yangian-vext}

In this section, we analyze vertex algebroid extensions of the cotangent Yangian Lie algebroid introduced in the previous section. Via the isomorphism of Cor.~\ref{thm:action-cotangent-comp}, we may equivalently work with the action Lie algebroid $\cT(G_-, \dplus)$, which is more convenient for computations. For most of the analysis, we first study extensions of a general action Lie algebroid $\cT(A, \mathfrak a)$ and then specialize to the Yangian case, which presents its own unique features.

\subsection{Action Lie algebroid extensions}
\label{sec:action-ext}
 
Let $A$ be a commutative $\kk$-algebra and let
$\rho:\mathfrak a\to\operatorname{Der}_{\kk}(A)$ be a Lie algebra
action.  Recall from Ex.~\ref{ex:action-ext-alg}
that we have the extended action Lie algebroid $\cT(A, \mathfrak a) = \bigl(A,T=A\otimes_{\kk}\mathfrak a,\Omega=\Omega^1_{A},\partial=\dd,\langle-,-\rangle\bigr)$. A vertex algebroid extension requires additional data $(\gamma,\chi,c)$ satisfying the identities (A1)--(A5) of Sec.~\ref{sec:valg-def}.  We can first eliminate the freedom in $\gamma$ by the change-of-splitting isomorphism of Prop.~\ref{prop:change-splitting}.

\begin{proposition}
\label{sec:action-levels-3form}
Every vertex algebroid extending $\cT(A, \mathfrak a)$ is isomorphic, by a change of splitting $\Theta: T \to \Omega$, to one satisfying
\begin{equation*}
	\gamma(a,\xi)=0, \qquad a\in A,\quad \xi\in\mathfrak a.
\end{equation*}
Moreover, the change of splitting can be chosen to vanish on the $\kk$-subspace $T_0 = 1\otimes\mathfrak a$. Consequently, it does not change the values of the pairing $\chi$ or of $c$ on $T_0$.
\end{proposition}

\begin{proof}
For the given vertex algebroid $\cA$ extending $\cT(A, \mathfrak a)$, define the $\kk$-linear map $\Theta: A\otimes_{\kk}\mathfrak a\to\Omega^1_A$ by $\Theta(a\xi)=-\gamma(a,\xi)$ for $a\in A, \xi\in\mathfrak a$. Applying Prop.~\ref{prop:change-splitting} gives a vertex algebroid isomorphism $\cA \overset{\cong}{\longrightarrow} \cA^\Theta$. For $\cA^\Theta$ we have
\begin{equation*}
\gamma^\Theta(a,\xi) = \gamma(a,\xi)+\Theta(a\xi)-a\Theta(\xi)=0.
\end{equation*}

Identity~\eqref{eq:GMS-A1} gives $\gamma(1,\xi)=0$, so $\Theta(\xi)=0$, which implies that the $\chi$ and $c$-terms are unchanged for $\mathcal{A}^{\Theta}$.
\end{proof}

Although $\gamma(a,\xi)=0$ for $\xi\in\mathfrak a$, the full map $\gamma$ need not vanish: identity~\eqref{eq:GMS-A1} forces
\begin{equation*}
\gamma(a,b\xi)=-\rho_\xi(a)\,\dd b-\rho_\xi(b)\,\dd a,
\qquad a,b\in A,\quad\xi\in\mathfrak a.
\end{equation*}

With $\gamma$ satisfying the above condition, the vertex algebroid extension is determined by the following maps:
\begin{equation*}
\chi_0:\Sym^2\mathfrak a\longrightarrow A,
\qquad
c_0:\Lambda^2\mathfrak a\longrightarrow\Omega,
\end{equation*}
As a corollary of Thm.~\ref{lem:frame-reconstruct}, we have 
\begin{proposition}
\label{prop:action-ext-crit}
For $\gamma$ satisfying the above condition, the maps $\chi_0,c_0$ admit a vertex algebroid extension from $\cT(A, \mathfrak a)$ if and only if, for every $\xi,\eta,\zeta\in\mathfrak a$,
\begin{equation}
	\chi([\xi,\eta],\zeta)+\chi(\eta,[\xi,\zeta]) = 
\rho_\xi\chi(\eta,\zeta)
-\frac12\rho_\eta\chi(\xi,\zeta)
-\frac12\rho_\zeta\chi(\xi,\eta) +\iota_{\rho_\eta}c(\xi,\zeta)
+\iota_{\rho_\zeta}c(\xi,\eta),
\label{eq:norm-action-A4}
\end{equation}
and
\begin{align}
(d_{\mathrm{Lie}}c)(\xi,\eta,\zeta)
={}&-\frac12\dd\Bigl(
\chi([\xi,\eta],\zeta)
+\chi([\xi,\zeta],\eta)
-\chi([\eta,\zeta],\xi)
\notag\\
&\hspace{20mm}
-\rho_\xi\chi(\eta,\zeta)
+\rho_\eta\chi(\xi,\zeta)
-2\iota_{\rho_\zeta}c(\xi,\eta)
\Bigr).
\label{eq:norm-action-A5}
\end{align}
\end{proposition}

Now suppose we have a fixed choice of $\chi_0$, the extra choice of $c$ has a clean classification. Following \cite{GMS}, a \emph{vertex prealgebroid} is the datum
\begin{equation*}
\mathcal B=(\mathcal T,\gamma,\chi_0)
\end{equation*}
satisfying the vertex-algebroid axioms that do not involve $c$. We have the following classification from \cite{GMS}.

\begin{theorem}
\label{thm:GMS-prealg-class}
Let $\mathcal B$ be a vertex prealgebroid with underlying extended Lie algebroid $\mathcal T$. Assume $\mathcal B$ admits at least one vertex algebroid extension $\mathcal A_c = (\mathcal B, c)$.
\begin{enumerate}[label=\textup{(\roman*)}]
\item Any other map $c'$ completing $\mathcal B$ to a vertex algebroid has the form $c' = c + \alpha$ for a unique closed $3$-form $\alpha \in Z^3_{\mathrm{DR}}(\mathcal T)$, and conversely every such $\alpha$ defines a vertex algebroid $\mathcal A_{c+\alpha}$.
\item For $c' = c + \alpha$, the morphisms $\Phi: \mathcal A_c \to \mathcal A_{c'}$ acting as the identity on $A, T$, and $\Omega$ are precisely those of the form $\Phi = (\id_A, \id_T, \id_\Omega, h)$ with $h \in \Omega^2_{\mathrm{DR}}(\mathcal T)$ satisfying
\[
d_{\mathrm{DR}}h = \alpha.
\]
In particular, automorphisms of $\mathcal A_c$ correspond to closed $2$-forms via addition ($\Phi_{h_2}\circ\Phi_{h_1} = \Phi_{h_1+h_2}$), and the isomorphism classes of extensions of $\mathcal B$ form a torsor under $H^3_{\mathrm{DR}}(\mathcal T)$.
\end{enumerate}
\end{theorem}

A complete classification of vertex algebroid extensions of an action Lie algebroid is beyond the scope of this paper. In the following we consider two convenient classes of vertex algebroid extension of an action Lie algebroid: extension induced by an invariant bilinear form (level) $\kappa$ and a closed three-form class.

First we recall that the level $\kappa$ extension is already defined in example~\ref{ex:affine-level-action-ext}, by
\begin{equation*}
\gamma(a,\xi)= 0,\quad \chi(\xi,\eta)=\kappa(\xi,\eta), \quad c(\xi,\eta)=0, \quad \text{for } a \in A, \xi,\eta\in\mathfrak a.
\end{equation*}
We denote this vertex algebroid $\cA_{\kappa}(A,\mathfrak a)$.

For the full skew map $c_{\kappa}$ of $\cA_{\kappa}$ and $\mathcal H\in Z^3_{\mathrm{DR}}(\mathcal T_\rho)$, define $\cA_{\kappa,\mathcal H}(A,\mathfrak a)$ to have the map $c =  c_{\kappa}+\mathcal H$ and with the maps $\gamma$ and $\chi$ unchanged. By Thm.~\ref{thm:GMS-prealg-class}, $\cA_{\kappa,\mathcal H}(A,\mathfrak a)$ is a vertex algebroid for any invariant level $\kappa$ and any closed three-form $\mathcal H$. We also denote $V_{\kappa,\mathcal{H}}(A, \mathfrak a) = \Uenv(\cA_{\kappa,\mathcal H}(A,\mathfrak a))$ the envelope vertex algebra.

By the PBW theorem $V_{\kappa,\mathcal{H}}(A, \mathfrak a)$ is generated by functions $a \in A$ and currents $\Lcur_\xi$ for $\xi \in \mathfrak a$. Their OPEs are:
\begin{equation}\label{eq:action-level-OPE}
  \begin{aligned}
    a(z)b(w)&\sim0,\\
\Lcur_\xi(z)a(w)&\sim \frac{\rho_\xi(a)(w)}{z-w},\\
\Lcur_\xi(z)\Lcur_\eta(w)&\sim \frac{\kappa(\xi,\eta)}{(z-w)^2} +\frac{ \Lcur_{[\xi,\eta]}(w)-\mathcal H(\xi,\eta)(w)}{z-w}.
\end{aligned}
\end{equation}

Oftentimes we have little knowledge of the (extended) Lie algebroid cohomology $H^\bullet_{\mathrm{DR}}(\mathcal T)$. Therefore we also consider the subclass obtained from algebraic closed three-forms $H\in\Omega^3_{A}$. Via the cochain map $K^\bullet: \Omega^\bullet_A \to \Omega^\bullet_{\mathrm{DR}}(\mathcal T)$, we define
\begin{equation*}
\beta_H(\tau,\nu)
:=
\iota_{\rho(\nu)}\iota_{\rho(\tau)}H
=K^3(H)(\tau,\nu).
\end{equation*}
Thus we obtain a family of vertex algebroids, denoted by $\cA_{\kappa,H}(A, \mathfrak a)$, parameterized by $H\in\Omega^3_{A}$. We emphasize that the induced map $K^{\bullet}: H^\bullet_{\mathrm{dR}}(A) \to H^\bullet_{\mathrm{DR}}(\mathcal{T})$ need not be injective or surjective in general.

For the Yangian action Lie algebroid $\cT(G_-, \dplus)$ of Sec.~\ref{sec:coYangian-alg}, the preceding constructions apply, but algebraic $3$-forms on $G_-$ give no new isomorphism classes of extensions. Indeed, using the Euler field $E_u$ of Lem.~\ref{lem:spectral-affine-vf}, one easily verifies that $H^j_{\mathrm{dR}}(A)=0$ for every $j>0$. Specifically, a closed homogeneous form $\omega$ of positive degree has Euler degree $N>0$, so Cartan's formula gives $\omega=\frac1N\dd(\iota_{E_u}\omega)$. 

Although one can still define extensions of $\cT(G_-, \dplus)$ using closed three-forms in $Z^3_{\mathrm{DR}}(\cT)$, unfortunately we do not know much about the extended Lie algebroid cohomology $H^3_{\mathrm{DR}}(\mathcal{T})$. Therefore, in this work, we focus on the level twisted family and write
\begin{equation*}
\widehat\kappa\in
\operatorname{Hom}_{\kk}
\bigl(\Sym^2_{\kk}\dplus,\kk\bigr)^{\dplus}
\end{equation*}
The hat distinguishes this form on $\dplus$ from a quadratic form $\kappa$ on $\g$. As from the Remark \ref{rem:dep-quadratic-form}, we do not need to require $\widehat\kappa$ to be induced from $\kappa$ that defines the Manin triple $(\mathfrak{l},\dplus,\dminus)$. 

\begin{definition}[Chiral cotangent Yangian]
\label{def:chiral-coYangian}
The level-$\widehat\kappa$ \emph{chiral cotangent Yangian} is
\begin{equation*}
\Ych_{\whkappa}(T^*\g)
:=
\Uenv\bigl(\cA_{\whkappa}(A,\dplus)\bigr),
\end{equation*}
We also write $\Ych(T^*\g):=\Ych_0(T^*\g)$ for the level $\widehat\kappa=0$ case.
\end{definition}

\subsection{Extensions from spectral translation}
\label{sec:cocycle-spectral-ext}
The discussion in the preceding subsection provides convenient classes of extensions of the cotangent Yangian Lie algebroid, but these do not exhaust all possibilities. In this subsection, we construct a further family of vertex algebroids, independent of the preceding level and three-form constructions.

Let $\mathsf V$ be a finite-dimensional trivial $\mathfrak a$-module with a symmetric form $K:\Sym^2_{\kk}\mathsf V\to\kk$. A Chevalley--Eilenberg cocycle $\lambda\in C^1(\mathfrak a,A\otimes\mathsf V)$ for the action
$\rho\otimes\id_{\mathsf V}$ is an element $\lambda \in \operatorname{Hom}(\mathfrak{a},A\otimes \mathsf V)$ satisfying
\begin{equation*}
\rho_\xi\lambda_\eta-\rho_\eta\lambda_\xi= \lambda_{[\xi,\eta]},
\end{equation*}
where we write $\lambda_\xi:=\lambda(\xi)$.

\begin{theorem}
\label{thm:cocycle-action-ext}
Let $\cA = (\cT(A, \mathfrak a), \gamma, \chi, c)$ be a vertex algebroid extension of the extended action Lie algebroid $\cT(A, \mathfrak a)$ satisfying $\gamma(a,\xi)=0$ for all $a\in A, \xi\in\mathfrak a$.
For every cocycle $\lambda$ as above, the data on $\mathfrak a$
\begin{equation}
  \label{eq:cocycle-added}
  \begin{aligned}
    \chi_\lambda(\xi,\eta) &=\chi(\xi,\eta)+K(\lambda_\xi,\lambda_\eta), \\
c_\lambda(\xi,\eta) &=c(\xi,\eta) +\frac12\left( K(\lambda_\xi,\dd\lambda_\eta) -K(\lambda_\eta,\dd\lambda_\xi)\right)
  \end{aligned}
\end{equation}
determine a unique vertex algebroid extension $\cA_\lambda = (\cT(A, \mathfrak a), \gamma, \chi_\lambda, c_\lambda)$. Applied to the level extension $\cA_\kappa(A,\mathfrak a)$, which has $\chi=\kappa$ and $c=0$ on $\mathfrak{a}$, we denote the resulting vertex algebroid by $\cA_{\kappa,\lambda} = \mathcal{A}_{\kappa,\lambda}(A,\mathfrak{a})$.
\end{theorem}

\begin{proof}
First consider the level extension $\cA = \cA_\kappa(A,\mathfrak a)$. We adjoin $\mathsf V$ to $\mathfrak a$ as a central Lie algebra summand acting trivially on $A$, and equip $\widetilde{\mathfrak a} := \mathfrak a \oplus \mathsf V$ with the invariant symmetric bilinear form $\widetilde\kappa := \kappa \oplus K$. In the associated level extension $\widetilde\cA := \cA_{\widetilde\kappa}(A,\widetilde{\mathfrak a})$ over $\cT(A, \widetilde{\mathfrak a})$, consider the $A$-linear graph map
\begin{equation*}
\Gamma_\lambda : A\otimes\mathfrak a \longrightarrow A\otimes\widetilde{\mathfrak a},
\qquad \Gamma_\lambda(\xi) = \xi + \lambda_\xi.
\end{equation*}
Because $\mathsf V$ is central and acts trivially on $A$, the cocycle condition implies that
\begin{equation*}
[\Gamma_\lambda(\xi), \Gamma_\lambda(\eta)]
= [\xi, \eta] + \rho_\xi\lambda_\eta - \rho_\eta\lambda_\xi
= [\xi, \eta] + \lambda_{[\xi,\eta]}
= \Gamma_\lambda([\xi, \eta]),
\end{equation*}
while the anchor $\rho(\Gamma_\lambda(\xi)) = \rho_\xi$ is unchanged. The image $\Gamma_\lambda(A\otimes\mathfrak a)$ and $\Omega_A^1$ therefore form an extended Lie subalgebroid of $\cT(A, \widetilde{\mathfrak a})$ canonically isomorphic to $\cT(A, \mathfrak a)$ via $\Gamma_\lambda$.

Restricting the vertex algebroid operations of $\widetilde\cA$ to this subalgebroid (or applying the reconstruction formulas of Thm.~\ref{lem:frame-reconstruct} on the generators $\Gamma_\lambda(\xi)$) yields $\gamma(a, \xi) = 0$ on $A\times\mathfrak a$, together with the data~\eqref{eq:cocycle-added} with $(\chi, c) = (\kappa, 0)$. In particular, $\cA_{\kappa,\lambda}$ is a well-defined vertex algebroid extension of $\cT(A, \mathfrak a)$, and
\begin{equation*}
(\id_A,\Gamma_\lambda,\id_{\Omega_A^1},0):
\cA_{\kappa,\lambda}\longrightarrow\widetilde\cA
\end{equation*}
is an embedding of vertex algebroids.

Now consider a general vertex algebroid extension $\cA = (\cT(A, \mathfrak a), \gamma, \chi, c)$ satisfying $\gamma(a,\xi)=0$ on $A\times\mathfrak a$. The extension identities~\eqref{eq:norm-action-A4}--\eqref{eq:norm-action-A5} of Prop.~\ref{prop:action-ext-crit} are linear in the data $(\chi, c)$. Taking $\kappa=0$ in the level construction shows that the cocycle perturbation terms
\begin{equation*}
\Delta\chi(\xi,\eta) = K(\lambda_\xi,\lambda_\eta),
\qquad
\Delta c(\xi,\eta) = \frac12\left(K(\lambda_\xi,\dd\lambda_\eta) - K(\lambda_\eta,\dd\lambda_\xi)\right)
\end{equation*}
themselves satisfy these extension identities. Therefore, adding them to the existing data $(\chi, c)$ of $\cA$ yields another solution $(\chi_\lambda, c_\lambda) = (\chi + \Delta\chi, c + \Delta c)$. Uniqueness of the resulting vertex algebroid extension $\cA_\lambda$ follows from Prop.~\ref{prop:action-ext-crit}.
\end{proof}

For a fixed extension $\cA$, cohomologous cocycles give isomorphic
extensions.
\begin{proposition}
Let $F\in A\otimes\mathsf V$ and
$\lambda'_\xi=\lambda_\xi+\rho_\xi F$.  Define the map $\Theta_F:T\to\Omega_A^1$ by
\begin{equation*}
\Theta_F(\xi)
=K\left(\lambda_\xi+\frac12\rho_\xi F, \partial F\right).
\end{equation*}
Then the vertex algebroids $\cA_{\lambda}$ and
$\cA_{\lambda'}$ are isomorphic via the change of splitting
\begin{equation*}
(\id_A,\id_T,\id_{\Omega_A^1},-\Theta_F):
\cA_\lambda\xrightarrow{\sim}\cA_{\lambda'} 
\end{equation*}
\end{proposition}

\begin{proof}
Since $\Theta_F$ is $A$-linear, the change-of-splitting \eqref{eq:GMS-change} gives $\gamma^{\Theta_F}(a,\xi) = \gamma(a,\xi) = 0$ for $a\in A, \xi\in\mathfrak a$. 

Next, we evaluate the change in the pairing $\chi$. Using the definition of $\Theta_F$ and the duality pairing $\langle\xi, \omega\rangle = \iota_{\rho_\xi}\omega$, a straightforward contraction gives
\begin{equation*}
\iota_{\rho_\xi}\Theta_F(\eta) + \iota_{\rho_\eta}\Theta_F(\xi)
= K(\lambda_\xi + \rho_\xi F, \lambda_\eta + \rho_\eta F) - K(\lambda_\xi, \lambda_\eta).
\end{equation*}
By \eqref{eq:GMS-change}, this is precisely the difference $\chi_{\lambda'}(\xi,\eta) - \chi_\lambda(\xi,\eta)$ on $\mathfrak a$.

Similarly, evaluating the change-of-splitting \eqref{eq:GMS-change} for $c^{\Theta_F}$ and using the cocycle identity along with $\mathcal L_{\rho_\xi}\dd F = \dd(\rho_\xi F)$, a direct calculation shows that 
\begin{equation*}
c^{\Theta_F}(\xi,\eta) - c_\lambda(\xi,\eta)
= \frac12\left(K(\lambda'_\xi, \dd\lambda'_\eta) - K(\lambda'_\eta, \dd\lambda'_\xi)\right)
- \frac12\left(K(\lambda_\xi, \dd\lambda_\eta) - K(\lambda_\eta, \dd\lambda_\xi)\right).
\end{equation*}
Hence $c^{\Theta_F}(\xi,\eta) = c_{\lambda'}(\xi,\eta)$ on $\mathfrak a$. By the uniqueness statement of Prop.~\ref{prop:action-ext-crit}, the vertex algebroid data of $\cA_\lambda^{\Theta_F}$ and $\cA_{\lambda'}$ coincide, giving the identification
\begin{equation*}
\cA_{\lambda'}=\cA_\lambda^{\Theta_F}.
\end{equation*}
Prop.~\ref{prop:change-splitting} therefore yields the vertex algebroid isomorphism.
\end{proof}

Thus, for a fixed vertex algebroid extension $\cA$ of $\cT(A, \mathfrak a)$, the construction descends to a map from $H^1(\mathfrak a,A)\otimes\mathsf V$ to isomorphism classes of vertex algebroids extending $\cT(A, \mathfrak a)$. This map is not injective in general, as the quadratic dependence in~\eqref{eq:cocycle-added} gives $\cA_{-\lambda} = \cA_\lambda$.

We now specialize this construction to the Yangian case. We set $\mathsf V= \kk$ from now on. For $\xi \in\dplus$, define
\begin{equation}
\lambda_\xi(g)
:=
\Res_u\kappa\left(
\xi(u),(\Tu g)g^{-1}
\right),
\qquad g\in G_-.
\label{eq:spectral-log-cocycle}
\end{equation}
Here $\Tu$ differentiates the Yangian spectral coordinate. 

\begin{theorem}
\label{thm:spectral-trans-cocycle}
For every $\xi\in\dplus$, $\lambda_\xi \in A=\cO(G_-)$.  The assignment $\xi\mapsto\lambda_\xi$ is a
Chevalley one-cocycle, and defines a nonzero class in
$H^1(\dplus,A)$.
\end{theorem}

\begin{proof}
We consider the translation action $\sigma_s$ of Lem.~\ref{lem:spectral-affine-vf}, and the associated vector fields $D_u$.  Contracting $D_u$ with the right-invariant form $\theta^R_{\xi}$, for $\xi \in \dplus$, gives
\begin{equation*}
\theta_\xi^R(D_u)
=\Res_u\kappa\left(\xi(u),(\Tu g)g^{-1}\right)
=\lambda_\xi.
\end{equation*}
This proves $\lambda_\xi\in A$.

Yang's $r$-matrix is invariant under the translation action:
$\frac{\Omega_\g}{(u+s)-(v+s)}=\frac{\Omega_\g}{u-v}$. Thus
$\sigma_s$ preserves the Poisson tensor, and the vector field $D_u$ satisfies
\begin{equation*}
\mathcal L_{D_u}\pi_{\mathrm{rat}}=0.
\end{equation*}

Under the identification
$\Gamma(TG_-)=\operatorname{Hom}_A(\Omega_A^1,A)$, the vector field
$D_u$ gives the Lie-algebroid one-cochain $-\iota_{D_u}$. Its Lie algebroid cochain differential $d_{\mathrm{Lie}}(-\iota_{D_u}) = \mathcal L_{D_u}\pi_{\mathrm{rat}} = 0$. So this cochain is closed in the Lie-algebroid complex. Under the isomorphism of Cor.~\ref{thm:action-cotangent-comp}, $\jmath(\xi)=-\theta_\xi^R$, and hence $(\jmath^*(-\iota_{D_u}))(\xi)=\lambda_\xi$. We conclude $d_{\mathrm{Lie}}\lambda = 0$ and $\lambda$ defines a class in $H^1(\dplus,A)$. 

Choose $\mathsf h\in\g$ with $\kappa(\mathsf h,\mathsf h)\ne0$, and set $g_t=\exp(t\mathsf h u^{-1})$ and $\xi_{\mathsf h}=\mathsf h u$. Then $\Ad_{g_t^{-1}}\xi_{\mathsf h}=\xi_{\mathsf h}$, so $\rho_{\xi_{\mathsf h}}(g_t)=0$, whereas $\lambda_{\xi_{\mathsf h}}(g_t)=-t\kappa(\mathsf h,\mathsf h)$. If $\lambda_\xi =\rho_\xi f$ for some $f\in A$, evaluation along this $\kk[t]$-point would give $\lambda_{\xi_{\mathsf h}}(g_t)=0$. This is a contradiction, so $[\lambda]\ne0$.
\end{proof}

Now we set $K(a,b):= ab$. Applying Thm.~\ref{thm:cocycle-action-ext} to the level zero vertex algebroid $\cA(A, \dplus)$ gives an extension $\cA_{\lambda}(A, \dplus)$ with the following data for $\xi,\eta\in\dplus$:
\begin{equation}\label{eq:spectral-ext}
    \chi_\lambda(\xi,\eta) = \lambda_\xi\lambda_\eta, \quad c_\lambda(\xi,\eta) =\frac12\left( \lambda_\xi\dd\lambda_\eta-\lambda_\eta\dd\lambda_\xi \right).
\end{equation}
\begin{theorem}
\label{thm:spectral-ext-param}
Let $\mathfrak g$ be a reductive Lie algebra with a non-trivial center $\mathfrak{z}$. Then $\cA_{\lambda}(A, \dplus)$ is not isomorphic to $\cA_{\whkappa,\mathcal{H}}(A, \dplus)$ for any level $\widehat{\kappa}$ and closed three form $\mathcal{H}$.
\end{theorem}

\begin{proof}
As in the previous proof, we set $g_t = \exp(t\mathsf h u^{-1})$ and $\xi_{\mathsf h} = \mathsf h u$, but now we choose $\mathsf h \in \mathfrak z$ with $\kappa(\mathsf h, \mathsf h) \ne 0$. In this case, $\rho_{\xi}(g_t) = 0$ for any $\xi \in \dplus$. By axiom (A2) we have, for $a\in A$ and $\tau,\nu \in T:=A\otimes \dplus$,
\begin{equation*}
  \chi(a\tau,\nu)(g_t) = a(g_t)\chi(\tau,\nu)(g_t).
\end{equation*}
The same argument applies at any point $p$ where all anchor fields vanish. By symmetry, evaluation then defines a bilinear form $\chi_p$ on $T_p:=T/\mathfrak m_pT\simeq\dplus$.

Suppose that an isomorphism $(\phi_A,\phi_T,\phi_\Omega,h): \cA_{\lambda}(A, \dplus) \overset{\sim}{\to} \cA_{\whkappa,\mathcal{H}}(A, \dplus)$ exists. Write $F: G_- \to G_-$ for the base map such that $\phi_A = F^*$, and set $p_t:=F^{-1}(g_t)$. Anchor compatibility gives $\rho_\xi(p_t)=0$ for every $\xi\in\dplus$. Thus the contraction terms in~\eqref{eq:GMS-mor-pair} vanish at $p_t$, giving
\begin{equation*}
  \chi_\lambda(\tau,\nu)(g_t)
  = \chi_{\whkappa,\mathcal H}(\phi_T\tau,\phi_T\nu)(p_t).
\end{equation*}
Consequently, $\phi_T$ induces an isometry $(T_{g_t},\chi_\lambda|_{g_t})\xrightarrow{\sim}(T_{p_t},\chi_{\whkappa,\mathcal H}|_{p_t})$. Under $T_{p_t}\simeq\dplus$, $\chi_{\whkappa,\mathcal H}|_{p_t}$ is simply $\widehat\kappa$.

At $t=0$, we have $g_0=e$ and $\lambda_\xi(e)=0$ for every $\xi\in\dplus$, so $\chi_\lambda|_{e}$ is zero. The isometry therefore forces $\widehat\kappa=0$. At $t=1$, however,~\eqref{eq:spectral-log-cocycle} and~\eqref{eq:spectral-ext} give
\begin{equation*}
  \chi_\lambda(\xi_{\mathsf h},\xi_{\mathsf h})(g_1)
  =\lambda_{\xi_{\mathsf h}}(g_1)^2
  =\kappa(\mathsf h,\mathsf h)^2\ne0,
\end{equation*}
which is a contradiction. 
\end{proof}

We conjecture that the same non-isomorphism statement holds for nonzero semisimple $\mathfrak g$.

\subsection{Quasi-conformal structures}
\label{sec:qc}
Let $\cO_z$ denote the topological algebra $\kk[[z]]$, and let $\operatorname{Aut}(\cO_z)$ be the group of continuous automorphisms of $\cO_z$. Its Lie algebra is denoted by $\Der_0\cO_z := z\kk[[z]]\partial_z$. We also set 
\begin{equation*}
\Der\cO_z:=\kk[[z]]\partial_z, \qquad \Der_+\cO_z:=z^2\kk[[z]]\partial_z. 
\end{equation*}
We recall the definition of a quasi-conformal vertex algebra.
\begin{definition}
A $\mathbb{Z}$ (or $\mathbb{Z}_{\geq 0}$)-graded vertex algebra $V=\bigoplus_{n\in\mathbb Z}V_n$ is \emph{quasi-conformal} if it carries a continuous Lie algebra action of $\Der\cO_z$, denoted by $\{\mathscr L_n\}_{n\geq -1}$, such that
\begin{equation*}
\mathscr L_{-1}=\partial,
\qquad
\mathscr L_0|_{V_n}=n\,\id,
\end{equation*}
the operators in $\Der_+\cO_z$ act locally nilpotently, and
\begin{equation}
[\mathscr L_m,Y(v,z)]
=
\sum_{j=0}^{m+1}
\binom{m+1}{j}z^{m+1-j}
Y(\mathscr L_{j-1}v,z)
\label{eq:qc-covariance}
\end{equation}
for all $v\in V$ and $m\geq-1$. Here continuity means that sufficiently high modes annihilate each fixed vector, so formal vector fields act by finite sums on that vector. The restriction to $\Der_0\cO_z$ integrates to $\operatorname{Aut}(\cO_z)$, with the integral $\mathscr L_0$-grading giving the scaling action.
\end{definition}

By contrast, a vertex algebra is called \emph{conformal} if there exists $\omega \in V_2$ whose field
\begin{equation*}
Y(\omega,z)=\sum_{n\in\mathbb Z}L_n^\omega z^{-n-2}
\end{equation*}
satisfies the Virasoro OPE, with $L_{-1}^\omega=\omega_{(0)}=\partial$ and $L_0^\omega$ equal to the grading operator. The modes $L_n^\omega$ therefore endow $V$ with an internal quasi-conformal action, and $\omega$ is called the \emph{conformal vector}.

As we will see, the vertex algebras considered in this paper admit quasi-conformal structures (in fact, more than one in the case of $\Ych(T^*\g)$), but do not possess any conformal vector. 

The following criterion follows from the semi-conformal construction of \cite[Thm.~3.6]{MingLiSemiConformal} and the degree argument below.
\begin{theorem}
\label{thm:Li-qc-crit}
Let $\cA$ be a vertex algebroid, and write $V=\Uenv(\cA)$ with its nonnegative grading $V_0=A$ and $V_1=T\oplus\Omega$. Suppose a $\kk$-linear map $D:V_1\to A$ satisfies $D(\partial a)=0$ and 
\begin{equation*}
D(a_{(-1)}v)=aD(v)-a_{(0)}v, \quad  D(u_{(0)}v)=(Du)_{(0)}v+u_{(0)}Dv
\end{equation*}
for $a\in A$ and $u,v\in V_1$. Then $V$ has a unique quasi-conformal structure with $\mathscr L_0$ the given grading operator and $\mathscr L_1|_{V_1}=D$. Here $\mathscr L_{-1}=\partial$, $\mathscr L_1(A)=0$, and $\mathscr L_n(A\oplus V_1)=0$ for $n\geq2$.
\end{theorem}

\begin{proof}
The cited theorem constructs the action of polynomial vector fields $\kk[z]\partial_z$ and the identity \eqref{eq:qc-covariance}. Since $V$ is nonnegatively graded and $\deg(\mathscr L_n) = -n$, we have $\mathscr L_n V_r \subseteq V_{r-n} = 0$ for all $n > r$. Hence the action extends continuously to $\Der\cO_z$ with $\Der_+\cO_z$ acting locally nilpotently.
\end{proof}

For any vertex algebroid $\cA_{\kappa,\mathcal H}(A,\mathfrak a)$ constructed in Sec.~\ref{sec:action-levels-3form} via the action $\rho:\mathfrak{a}\to\Der(A)$, define
\begin{equation*}
D:(A\otimes_\kk\mathfrak a)\oplus\Omega_A^1\longrightarrow A,
\qquad
D\left(\sum_i a_i\xi_i+\omega\right)
:=\sum_i\rho_{\xi_i}(a_i).
\end{equation*}
It is easy to check, for $\tau,\nu\in T$,
\begin{equation*}
D(a\tau)=aD(\tau)+\rho_\tau(a),\quad  D([\tau,\nu])=\rho_\tau(D\nu)-\rho_\nu(D\tau).
\end{equation*}
In the associated truncated vertex algebra $u(\cA_{\kappa, \mathcal{H}}(A,\mathfrak{a}))$, the products $a_{(-1)}\tau$ and $\tau_{(0)}\nu$ differ from $a\tau$ and $[\tau,\nu]$ only by one-forms.  Since $D$ kills one-forms and $a_{(0)}\tau=-\rho_\tau(a)$, thus we verify the hypotheses of Thm.~\ref{thm:Li-qc-crit}.

\begin{corollary}
\label{prop:Yangian-qc}
Every vertex algebra $V_{\kappa,\mathcal H}(A,\mathfrak a)$ constructed in Sec.~\ref{sec:action-levels-3form} has a quasi-conformal structure with its given nonnegative grading and $\mathscr L_1|_{V_1}=D$. In particular, this applies to $\Ych_{\whkappa}(T^*\g)$ at every affine level (and to the spectral-cocycle extensions of Sec.~\ref{sec:cocycle-spectral-ext}).
\end{corollary}

In this paper, we will call the quasi-conformal structure given above the standard quasi-conformal structure.

The additional spectral grading on $\Ych_{\whkappa}(T^*\g)$ gives a second quasi-conformal structure, subject to the level condition below. Recall the dilation action $\delta_\lambda$ of Lem.~\ref{lem:spectral-affine-vf} and its infinitesimal generator $E_u$. On $A=\cO(G_-)$ it gives
\begin{equation*}
A=\bigoplus_{q\geq0}A_{[q]},
\qquad
A_{[q]}
=
\{a\in A:E_ua=qa\}.
\end{equation*}
Extend $E_u$ to the generators of $\mathcal{Y}^{ch}_{\whkappa}(T^*\mathfrak{g})$ by
\begin{equation*}
E_u(a)=qa\quad(a\in A_{[q]}),
\qquad
E_u(\Lcur_\xi)=\Lcur_{u\partial_u\xi}
\quad(\xi\in\dplus).
\end{equation*}
and require it to commute with vertex translation. The OPEs~\eqref{eq:action-level-OPE}, with $\mathcal H=0$, are homogeneous for this grading precisely when
\begin{equation}
\widehat\kappa(xu^m,yu^n)=0
\qquad \text{for }x,y\in\g,\quad m,n\geq0,\quad m+n>0.
\label{eq:mt-level-cond}
\end{equation}
In this case, we can write $\widehat\kappa(xu^m,yu^n)=\delta_{m0}\delta_{n0}\kappa(x,y)$. We abbreviate $\Ych_{\whkappa}$ to $\Ych_\kappa$ in this case. Throughout, $V_n$ denotes the subspace of standard weight $n$, and $V_{[q]}$ the subspace of spectral degree $q$.

\begin{proposition}
\label{prop:mt-qc}
Assume the level condition~\eqref{eq:mt-level-cond}. Then $\Ych_{\kappa}(T^*\g)$ admits a unique quasi-conformal structure $\{\mathscr L_n^{\mathrm{MT}}\}_{n\geq-1}$ in which the generators $a\in A_{[q]}$ and $\Lcur_{xu^m}$, for $x\in\g$ and $q,m\geq0$, are primary of respective mini-twistor weights
\begin{equation}
\Delta_a:=-q,
\qquad
\Delta_{\Lcur_{xu^m}}:=1-m.
\label{eq:mt-gen-weights}
\end{equation}
Equivalently, for each such generator $v$,
\begin{equation}
[\mathscr L_n^{\mathrm{MT}},Y(v,z)]
=
z^{n+1}\partial_zY(v,z)
+(n+1)\Delta_vz^nY(v,z),
\qquad n\geq-1.
\label{eq:mt-primary-cov}
\end{equation}
Its grading operator is $\mathscr L_0^{\mathrm{MT}}=\mathscr L_0-E_u$, where $\mathscr L_0$ is the standard grading operator of Cor.~\ref{prop:Yangian-qc}.
\end{proposition}

\begin{proof}
Put $\mathcal{Y}:=\Ych_{\kappa}(T^*\g)$. We first construct an action of $\kk[z]\partial_z$ on $\mathcal{Y}$. The vertex algebra $\mathcal{Y}$ is generated by the elements $a\in A$ and $\Lcur_\xi$, $\xi\in\dplus$. Thus every state is represented by a linear combination of finite expressions $v^1_{(j_1)}\cdots v^s_{(j_s)}\mathbf1$ in homogeneous generators.

Assign to each homogeneous generator $v$ the number $\Delta_v$ in~\eqref{eq:mt-gen-weights}. On these expressions define $\mathscr L_n^{\mathrm{MT}}\mathbf1=0$ and, recursively,
\begin{equation}
\mathscr L_n^{\mathrm{MT}}(v_{(j)}w)
=
v_{(j)}\mathscr L_n^{\mathrm{MT}}w+
\bigl((n+1)(\Delta_v-1)-j\bigr)v_{(n+j)}w,
\qquad n\geq-1.
\label{eq:mt-primary-mode-action}
\end{equation}
We must show that this action is well defined, i.e., independent of the expression representing a state of $\mathcal{Y}$. The OPEs translate into the following mode relations:
\begin{equation*}
\begin{aligned}
&[a_{(i)},b_{(j)}]=0,\quad [(\Lcur_\xi)_{(i)},a_{(j)}] =(\rho_\xi a)_{(i+j)},\\
&[(\Lcur_\xi)_{(i)},(\Lcur_\eta)_{(j)}] =(\Lcur_{[\xi,\eta]})_{(i+j)}
+i\,\kappa(\xi,\eta)\delta_{i+j,0}\,\id.
\end{aligned}
\end{equation*}
Together with $1_A=\mathbf1$, $a_{(-1)}b=ab$, and the vacuum relations $v_{(j)}\mathbf1=0$ for $j\geq0$, these present the universal vertex envelope. We check that the operators $\{\mathscr L_n^{\mathrm{MT}}\}_{n\geq-1}$ preserve these defining relations.
For $\xi=xu^m$, we have $\rho_\xi(A_{[q]})\subset A_{[q+m]}$, and in particular $\Delta_{\rho_\xi a}=\Delta_{\Lcur_\xi}+\Delta_a-1$. Consequently, the mixed commutator relation is preserved, since
\begin{align*}
\bigl[\bigl[\mathscr L_n^{\mathrm{MT}},(\Lcur_\xi)_{(i)}\bigr],a_{(j)}\bigr]
+\bigl[(\Lcur_\xi)_{(i)},
\bigl[\mathscr L_n^{\mathrm{MT}},a_{(j)}\bigr]\bigr] & = \left((n+1)(\Delta_{\rho_\xi a}-1)-(i+j)\right)
(\rho_\xi a)_{(n+i+j)}\\
& = \bigl[\mathscr L_n^{\mathrm{MT}},(\rho_\xi a)_{(i+j)}\bigr].
\end{align*}
The same calculation applies to the $[(\Lcur_\xi)_{(i)},(\Lcur_\eta)_{(j)}]$ relation because $\Delta_{\Lcur_{[\xi,\eta]}} =\Delta_{\Lcur_\xi}+\Delta_{\Lcur_\eta}-1$ and by the level condition \eqref{eq:mt-level-cond}. For homogeneous elements $a,b$ and $n\geq1$, $\mathscr L_n^{\mathrm{MT}}(a_{(-1)}b-ab)=\bigl((n+1)(\Delta_a-1)+1\bigr)a_{(n-1)}b = 0$, since $a_{(n-1)}b = 0$ for $n\geq1$. The remaining relations are straightforward to verify.

Applying the recursion twice gives
\begin{equation*}
[[\mathscr L_m^{\mathrm{MT}},\mathscr L_n^{\mathrm{MT}}],v_{(j)}]
=(m-n)\bigl((m+n+1)(\Delta_v-1)-j\bigr)v_{(m+n+j)}.
\end{equation*}
This is the commutator of $(m-n)\mathscr L_{m+n}^{\mathrm{MT}}$ with
$v_{(j)}$.  Both operators kill the vacuum, so induction on the number
of modes proves
\begin{equation*}
[\mathscr L_m^{\mathrm{MT}},\mathscr L_n^{\mathrm{MT}}]
=(m-n)\mathscr L_{m+n}^{\mathrm{MT}}
\qquad(m,n\geq-1,\ m+n\geq-1).
\end{equation*}

We have therefore constructed an action of $\kk[z]\partial_z$ on $\mathcal{Y}$, with $\mathscr L_{-1}^{\mathrm{MT}}=\partial$. The operators $\mathscr L_0^{\mathrm{MT}}$ and $\mathscr L_0-E_u$ have the same commutators with every generating mode and both kill the vacuum, so they agree on $\mathcal{Y}$. In particular, $\mathscr L_0^{\mathrm{MT}}$ acts on $\mathcal{Y}_r\cap \mathcal{Y}_{[q]}$ by the integer $r-q$.

To extend the action to $\Der\cO_z=\kk[[z]]\partial_z$, use the standard nonnegative grading $\mathcal{Y}=\bigoplus_{r\geq0}\mathcal{Y}_r$. If $v$ has standard weight $h$, then $v_{(j)}$ shifts standard weight by $h-j-1$, so induction in~\eqref{eq:mt-primary-mode-action} gives $\mathscr L_n^{\mathrm{MT}}\mathcal{Y}_r\subset \mathcal{Y}_{r-n}$. Hence, for $w\in\bigoplus_{s=0}^r\mathcal{Y}_s$, an arbitrary formal vector field acts by the finite sum $\bigl(\sum_{n\geq-1}f_n\mathscr L_n^{\mathrm{MT}}\bigr)w=\sum_{n=-1}^r f_n\mathscr L_n^{\mathrm{MT}}w$. This defines a continuous action of $\Der\cO_z$. Every $D\in\Der_+\cO_z$ lowers standard weight by at least one, so $D^{r+1}w=0$.

We can also verify \eqref{eq:qc-covariance} on every state. Applying the recursion to $v=v_{(-1)}\mathbf1$ gives $\mathscr L_0^{\mathrm{MT}}v=\Delta_vv$ and $\mathscr L_n^{\mathrm{MT}}v=0$ for $n\geq1$ on each homogeneous generator.  Summing the mode commutators in \eqref{eq:mt-primary-mode-action} gives \eqref{eq:mt-primary-cov}. This is exactly
\begin{equation*}
[\mathscr L_m^{\mathrm{MT}},Y(v,z)]
=
\sum_{j=0}^{m+1}\binom{m+1}{j}z^{m+1-j}
Y(\mathscr L_{j-1}^{\mathrm{MT}}v,z)
\end{equation*}
for a generating state $v$.

To pass to all states, let $W\subset \mathcal{Y}$ be the subspace on which this identity holds for every $m\geq-1$. The vertex-algebra (Borcherds) Jacobi identity implies that $u,v\in W \implies u_{(j)}v\in W$ for every $j\in\mathbb Z$. This is the calculation in \cite[proof of Thm.~3.6, Eq.~(3.18)]{MingLiSemiConformal}, which uses only the Jacobi identity and the binomial theorem. Since $W$ contains the vacuum and the generators of $\mathcal{Y}$, it equals $\mathcal{Y}$. This proves \eqref{eq:qc-covariance} for every state. Finally, any quasi-conformal structure with the prescribed primary generator weights must satisfy the recursion~\eqref{eq:mt-primary-mode-action} and annihilate the vacuum, which proves uniqueness.
\end{proof} 

We call this the \emph{mini-twistor quasi-conformal structure}. For a nonzero homogeneous state $v$, write $\operatorname{wt}_{\mathrm{MT}}(v)$ for its $\mathscr L_0^{\mathrm{MT}}$-eigenvalue; thus $\operatorname{wt}_{\mathrm{MT}}(v)=r-q$ when $v\in \mathcal{Y}_r\cap \mathcal{Y}_{[q]}$.

\begin{remark}
\label{rem:qc-pencil}
More generally, for each $k\in\mathbb Z$, the affine combination
\begin{equation*}
\mathscr L_n^{(k)} := (1-k)\mathscr L_n + k\mathscr L_n^{\mathrm{MT}}, \qquad n\geq-1,
\end{equation*}
defines a quasi-conformal structure on $\Ych_\kappa(T^*\g)$ with grading operator $\mathscr L_0^{(k)}=\mathscr L_0-kE_u$. We will not, however, pursue this larger family in the present work.
\end{remark}

For $\kk=\bbC$, simple $\g$, and non-critical level $k\neq-h^\vee$, the currents $\Lcur_x$ ($x\in\g$) generate the subalgebra $V_k(\g) \subset \Ych_{\kappa}(T^*\g)$ carrying the Sugawara conformal vector \cite{FrenkelBenZvi}
\begin{equation*}
\omega_{\mathrm{Sug}} = \frac{1}{2(k+h^\vee)}\sum_i\normal{\Lcur_{x_i}\Lcur_{x^i}}.
\end{equation*}
While $\omega_{\mathrm{Sug}}$ is conformal for $V_k(\g)$, it does not act conformally on the full chiral cotangent Yangian. In fact, the following theorem rules out any state whose zero mode generates vertex translation.

\begin{theorem}
Let $V$ be either $\Ych_{\kappa}(T^*\g)$, or $V_{\kappa,\mathcal H}(A, \mathfrak p_+)$. No state $\omega\in V$ satisfies
\begin{equation*}
\omega_{(0)}=\partial.
\end{equation*}
Consequently, $V$ has no conformal vector. 
\end{theorem}

\begin{proof}
Write $V=\bigoplus_{r\geq0}V_r$ for the standard grading.  A state has only finitely many homogeneous components, $(\omega_r)_{(0)}$ has degree $r-1$, and $\partial$ has degree one.  Hence it is enough to consider $\omega\in V_2$.

For both cases write $A=\mathcal O(\mathsf P_-)$ and $T=A\otimes\mathfrak p_+$, with $\mathsf P_-=G_-$ and $\mathfrak p_+=\dplus$ in the Yangian case.  Let $e\in \mathsf P_-$ be the identity and set
\begin{equation*}
\Omega_e^1:=\Omega_A^1/\mathfrak m_e\Omega_A^1,
\qquad
\mathfrak m_e:=\{a\in A:a(e)=0\},
\end{equation*}
and write $\overline\alpha$ for the class of $\alpha\in\Omega_A^1$. The dressing action fixes $e$, so $\rho_\xi(a)(e)=0$ for every $\xi\in\mathfrak p_+$ and $a\in A$.  It therefore induces an operator on $\Omega_e^1$, defined by $R_\xi(\overline{\dd a})
:=\overline{\dd(\rho_\xi a)}$.  Let $\operatorname{pr}_T:V_1=T\oplus\Omega_A^1\longrightarrow T$ be the canonical projection. We also use
\begin{equation*}
q_e:V_1=T\oplus\Omega_A^1\longrightarrow\Omega_e^1,
\qquad
q_e(\tau+\alpha)=\overline\alpha.
\end{equation*}

Suppose, for a contradiction, that $\omega_{(0)}=\partial$.  Then $\omega_{(0)}a=\dd a$ for every $a\in A$, and hence $\operatorname{pr}_T(\omega_{(0)}a)=0$.  The PBW theorem~\ref{thm:vertex-env-PBW} allows us to decompose:
\begin{equation*}
\omega=\omega_{TT}+\omega_{T\Omega}+\omega_{\Omega\Omega}+\partial\gamma, \qquad \gamma\in V_1,
\end{equation*}
where $\omega_{T\Omega}$ is a finite sum of terms $b_{(-1)}\normal{\Lcur_\xi\alpha}$ and $\omega_{\Omega\Omega}$ is a finite sum of terms $b_{(-1)}\normal{\alpha\beta}$, with $b\in A$, $\xi\in\mathfrak p_+$, and $\alpha,\beta\in\Omega_A^1$.  We may choose $\omega_{TT}$ to be of the form
\begin{equation*}
\omega_{TT}
=\frac12\sum_{i,j=1}^N
(B^{ij})_{(-1)}\normal{\Lcur_{\xi_i}\Lcur_{\xi_j}},
\end{equation*}
with $B^{ij}=B^{ji}$ and $\xi_1,\ldots,\xi_N\in\mathfrak p_+$ linearly independent. Non-symmetric part of $B^{ij}$ can be absorbed into the other three components. 

We can compute
\begin{equation*}
\bigl(\normal{\Lcur_\xi\Lcur_\eta}\bigr)_{(0)}a
=\rho_\eta(a)\Lcur_\xi+\rho_\xi(a)\Lcur_\eta
+\dd\bigl(\rho_\xi\rho_\eta(a)+\rho_\eta\rho_\xi(a)\bigr).
\end{equation*}
Consequently, $ \operatorname{pr}_T((\omega_{TT})_{(0)}a) =\sum_{i,j}B^{ij}\rho_{\xi_i}(a)\xi_j$. Moreover, $(\omega_{T\Omega})_{(0)}a\in\Omega_A^1$, $(\omega_{\Omega\Omega})_{(0)}a=0$ because one-forms have regular OPE with $A$, and $(\partial\gamma)_{(0)}=0$. It follows from $\operatorname{pr}_T(\omega_{(0)}a)=0$ and the freeness of
$T=A\otimes\mathfrak p_+$ that $\sum_iB^{ij}\rho_{\xi_i}=0$ for every $j$. Taking the K\"{a}hler differential and then projecting to $\Omega_e^1$ yields
\begin{equation*}
\sum_iB^{ij}(e)R_{\xi_i}(\overline{\dd a}) = 0.
\end{equation*}
Here the terms containing $\dd B^{ij}$ vanish because $\rho_{\xi_i}(a)(e)=0$.  Since the classes $\overline{\dd a}$ span $\Omega_e^1$, we obtain
\begin{equation*}
\sum_iB^{ij}(e)R_{\xi_i}=0
\qquad\text{for every }j.
\end{equation*}

We next apply $q_e$ to $\omega_{(0)}a$.  For any $b\in A$ we have
\begin{equation*}
q_e\!\left(
\bigl(b_{(-1)}\normal{\Lcur_\xi\Lcur_\eta}\bigr)_{(0)}a
\right)
=b(e)(R_\xi R_\eta+R_\eta R_\xi)\overline{\dd a}
+(\rho_\eta\rho_\xi(a))(e)\,\overline{\dd b}.
\end{equation*}
The last term is zero because $\rho_\eta(h)(e)=0$ for every $h\in A$. Therefore
\begin{equation*}
q_e((\omega_{TT})_{(0)}a)
=\frac12\sum_{i,j}B^{ij}(e)
(R_{\xi_i}R_{\xi_j}+R_{\xi_j}R_{\xi_i})\overline{\dd a}
=0,
\end{equation*}
For a term $b_{(-1)}\normal{\Lcur_\xi\alpha}$ occurring in
$\omega_{T\Omega}$, $q_e\!\left(
\bigl(b_{(-1)}\normal{\Lcur_\xi\alpha}\bigr)_{(0)}a
\right)
=b(e)\rho_\xi(a)(e)\overline\alpha=0$. Since $(\omega_{\Omega\Omega})_{(0)}a=0$ and
$(\partial\gamma)_{(0)}a=0$, we have $ q_e(\omega_{(0)}a)=0$ for every $a\in A$. On the other hand, $\omega_{(0)}=\partial$ gives $q_e(\omega_{(0)}a)=\overline{\dd a}$, contradiction. 
\end{proof}

\subsection{Cocycle-twisted quasi-conformal families}
\label{sec:twisted-qc}
In this section, we further extend the quasi-conformal structures established in the last section into one-parameter families. This can be done for both the standard quasi-conformal structure $\{\mathscr{L}_n\}$ on $V_{\kappa,\mathcal{H}}(A,\mathfrak{a})$ (thus also on $\Ych_{\whkappa}(T^*\mathfrak{g})$), and the mini-twistor quasi-conformal structure $\{\mathscr{L}_n^{\mathrm{MT}}\}$ on $\Ych_{\kappa}(T^*\mathfrak{g})$

As in Sec~\ref{sec:cocycle-spectral-ext}, we let $\lambda \in Z^1(\mathfrak{a},A)$ be a cocycle. We first define operators $\{\mathscr{S}_r\}_{r\geq 0}$ on $V_{\kappa,\mathcal{H}}(A,\mathfrak{a})$ by requiring $\mathscr S_r\mathbf1=0$ and
\begin{equation}\label{eq:Cr_def}
  [\mathscr S_r,Y(f,z)]=0,
\qquad
[\mathscr S_r,Y(\Lcur_\xi,z)]=-z^rY(\lambda_\xi,z),
\qquad f\in A,\quad \xi\in\mathfrak{a}.
\end{equation}
For $\Ych_{\whkappa}(T^*\mathfrak{g})$, we will specifically use the cocycle $\lambda \in Z^1(\mathfrak{l}_+,A)$ from \eqref{eq:spectral-log-cocycle} of Sec.~\ref{sec:cocycle-spectral-ext}:
$$
\lambda_\xi(g)=\Res_u\kappa
\left(\xi(u),(\Tu g)g^{-1}\right),
\qquad \xi\in\dplus,\quad g\in G_-.
$$

\begin{lemma}
\label{lem:qc-cocycle-operators}
For any cocycle $\lambda\in Z^1(\mathfrak a,A)$, The above formula defines uniquely operators $\{\mathscr S_r\}_{r\geq0}$ on $V_{\kappa,\mathcal H}(A,\mathfrak a)$, which satisfy $[\mathscr S_r,\mathscr S_\ell]=0$ for $r,\ell\geq0$, $[\partial,\mathscr S_0]=0$, and
$$
[\mathscr L_m,\mathscr S_r]
=-(m+r+1)\mathscr S_{m+r}
\quad(m\geq-1,\ r\geq0,\ m+r\geq0).
$$

Moreover, for $\Ych_{\kappa}(T^*\mathfrak{g})$ with the cocycle \eqref{eq:spectral-log-cocycle}, the same commutation relation holds with $\mathscr L_m^{\mathrm{MT}}$ in place of $\mathscr L_m$ under the level condition~\eqref{eq:mt-level-cond}.
\end{lemma}

\begin{proof}
We adjoin to $A$ an auxiliary commuting generator $s$ and extend $\rho$ to $A[s]$ by setting $\rho_\xi s=\lambda_\xi$. Since $[\rho_\xi,\rho_\eta]s =\rho_\xi\lambda_\eta-\rho_\eta\lambda_\xi =\lambda_{[\xi,\eta]}$, this defines a Lie algebra action of $\mathfrak a$ on $A[s]$. We extend the three form $\mathcal H$ by $A[s]$-linearity, using the natural map $\Omega_A^1\to\Omega_{A[s]}^1$. The construction in Sec.~\ref{sec:action-ext} therefore gives a vertex algebra $\Uenv(\cA_{\kappa,\mathcal H}(A[s],\mathfrak a))$, into which
$V_{\kappa,\mathcal H}(A,\mathfrak a)$ naturally embeds via the original generators. In this larger vertex algebra, we have the following OPEs:
$$
s(z)f(w)\sim0,\qquad s(z)\Lcur_\xi(w)\sim-\frac{\lambda_\xi(w)}{z-w},\qquad s(z)s(w)\sim0.
$$

Set $\mathscr S_r:=s_{(r)}|_{V_{\kappa,\mathcal H}(A,\mathfrak a)}$ for $r\geq0$. The OPEs above show that $s_{(r)}$ preserves $V_{\kappa,\mathcal H}(A,\mathfrak a)$ for $r\geq0$ and that the commutator identities~\eqref{eq:Cr_def} and $[\mathscr S_r,\mathscr S_\ell]=0$ hold. The relations~\eqref{eq:Cr_def}, together with $\mathscr S_r\mathbf1=0$, determine $\mathscr S_r$ uniquely because the elements $f\in A$ and the currents $\Lcur_\xi$ generate $V_{\kappa,\mathcal H}(A,\mathfrak a)$.

Applying Cor.~\ref{prop:Yangian-qc} to $\Uenv(\cA_{\kappa,\mathcal H}(A[s],\mathfrak a))$ makes $s$ primary of weight zero for the standard quasi-conformal structure. Consequently, $\mathscr S_r$ shifts standard weight by $-r-1$, and we have $[\mathscr L_m,s_{(r)}]=-(m+r+1)s_{(m+r)}$ and $[\partial,s_{(0)}]=0$. 

For $\Ych_{\kappa}(T^*\mathfrak{g})$, we can assign $s$ spectral degree zero. The proof of Prop.~\ref{prop:mt-qc} then applies because $\rho_{xu^m}s=\lambda_{xu^m}\in A_{[m]}$, making $s$ primary of weight zero for the mini-twistor quasi-conformal structure. Thus the same
identities hold with $\mathscr L_m^{\mathrm{MT}}$ in place of $\mathscr L_m$.
\end{proof}

Using the operators $\{\mathscr{S}_r\}$, we define, for each $\theta \in \kk$:
$$
\mathscr L_{-1}^{\theta}=\partial,
\qquad
\mathscr L_0^{\theta}=\mathscr L_0,
\qquad
\mathscr L_n^{\theta}=\mathscr L_n+
\theta\binom{n+1}{2}\mathscr S_{n-1}
\quad(n\geq1).
$$
For $\Ych_{\kappa}(T^*\mathfrak{g})$ with the cocycle \eqref{eq:spectral-log-cocycle}, define similarly $\mathscr L_{-1}^{\mathrm{MT},\theta}=\partial$, $\mathscr L_0^{\mathrm{MT},\theta}=\mathscr L_0-E_u$ and $\mathscr L_n^{\mathrm{MT},\theta}=\mathscr L_n^{\mathrm{MT}}+
\theta\binom{n+1}{2}\mathscr S_{n-1}$ for $n\geq1$.

\begin{proposition}
\label{prop:twisted-qc}
For every level $\kappa$ and closed three-form $\mathcal H$, the operators $\{\mathscr L_n^\theta\}_{n\geq-1}$ define a quasi-conformal structure on $V_{\kappa,\mathcal H}(A,\mathfrak a)$. 

For $\Ych_{\kappa}(T^*\mathfrak{g})$ with the cocycle \eqref{eq:spectral-log-cocycle}, the operators $\{\mathscr L_n^{\mathrm{MT},\theta}\}_{n\geq-1}$ also define a quasi-conformal structure on $\Ych_\kappa(T^*\mathfrak g)$.
\end{proposition}

\begin{proof}
We first consider the family $\{\mathscr L_n^\theta\}_{n\geq-1}$. For $m,n\geq0$ with $m+n\geq1$, applying Lem.~\ref{lem:qc-cocycle-operators} gives
$$
\begin{aligned}
[\mathscr L_m^\theta,\mathscr L_n^\theta]
&=(m-n)\mathscr L_{m+n}
+\theta(m+n)\left(\binom{m+1}{2}-\binom{n+1}{2}\right)
\mathscr S_{m+n-1}\\
&=(m-n)\mathscr L_{m+n}^\theta.
\end{aligned}
$$

The relations $[\partial,\mathscr S_r]=-r\mathscr S_{r-1}$ for $r\geq1$ and $[\partial,\mathscr S_0]=0$ give $[\partial,\mathscr L_n^\theta]=-(n+1)\mathscr L_{n-1}^\theta$ for $n\geq0$. Thus $\{\mathscr L_m^\theta\}$ furnish a representation of $\kk[z]\partial_z$.

In the auxiliary vertex algebra $\Uenv(\cA_{\kappa,\mathcal H}(A[s],\mathfrak a))$, we have $(\partial^2s)_{(j)}=j(j-1)s_{(j-2)}$. Thus we can identify
$$
\mathscr L_n^\theta-\mathscr L_n
=\frac{\theta}{2}(\partial^2s)_{(n+1)}
\big|_{V_{\kappa,\mathcal H}(A,\mathfrak a)}
\qquad(n\geq-1).
$$

For $v\in V_{\kappa,\mathcal H}(A,\mathfrak a)$ and $n\geq-1$, by vertex algebra Borcherds identity we have

$$
[(\partial^2s)_{(n+1)},Y(v,z)]
=\sum_{j=0}^{n+1}\binom{n+1}{j}
z^{n+1-j}Y((\partial^2s)_{(j)}v,z).
$$

Adding the action of $\{\mathscr L_n\}_{n\geq-1}$ proves~\eqref{eq:qc-covariance} for the action of twisted operators $\{\mathscr L_n^\theta\}_{n\geq-1}$.

Finally, for $n\geq1$, $\mathscr L_n^\theta$ lowers standard weight by $n$, so sufficiently high modes annihilate each state. Hence the action extends continuously to $\kk[[z]]\partial_z$ with $z^2\kk[[z]]\partial_z$ acting locally nilpotently. This proves the quasi-conformal axioms.

Under the level condition~\eqref{eq:mt-level-cond}, Lem.~\ref{lem:qc-cocycle-operators} supplies the same commutation relations for $\mathscr L_n^{\mathrm{MT}}$. The preceding calculations and weight argument therefore apply to the mini-twistor family as well.
\end{proof}

By definition, the above cocycle twisted quasi-conformal structures preserve the standard and mini-twistor weights, respectively. Homogeneous functions $f\in A$ remain primary, while the currents $\Lcur_\xi$ are not necessarily primary and satisfy
$$
\mathscr L_1^{\theta}\Lcur_\xi
=-\theta\lambda_\xi,
\qquad
\mathscr L_n^{\theta}\Lcur_\xi
=0
\quad(n\geq2).
$$
The same formulas hold for $\mathscr L_n^{\mathrm{MT},\theta}$.

\section{Vertex bialgebra structure}
\label{sec:ord-coprod}

The construction of our vertex algebroids and their vertex envelopes depends essentially only on the underlying action Lie algebroid. When this action Lie algebroid arises from an integrated matched pair of Lie algebras, the group multiplication on the negative factor provides additional structure. At zero level, this lifts to a vertex bialgebra structure, while at nonzero levels it yields a comodule vertex algebra over the zero-level vertex bialgebra.  

\subsection{Vertex bialgebras and comodule vertex algebras}
\label{sec:local-vbialgs}

We first recall the notion of vertex bialgebra, i.e. a vertex algebra with a compatible coalgebra structure. The standard definition of a vertex bialgebra allows the underlying vertex algebra to be non-local \cite[Def.~4.1]{LiSmashProduct}.  We use the more restrictive version, where the underlying vertex algebra is local.

Recall that the ordinary tensor product of two vertex algebras is again a vertex algebra, with
\begin{equation*}
Y_{V_1\otimes V_2}(a_1\otimes a_2,z)(b_1\otimes b_2)
=
Y_{V_1}(a_1,z)b_1\otimes Y_{V_2}(a_2,z)b_2.
\end{equation*}

\begin{definition}
\label{def:graded-local-vbialg}
A \emph{local vertex bialgebra} is a (local) vertex algebra $V$ equipped with
vertex-algebra morphisms
\begin{equation*}
\Delta:V\longrightarrow V\otimes V,
\qquad
\epsilon:V\longrightarrow\kk,
\end{equation*}
such that
\begin{equation*}
(\Delta\otimes\id)\Delta
=(\id\otimes\Delta)\Delta,
\qquad
(\epsilon\otimes\id)\Delta
=\id
=(\id\otimes\epsilon)\Delta.
\end{equation*}
\end{definition}
All vertex bialgebras considered here will be $\mathbb{Z}$ or $\mathbb{Z}_{\geq 0}$-graded and we further require that all morphisms preserve the grading.

The notion of local vertex bialgebra is much less studied than vertex bialgebra, and such structure turns out to be very restrictive. Since $\ker\epsilon$ is a proper vertex ideal, a simple vertex bialgebra is necessarily the one-dimensional vacuum algebra. Moreover, a degree-preserving counit is typically incompatible with a nonzero central term. Connected cocommutative vertex bialgebra is classified in \cite{HanLiXiao}, and is shown to be canonically isomorphic to the universal vertex bialgebra generated by its primitive vertex Lie algebra. In this section we construct vertex bialgebra structures on the universal vertex envelope of the zero-level vertex algebroid $\cA_0(A,\mathfrak p_+)$ introduced in Sec.~\ref{sec:action-levels-3form}. The resulting vertex bialgebra is generally noncocommutative: its weight-zero coproduct is pullback by multiplication on $\mathsf P_-$. Thus the cocommutative classification just cited does not apply.

Though vertex envelope in nonzero levels studied in this paper does not have bialgebra structure, they instead give comodule vertex algebra of the level zero vertex bialgebra. Here we recall the definition of comodule vertex algebra.
\begin{definition}
\label{def:right-comodule-valg}
Let $V$ be a (local) vertex bialgebra.  A \emph{right $V$- comodule vertex algebra} is a vertex algebra $W$ equipped with a vertex-algebra morphism
\begin{equation*}
\delta^R:W\longrightarrow W\otimes V
\end{equation*}
satisfying
\begin{equation*}
(\delta^R\otimes\id)\delta^R
=(\id\otimes\Delta)\delta^R,
\qquad
(\id\otimes\epsilon)\delta^R=\id.
\end{equation*}
\end{definition}

\subsection{Product vertex algebroid}
\label{sec:valg-ext-prods}

Given two Lie algebroids \((T_1,\rho_1)\) and \((T_2,\rho_2)\) on spaces \(X_1\) and \(X_2\), respectively, we have a natural product Lie algebroid \(T_1\boxplus T_2\) over \(X_1\times X_2\), whose underlying vector bundle is the external direct sum \(\pi_1^*T_1\oplus\pi_2^*T_2\) and whose anchor is \(\pi_1^*\rho_1\oplus\pi_2^*\rho_2\). This notion naturally generalizes to extended Lie algebroids: given extended Lie algebroids \((T_i,\Omega_i,\partial_i,\langle-,-\rangle_i)\) on \(X_i\), we define $T_1\boxplus T_2$ as above, the underlying \(\Omega\)-module of their product to be the external direct sum $\Omega_1\boxplus\Omega_2:=\pi_1^*\Omega_1\oplus\pi_2^*\Omega_2,$ while the derivation, the action of \(T_1\boxplus T_2\) on \(\Omega_1\boxplus\Omega_2\), and the pairing between them are induced componentwise from the corresponding structures on the two factors, with all mixed terms vanishing on sections pulled back from different factors.

As a simple example, for smooth Poisson varieties $(X_i,\pi_i)$, $i=1,2$. We have a product Poisson structure $\pi_1\oplus\pi_2$ on $X_1\times X_2$ characterized by
\begin{equation*}
\{a_1\otimes a_2,b_1\otimes b_2\}
=
\{a_1,b_1\}_{\pi_1}\otimes a_2b_2
+
a_1b_1\otimes\{a_2,b_2\}_{\pi_2}.
\end{equation*}
For the standard cotangent Lie algebroid $T(X_i,\pi_i)$ associated with the Poisson varieties, we have identification
\begin{equation*}
  T(X_1\times X_2,\pi_1\oplus\pi_2)\cong T(X_1,\pi_1)\boxplus T(X_2,\pi_2).
\end{equation*}
Under the canonical decomposition of (K\"{a}hler or smooth) differential forms $\Omega^1_{X_1\times X_2} \cong \pi_1^*\Omega^1_{X_1}\oplus\pi_2^*\Omega^1_{X_2}$, we also have an identification of extended Poisson cotangent Lie algebroids
\begin{equation*}
\cT(X_1\times X_2,\pi_1\oplus\pi_2)
\cong
\cT(X_1,\pi_1)\boxplus\cT(X_2,\pi_2).
\end{equation*}

We have a similar product construction for vertex algebroids. Let 
\begin{equation*}
\cA_i=(\mathcal T_i,\gamma_i,\chi_i,c_i), \qquad i=1,2,
\end{equation*}
be vertex $A_i$-algebroids with underlying extended Lie algebroids $\mathcal T_i$. We define the vertex $A_{12} = A_1\otimes_{\kk}A_2$-algebroid $\cA_1\boxplus\cA_2 = (A_{12},T_{12},...)$ as follows.  The underlying extended Lie $A_{12}$-algebroid is $\cT_1\boxplus\cT_2$ defined as above. In particular $T_{12} = (T_1\otimes A_{2}) \oplus(A_1\otimes T_2)$. Let $T_{12,0}=T_1\otimes 1\oplus 1\otimes T_2$ be the $\kk$-submodule of $T_{12}$. It generates $T_{12}$ over $A_{12}$. Write $\tau_1^{(1)}=\tau_1\otimes1$ and $\tau_2^{(2)}=1\otimes\tau_2$. The vertex-algebroid structure on the submodule $T_{12,0}$ is given by
\begin{equation}
  \label{eq:ext-prod-gamma}
  \begin{aligned}
    \gamma_{12}(a_1\otimes a_2,\tau_1^{(1)})
&=\gamma_1(a_1,\tau_1)\otimes a_2
+\tau_1(a_1)\otimes\partial_2a_2,\\
\gamma_{12}(a_1\otimes a_2,\tau_2^{(2)})
&=a_1\otimes\gamma_2(a_2,\tau_2)
+\partial_1a_1\otimes\tau_2(a_2),
  \end{aligned}
\end{equation}
together with
\begin{equation}
  \label{eq:ext-prod-frame-data}
  \begin{aligned}
    \chi_{12}(\tau_1^{(1)},\nu_1^{(1)})
&=\chi_1(\tau_1,\nu_1)\otimes1,
&
c_{12}(\tau_1^{(1)},\nu_1^{(1)})
&=c_1(\tau_1,\nu_1)\otimes1,
\\
\chi_{12}(\tau_2^{(2)},\nu_2^{(2)})
&=1\otimes\chi_2(\tau_2,\nu_2),
&
c_{12}(\tau_2^{(2)},\nu_2^{(2)})
&=1\otimes c_2(\tau_2,\nu_2),
\\
\chi_{12}(\tau_1^{(1)},\tau_2^{(2)})
&=0,
&
c_{12}(\tau_1^{(1)},\tau_2^{(2)})
&=0.
\end{aligned}
\end{equation}
It is easy to check that these data satisfy the requirement of the extension theorem~\ref{lem:frame-reconstruct}. We therefore obtain a vertex algebroid, denoted by $\cA_1\boxplus\cA_2$.

\begin{proposition}
\label{prop:ext-prod-envelope}
There is a natural isomorphism of graded vertex algebras
\begin{equation*}
\Uenv(\cA_1\boxplus\cA_2)
\xrightarrow{\ \sim\ }
\Uenv(\cA_1)\otimes_{\kk}\Uenv(\cA_2).
\end{equation*}
\end{proposition}

\begin{proof}
Write $U_i = \Uenv(\cA_i)$ and $U_{12} = \Uenv(\cA_1 \boxplus \cA_2)$. Recall that the universal envelope has the property that the $1$-truncation of $U_i$ is $t(U_i) \cong u(\cA_i)$, which is the split $1$-truncated vertex algebra associated to $\cA_i$. Thus we can check that the $1$-truncation of $U_1 \otimes U_2$ is $t(U_1 \otimes U_2) \cong u(\cA_1 \boxplus \cA_2)$. In other words, we have an isomorphism $u(\cA_1 \boxplus \cA_2) \xrightarrow{\sim} t(U_1 \otimes U_2)$. By the universal property of the vertex envelope, this induces a vertex algebra morphism
\begin{equation*}
\Phi \colon U_{12} \longrightarrow U_1 \otimes U_2.
\end{equation*}

Conversely, the natural inclusions of the two factors $\cA_1 \hookrightarrow \cA_1 \boxplus \cA_2$ and $\cA_2 \hookrightarrow \cA_1 \boxplus \cA_2$ induce morphisms of vertex algebras $j_i \colon U_i \longrightarrow U_{12}$. By the definition of the product vertex algebroid, all nonnegative vertex operations between generators coming from the two factors vanish in the $1$-truncated vertex algebra $t(U_{12})$. By degree considerations, these are the only potentially nonzero mixed nonnegative vertex products between such generators in $U_{12}$. In other words, the images of $j_1$ and $j_2$ mutually strongly commute as vertex subalgebras of $U_{12}$. This implies that we have a vertex algebra morphism
\begin{equation*}
  \Psi \colon U_1 \otimes U_2 \longrightarrow U_{12}.
\end{equation*}
The composition $\Phi \circ \Psi$ is the identity when restricted to the two commuting subalgebras $U_1 \otimes 1$ and $1 \otimes U_2$, which generate $U_1 \otimes U_2$. Hence $\Phi \circ \Psi = \id_{U_1 \otimes U_2}$. On the other hand, the composition $\Psi \circ \Phi$ is the identity on the $1$-truncation $t(U_{12})$, and since $U_{12}$ is defined by the universal envelope of $t(U_{12}) \cong u(\cA_1 \boxplus \cA_2)$, we conclude that $\Psi \circ \Phi = \id_{U_{12}}$. This proves the isomorphism.
\end{proof}

\subsection{Vertex coproduct from group multiplication}
\label{sec:mult-factor-alg}

Throughout this subsection, $\mathfrak p=\mathfrak p_+\bowtie\mathfrak p_-$ is either a finite-dimensional matched pair with a compatible algebraic integration of $\mathfrak p_-$ to $\mathsf P_-$ as in Sec.~\ref{sec:bicrossed-action-algs}, or the rational loop splitting.  Put $A=\cO(\mathsf P_-)$. For $g\in\mathsf P_-$ and $\xi\in\mathfrak p_+$, write
\begin{equation*}
M_g\xi=p_+\Ad_{g^{-1}}\xi,
\end{equation*}
Thus $M_{g_1g_2}=M_{g_2}M_{g_1}$, which follows because $\Ad_{g_2^{-1}}$ preserves $\mathfrak p_-$. In the rational case, Lem.~\ref{lem:rat-factor-poly-closure} ensures that $g\mapsto M_g\xi$ has polynomial coefficients and that $\rho_\xi$ preserves $A$.

In this section we show that the group structure on $\mathsf P_-$ gives rise to a coproduct on the level-zero vertex algebra constructed in Sec.~\ref{sec:action-levels-3form}. We suppress the fixed arguments $(A,\mathfrak p_+)$ and write
\begin{equation*}
\cA:=\cA_0(A,\mathfrak p_+),
\qquad
V:=\Uenv(\cA)
\end{equation*}

Let $m:\mathsf P_-\times\mathsf P_-\longrightarrow\mathsf P_-$ be the
group multiplication morphism. We first construct a morphism of extended Lie algebroids
\begin{equation*}
 (\Delta_A,\Delta_T,\Delta_\Omega): \cT \longrightarrow\cT\boxplus\cT.
\end{equation*}
The three components are defined by
\begin{equation}
  \Delta_A(f) =m^*f, \quad \Delta_T(\xi) =\xi^{(1)}+\bigl(M_{g_1}\xi\bigr)^{(2)},\quad \Delta_\Omega(\dd f)=\dd(m^*f),
\label{eq:ord-prod-comp}
\end{equation}
where $(M_{g_1}\xi)^{(2)}$ denotes the current in the second tensor factor with coefficients evaluated at $g_1 \in \mathsf P_-$ in the first tensor factor.

We also define the counit morphism $\epsilon_{\cT}: \cT \to \cT_{\mathrm{pt}}$, where $\cT_{\mathrm{pt}}$ is the extended algebroid $(\kk,0,0,0,0)$ over the point, which behaves like $0$ under the external sum $\boxplus$. Let $e \in \mathsf P_-$ be the group identity. We define $\epsilon_{\mathcal{T}}$ by
\begin{equation*}
\epsilon_A(f)=f(e), \qquad \epsilon_T=0, \qquad \epsilon_\Omega=0.
\end{equation*}

The following proposition is a immediate consequence of the standard infinitesimal matched-pair identities, see e.g. \cite{Majid1990}.

\begin{proposition}
  \label{prop:mult-factor-alg}
The triple $\Delta_{\cT} = (\Delta_A,\Delta_T,\Delta_\Omega)$ is a morphism of extended Lie algebroids. It further satisfies the coassociative identity:
\begin{equation*}
  (\Delta_{\cT}\boxplus \mathrm{id})\Delta_{\cT} = (\mathrm{id}\boxplus \Delta_{\cT})\Delta_{\cT}.
\end{equation*}
together with counit identity
\begin{equation*}
  (\epsilon_{\cT}\boxplus \mathrm{id})\Delta_{\cT} = (\mathrm{id}\boxplus \epsilon_{\cT})\Delta_{\cT} = \mathrm{id}.
\end{equation*}
\end{proposition}

\begin{proof}
In our conventions, the standard matched-pair identities from \cite[Theorem~3.1 and Lemma~3.2]{Majid1990} are
\begin{align}
&\rho_\xi(g_1g_2)
=(R_{g_2})_*\rho_\xi(g_1)
 +(L_{g_1})_*\rho_{M_{g_1}\xi}(g_2),
\label{eq:matched-pair-mult}\\
&\rho_\xi(M_g\eta)-\rho_\eta(M_g\xi)
 +[M_g\xi,M_g\eta]
=M_g[\xi,\eta].
\label{eq:matched-pair-bracket}
\end{align}
Although those results are stated in the smooth setting, the same identities hold in the algebraic setting of Section~3.2 verbatim.

The anchor of $\Delta_T(\xi)$ is $\widetilde\rho_\xi(g_1,g_2) = \bigl(\rho_\xi(g_1),\rho_{M_{g_1}\xi}(g_2)\bigr)$. Thus \eqref{eq:matched-pair-mult} says that the anchor vector field $\widetilde\rho_\xi$ is $m$-related to $\rho_\xi$ (i.e., $\widetilde\rho_\xi \circ m^* = m^* \circ \rho_\xi$), and \eqref{eq:matched-pair-bracket} says that $[\Delta_T(\xi),\Delta_T(\eta)]=\Delta_T([\xi,\eta])$.
By the discussion following Corollary~2.8, semilinearity and the Leibniz rule extend these compatibilities to $T$, and pullback on K\"ahler differentials gives the extended morphism. Finally, associativity of $m$ and $M_{g_1g_2}=M_{g_2}M_{g_1}$ give coassociativity of $\Delta$. The identities $M_e=\mathrm{id}$ and $\rho_\xi(e)=0$ give the stated counit and its identities.
\end{proof}

When the matched pair comes from a Manin triple, this coproduct has a cotangent Lie algebroid interpretation provided by the dressing--cotangent isomorphism of Thm.~\ref{thm:dressing-cotangent}. Let $F:(X_1,\pi_1)\to(X_2,\pi_2)$ be a Poisson morphism. By the naturality of the cotangent Lie algebroid under Poisson morphisms \cite[Theorem~3.8 and Addendum~3.8.4]{Huebschmann1990}, the pullback on functions and one-forms satisfies:
\begin{equation*}
[F^*\alpha,F^*\beta]_{\pi_1} = F^*[\alpha,\beta]_{\pi_2},
\qquad
\pi_1^\sharp(F^*\alpha)(F^*a) = F^*\bigl(\pi_2^\sharp(\alpha)(a)\bigr).
\end{equation*}
Furthermore, the anchor identity implies compatibility with contraction $\iota_{\pi_1^\sharp(F^*\alpha)}F^*\omega = F^*(\iota_{\pi_2^\sharp(\alpha)}\omega)$, and since pullback commutes with the de Rham differential, Cartan's formula ensures compatibility with Lie derivatives. Consequently,
\[
(F^*,F^*,F^*):
\cT(X_2,\pi_2)\longrightarrow
\cT(X_1,\pi_1)
\]
is a morphism of extended Lie algebroids, where the three components act on functions, on the cotangent Lie algebroid, and on the one-form module, respectively.

Apply this to the multiplicative Poisson structure $\pi_{\mathfrak p}$ on $\mathsf P_-$. By definition of multiplicativity, $m:
(\mathsf P_-\times \mathsf P_-,\pi_{\mathfrak p}\oplus\pi_{\mathfrak p})
\longrightarrow
(\mathsf P_-,\pi_{\mathfrak p})$ is Poisson. Hence, using the identification $\cT(\mathsf P_-\times \mathsf P_-, \pi_{\mathfrak p}\oplus\pi_{\mathfrak p})\cong \cT(\mathsf P_-,\pi_{\mathfrak p}) \boxplus \cT(\mathsf P_-,\pi_{\mathfrak p})$, pullback by $m$ defines an extended Lie algebroid morphism
\[
m^*:
\cT(\mathsf P_-,\pi_{\mathfrak p})
\longrightarrow
\cT(\mathsf P_-,\pi_{\mathfrak p})
\boxplus
\cT(\mathsf P_-,\pi_{\mathfrak p}).
\]
Similarly, evaluation at the identity defines the counit, since
$\pi_{\mathfrak p}(e)=0$. Under the dressing--cotangent isomorphism $\jmath$ of Thm.~\ref{thm:dressing-cotangent}, this coproduct agrees with the matched-pair coproduct of Prop.~\ref{prop:mult-factor-alg}. Indeed, the right Maurer--Cartan identity gives
\begin{equation*}
m^*\theta^R_\xi = (\theta^R_\xi)^{(1)} + \bigl(\theta^R_{M_{g_1}\xi}\bigr)^{(2)},
\end{equation*}
and therefore $(\jmath\boxplus\jmath)\Delta_T = m^*\jmath$. Together with $\Delta_A=m^*$ on functions and $\Delta_\Omega=m^*$ on K\"ahler differentials, this identifies the two morphisms of extended Lie algebroids.

We now move forward to define the coproduct on the level-zero vertex algebroid $\cA$ and its enveloping vertex algebra $V$. We write
\begin{equation*}
\cA_{\mathrm{pt}}:=(\cT_{\mathrm{pt}},0,0,0)
\end{equation*}
for the vertex algebroid of a point.

\begin{theorem}
\label{thm:zero-algebroid-coprod}
The extended Lie algebroid coproduct $\Delta_{\cT}$ admits a vertex-algebroid lift
\[
(\Delta_A,\Delta_T,\Delta_\Omega,h):
\cA \longrightarrow\cA\boxplus \cA
\]
satisfying $h|_{\mathfrak p_+}=0$.
The counit lifts to the vertex-algebroid morphism
\[
(\epsilon_A,0,0,0):
\cA \longrightarrow\cA_{\mathrm{pt}}.
\]
They satisfy the coassociativity and counit identities. 
\end{theorem}

\begin{proof}
By Prop.~\ref{prop:mult-factor-alg}, $(\Delta_A,\Delta_T,\Delta_\Omega)$ is a morphism of
extended Lie algebroids. Cor.~\ref{cor:frame-mor-crit} gives a unique $h$ satisfying~\eqref{eq:GMS-mor-gamma} and $h|_{\mathfrak p_+}=0$. The level-zero condition gives $\chi(\xi,\eta)=c(\xi,\eta)=0$. Thus conditions~\eqref{eq:GMS-mor-pair} and~\eqref{eq:GMS-mor-c} reduce to
\begin{equation}
\label{eq:zero-coprod}
  \chi_{12}(\Delta_T\xi,\Delta_T\eta)=0,
  \qquad
  c_{12}(\Delta_T\xi,\Delta_T\eta)=0,
  \qquad \xi,\eta\in\mathfrak p_+.
\end{equation}

At level zero, \eqref{eq:ext-prod-gamma}, and \eqref{eq:ext-prod-frame-data}, together with the reconstruction
formulas~\eqref{eq:frame-pairing}--\eqref{eq:frame-c}, give
\[
  \chi_{12}(x,y)=0,
  \qquad
  c_{12}(x,y)=0,
  \qquad
  x,y\in
  \mathfrak p_+^{(1)}+(A\otimes1)\mathfrak p_+^{(2)}.
\]
Since $\Delta_T\xi=\xi^{(1)}+(M_{g_1}\xi)^{(2)}$, this proves \eqref{eq:zero-coprod}. Cor.~\ref{cor:frame-mor-crit} therefore gives
the vertex-algebroid morphism
\[
  (\Delta_A,\Delta_T,\Delta_\Omega,h):
  \cA\longrightarrow\cA\boxplus\cA.
\]

For the counit, Prop.~\ref{prop:mult-factor-alg} gives the underlying morphism $(\epsilon_A,0,0)$ of extended Lie algebroids. The map $h_\epsilon:T\to0$ vanishes identically because the target one-form module is zero. Conditions~\eqref{eq:GMS-mor-gamma}, \eqref{eq:GMS-mor-pair}, and~\eqref{eq:GMS-mor-c} are then automatic. Hence Cor.~\ref{cor:frame-mor-crit} gives $(\epsilon_A,0,0,0):\cA\longrightarrow\cA_{\mathrm{pt}}$.

It remains to check the coalgebra identities. On the generating submodule $T_{12,0}=T\otimes1\oplus1\otimes T$, the $h$ components of $\Delta\boxplus\mathrm{id}$ and $\mathrm{id}\boxplus\Delta$ are given by $h_{\Delta\boxplus\mathrm{id}}(\tau^{(1)}+\nu^{(2)})=h(\tau)\otimes1$ and $h_{\mathrm{id}\boxplus\Delta}(\tau^{(1)}+\nu^{(2)})=1\otimes h(\nu)$. By composition formula \eqref{eq:v_algebroid_comp}, their composites with $\Delta$ evaluated on $\xi\in\mathfrak p_+$ are
\[
  h_{(\Delta\boxplus\mathrm{id})\circ\Delta}(\xi)
  =
  (\Delta_\Omega\boxplus\mathrm{id}_\Omega)h(\xi) + h_{\Delta\boxplus\mathrm{id}}(\Delta_T\xi),
  \qquad
  h_{(\mathrm{id}\boxplus\Delta)\circ\Delta}(\xi)
  =
  (\mathrm{id}_\Omega\boxplus\Delta_\Omega)h(\xi) + h_{\mathrm{id}\boxplus\Delta}(\Delta_T\xi).
\]
Since $h|_{\mathfrak p_+}=0$ and $\Delta_T\xi=\xi^{(1)}+(M_{g_1}\xi)^{(2)}$, Eqs.~\eqref{eq:GMS-mor-gamma} and~\eqref{eq:ext-prod-gamma} show that both expressions vanish, so $h_{(\Delta\boxplus\mathrm{id})\circ\Delta}|_{\mathfrak p_+}=0$ and $h_{(\mathrm{id}\boxplus\Delta)\circ\Delta}|_{\mathfrak p_+}=0$. Because their underlying extended Lie algebroid morphisms agree by Prop.~\ref{prop:mult-factor-alg}, uniqueness in Cor.~\ref{cor:frame-mor-crit} yields
\[
  (\Delta\boxplus\mathrm{id})\circ\Delta
  =
  (\mathrm{id}\boxplus\Delta)\circ\Delta.
\]
The same argument with $h_\epsilon=0$ gives $(\epsilon\boxplus\mathrm{id})\circ\Delta=\mathrm{id}=(\mathrm{id}\boxplus\epsilon)\circ\Delta$.
\end{proof}

By the universal-envelope adjunction, we immediately obtain the coproduct and counit maps on the vertex envelope $V$ of $\cA$
\begin{corollary}
\label{cor:zero-bialg}
The morphisms of Thm.~\ref{thm:zero-algebroid-coprod} induce graded vertex-algebra morphisms
\[
\Delta:V\longrightarrow V\otimes V, \qquad \epsilon:V\longrightarrow \kk.
\]

Consequently, $(V,\Delta,\epsilon)$ is a graded local vertex bialgebra.
\end{corollary}

Finally, we discuss the compatibility of the coproduct with the quasi-conformal structure.

\begin{theorem}
\label{thm:zero-coprod-qc}
For the standard quasi-conformal structure on $V(\mathsf P_-,\mathfrak p_+)$, the coproduct and counit satisfy
$$
\Delta\mathscr L_n =(\mathscr L_n\otimes\id+\id\otimes\mathscr L_n)\Delta, \qquad \epsilon\mathscr L_n=0, \qquad n\geq-1.
$$

For $\Ych(T^*\mathfrak g)$ with the cocycle \eqref{eq:spectral-log-cocycle}, the same identities hold for both the cocycle-twisted standard quasi-conformal structure $\{\mathscr L_n^\theta\}$ and cocycle-twisted minitwistor quasi-conformal structure:
\begin{equation*}
\Delta\mathscr L_n^{\mathrm{MT},\theta} = (\mathscr L_n^{\mathrm{MT},\theta}\otimes\id +\id\otimes\mathscr L_n^{\mathrm{MT},\theta})\Delta,
\qquad
\epsilon\mathscr L_n^{\mathrm{MT},\theta}=0,
\qquad n\geq-1.
\end{equation*}
\end{theorem}

\begin{proof}
  We first prove the assertions for $\theta=0$. For the standard quasi-conformal structure, compatibility with $\mathscr L_{-1}=\partial$ and with $\mathscr L_0$ follows because $\Delta$ and $\epsilon$ are graded vertex-algebra morphisms. Positive $\mathscr L_n$ modes annihilate $a \in A$, and $\mathscr L_1\Lcur_\xi=0$.  If
$M_{g_1}\xi=\sum_i a_i\xi_i$, then
\begin{equation*}
(\mathscr L_1\otimes\id+\id\otimes\mathscr L_1)
\left(
\Lcur_\xi^{(1)}
+\sum_i a_i^{(1)}\Lcur_{\xi_i}^{(2)}
\right)=0,
\end{equation*}
The operators $\mathscr L_n$, $n\geq2$, vanish on the $\leq 1$ truncation. The commutator identity~\eqref{eq:qc-covariance} extends these identities to all of $V$ by induction on iterated vertex products.

For the minitwistor action, multiplication and the map $M$ are equivariant for the spectral dilations $E_u$.  Hence
\begin{equation*}
\Delta E_u = (E_u\otimes\id+\id\otimes E_u)\Delta.
\end{equation*}
Together with
$\mathscr L_0^{\mathrm{MT}}=\mathscr L_0-E_u$, this proves the degree-zero identity.  Every homogeneous coordinate function and every current $\Lcur_\xi$ with homogeneous $\xi$ is $\{\mathscr L_n^{\mathrm{MT}}\}$-primary.  The expression $\Lcur_\xi^{(1)} +\sum_i a_i^{(1)}\Lcur_{\xi_i}^{(2)}$ is primary as well. Thus $\Delta$ sends every primary generator to a primary state of the same mini-twistor weight. By induction, \eqref{eq:mt-primary-cov} gives compatibility with $\{\mathscr L_n^{\mathrm{MT}}\}$ on all of $V$.

For the cocycle twisted structures, we use $\Delta\lambda_\xi =\lambda_\xi\otimes1+\sum_i a_i\otimes\lambda_{\xi_i}$, which follows from the product rule $\Tu(g_1g_2)=(\Tu g_1)g_2+g_1(\Tu g_2)$ and the invariance of $\kappa$. 

Set $ D_r:=\Delta\mathscr S_r -(\mathscr S_r\otimes\id+\id\otimes\mathscr S_r)\Delta$. By~\eqref{eq:Cr_def},
$$
D_rY(f,z)=Y(\Delta f,z)D_r,
\qquad D_r\mathbf1=0,
$$
and
$$
D_rY(\Lcur_\xi,z)-Y(\Delta\Lcur_\xi,z)D_r =-z^rY\left( \Delta\lambda_\xi-\lambda_\xi\otimes1 -\sum_i a_i\otimes\lambda_{\xi_i},z \right)\Delta =0.
$$
Thus, by induction we prove that $D_r = 0$ for all $r \geq 0$. Hence
$$
\Delta\mathscr S_r
=(\mathscr S_r\otimes\id+\id\otimes\mathscr S_r)\Delta.
$$
Finally, applying $\epsilon\otimes\epsilon$ gives $\epsilon\mathscr S_r=2\epsilon\mathscr S_r$, hence $\epsilon\mathscr S_r=0$.
\end{proof}

\subsection{Level-additive coproducts and vertex comodules}
\label{sec:ord-affine-levels}

Now we extend the coproduct construction to vertex algebroids of non-zero levels. Let $\kappa_0,\kappa_1,\kappa_2$ be invariant bilinear pairings and suppose there is a vertex-algebroid morphism
\begin{equation*}
\cA_{\kappa_0}
\longrightarrow
\cA_{\kappa_1}\boxplus\cA_{\kappa_2}
\end{equation*}
whose underlying extended-algebroid components are \eqref{eq:ord-prod-comp}. By Cor.~\ref{cor:frame-mor-crit}, any such lift is determined by the component $h$ on $\mathfrak p_+$ and must satisfy the corresponding identities. Let $\widetilde\rho_\xi$ be the anchor of $\Delta_T\xi$
from \eqref{eq:ord-prod-comp}.  The identity \eqref{eq:GMS-mor-pair} reads
\begin{equation*}
\kappa_0(\xi,\eta)
=\kappa_1(\xi,\eta)
+\kappa_2(M_{g_1}\xi,M_{g_1}\eta)
+\iota_{\widetilde\rho_\xi}h(\eta)
+\iota_{\widetilde\rho_\eta}h(\xi).
\end{equation*}
At $(e,e)$, one has $M_e=\id$ and
$\widetilde\rho_\xi(e,e)=0$, since $\rho_\xi(e)=0$.  Hence $\kappa_0=\kappa_1+\kappa_2$. 

The same identity also shows why the strict choice $h=0$ generally fails: $M_{g_1}$ need not preserve $\kappa_2$. The following extension condition produces a correction satisfying this identity; 

\begin{definition}
\label{def:iso-extendable-level}
An invariant level $\kappa$ on $\mathfrak p_+$ is \emph{isotropically extendable} if there is a symmetric bilinear form $\kappa^{\mathrm d}$ on $\mathfrak p$ that is invariant under the Lie algebra $\mathfrak p$ and under the integrated adjoint action of $\mathsf P_-$, and such that
\begin{equation*}
\kappa^{\mathrm d}|_{\mathfrak p_+\times\mathfrak p_+}
=\kappa,
\qquad
\kappa^{\mathrm d}|_{\mathfrak p_-\times\mathfrak p_-}=0.
\end{equation*}
\end{definition}

Let $\theta^L\in\Omega^1(\mathsf P_-,\mathfrak p_-)$ be the left Maurer--Cartan form on $\mathsf P_-$. Suppose $\kappa_2$ is isotropically extendable and has an extension $\kappa_2^{\mathrm d}$. First define the one-form on $\mathsf P_-$
\begin{equation*}
\beta_{\kappa_2}(\xi) = -\kappa_2^{\mathrm d}\bigl(M_g\xi,\theta^L\bigr) \in\Omega^1_{\mathsf P_-}.
\end{equation*}
Equivalently, the $\Ad$-invariance of $\kappa_2^{\mathrm d}$ and the isotropy of $\mathfrak p_-$ give $\beta_{\kappa_2}(\xi) =-\kappa_2^{\mathrm d}(\xi,\theta^R)$, where $\theta^R=\Ad_g(\theta^L)$ is the right Maurer--Cartan form on $\mathsf P_-$. On the subspace $\mathfrak p_+$, we specify the map
\begin{equation*}
h_{\kappa_2}(\xi) =\operatorname{pr}_1^*\beta_{\kappa_2}(\xi)\in\Omega^1_{\mathsf P_-\times\mathsf P_-}.
\end{equation*}
By Cor.~\ref{cor:frame-mor-crit}, any vertex-algebroid lift of $(\Delta_A,\Delta_T,\Delta_\Omega)$ has its component $h: T\to\Omega^1_{\mathsf P_-\times\mathsf P_-}$ uniquely determined on $T=A\otimes\mathfrak p_+$ by \eqref{eq:frame-h-ext}, which here simplifies to 
\begin{equation*}
h_{\kappa_2}(a\xi) = m^*(a)h_{\kappa_2}(\xi)-\gamma_{12}\bigl(m^*(a),\,\xi^{(1)}+(M_{g_1}\xi)^{(2)}\bigr),
\end{equation*}
where $\gamma_{12}$ is given by \eqref{eq:ext-prod-gamma} on $\cA_{\kappa_1}\boxplus\cA_{\kappa_2}$. 

According to Cor.~\ref{cor:frame-mor-crit}, $(\Delta_A,\Delta_T,\Delta_\Omega,h_{\kappa_2})$ defines a vertex-algebroid morphism if and only if $h_{\kappa_2}$ satisfies identities~\eqref{eq:GMS-mor-pair} and~\eqref{eq:GMS-mor-c} on the subspace $\mathfrak p_+$. For $\xi,\eta\in\mathfrak p_+$, these read
\begin{align}
\kappa_0(\xi,\eta)
&=
\kappa_1(\xi,\eta)
+\kappa_2(M_{g_1}\xi,M_{g_1}\eta)
+\iota_{\widetilde\rho_\xi}h_{\kappa_2}(\eta)
+\iota_{\widetilde\rho_\eta}h_{\kappa_2}(\xi),
\label{eq:level-add-pair-id}\\
c_{12}(\Delta_T\xi,\Delta_T\eta)
&+\frac12\dd\langle\Delta_T\xi,h_{\kappa_2}(\eta)\rangle
-\frac12\dd\langle\Delta_T\eta,h_{\kappa_2}(\xi)\rangle
\notag\\
&-\widetilde\rho_\xi\bigl(h_{\kappa_2}(\eta)\bigr)
+\widetilde\rho_\eta\bigl(h_{\kappa_2}(\xi)\bigr)
+h_{\kappa_2}([\xi,\eta])
=0.
\label{eq:level-add-c-id}
\end{align}
When these conditions are satisfied, we obtain a vertex algebroid morphism $\cA_{\kappa_0} \to \cA_{\kappa_1}\boxplus\cA_{\kappa_2}$, with the induced morphism on vertex envelopes acts on the currents $\Lcur_\xi$ by
\begin{equation}
\Delta_{\kappa_1,\kappa_2}(\Lcur_\xi) = \Lcur_\xi^{(1)} +\Lcur_{M_{g_1}\xi}^{(2)} +h_{\kappa_2}(\xi).
\label{eq:level-add-currents}
\end{equation}

\begin{theorem}
\label{thm:level-add-coprod}
\begin{enumerate}[label=\textup{(\arabic*)}]
\item Let $\kappa_0,\kappa_1,$ and $\kappa_2$ be invariant levels such that
$\kappa_0=\kappa_1+\kappa_2$, and suppose that $\kappa_2$ is
isotropically extendable.  Fix an extension
$\kappa_2^{\mathrm d}$ as in
Def.~\ref{def:iso-extendable-level}.  Then
\begin{equation*}
\Delta_{\kappa_1,\kappa_2}
:=
(\Delta_A,\Delta_T,\Delta_\Omega,h_{\kappa_2}):
\cA_{\kappa_0}
\longrightarrow
\cA_{\kappa_1}\boxplus\cA_{\kappa_2}
\end{equation*}
is a vertex-algebroid morphism.  Consequently, passing to vertex envelopes yields a vertex-algebra morphism
\begin{equation*}
\Delta_{\kappa_1,\kappa_2}: V_{\kappa_0}
\longrightarrow
V_{\kappa_1}\otimes V_{\kappa_2}.
\end{equation*}

\item Let $\kappa_1,\kappa_2,$ and $\kappa_3$ be invariant levels. Suppose that $\kappa_2$, $\kappa_3$, and $\kappa_2+\kappa_3$ have chosen isotropic extensions satisfying
\begin{equation*}
(\kappa_2+\kappa_3)^{\mathrm d} = \kappa_2^{\mathrm d}+\kappa_3^{\mathrm d}.
\end{equation*}
Then the coproducts satisfy the coassociativity relation
\begin{equation*}
(\Delta_{\kappa_1,\kappa_2}\boxplus\id)\circ\Delta_{\kappa_1+\kappa_2,\kappa_3}
=
(\id\boxplus\Delta_{\kappa_2,\kappa_3})\circ\Delta_{\kappa_1,\kappa_2+\kappa_3}
\end{equation*}
as morphisms of vertex algebroids from $\cA_{\kappa_1+\kappa_2+\kappa_3}$ to $\cA_{\kappa_1}\boxplus\cA_{\kappa_2}\boxplus\cA_{\kappa_3}$, and consequently the coassociativity identity of vertex-algebra homomorphisms from $V_{\kappa_1+\kappa_2+\kappa_3}$ to $V_{\kappa_1}\otimes V_{\kappa_2}\otimes V_{\kappa_3}$.
\end{enumerate}
\end{theorem}

\begin{proof}
(1) It remains only to verify the identities \eqref{eq:level-add-pair-id} and \eqref{eq:level-add-c-id} for $\xi,\eta\in\mathfrak p_+$. In the next calculation all one-forms and contractions occur on the first factor of $\mathsf P_-\times\mathsf P_-$, and the corresponding pullbacks $\operatorname{pr}_1^*$ are suppressed.  Set $a_\xi=M_g\xi$ and $b_\xi=p_-\Ad_{g^{-1}}\xi$. Since $\iota_{\rho_\eta}\theta^L=-b_\eta$, one has
\begin{equation*}
\iota_{\rho_\eta}h_{\kappa_2}(\xi)
=\kappa_2^{\mathrm d}(a_\xi,b_\eta).
\end{equation*}
Invariance and the isotropy of $\mathfrak p_-$ give
\begin{equation*}  \kappa_2(\xi,\eta) =\kappa_2(a_\xi,a_\eta) + \kappa_2^{\mathrm d}(a_\xi,b_\eta) +\kappa_2^{\mathrm d}(b_\xi,a_\eta).
\end{equation*}
Together we verify \eqref{eq:level-add-pair-id}:
\begin{equation*}
  \kappa_1(\xi,\eta) +\kappa_2(a_\xi,a_\eta) +\iota_{\rho_\xi}h_{\kappa_2}(\eta) +\iota_{\rho_\eta}h_{\kappa_2}(\xi) =\kappa_0(\xi,\eta).
\end{equation*}

It remains to verify \eqref{eq:level-add-c-id} for $\xi,\eta\in\mathfrak p_+$. Differentiating $\Ad_{g^{-1}}\xi=a_\xi+b_\xi$ yields
\begin{equation*}
\dd(a_\xi+b_\xi)=-[\theta^L,a_\xi+b_\xi].
\end{equation*}
With the contration identity $\iota_{\rho_\xi}\theta^L=-b_\xi$, we get $\rho_{\xi}(\theta^L) = -db_{\xi} - [\theta^L,b_{\xi}] = da_{\xi} + [\theta^L,a_{\xi}]$.

By the definition of the external product~\eqref{eq:ext-prod-frame-data} and the reconstruction \eqref{eq:frame-c}, the operation $c_{12}$ on $\cA_1\boxplus\cA_2$ evaluates to
\begin{equation*}
c_{12}(\Delta_T\xi,\Delta_T\eta) = \frac{1}{2}(\kappa_2(a_\xi,\dd a_\eta) - \kappa_2(a_\eta,\dd a_\xi)). 
\end{equation*}
Using the matched-pair bracket identity~\eqref{eq:matched-pair-bracket}, we reduce the combination $\widetilde\rho_\xi\bigl(h_{\kappa_2}(\eta)\bigr) - \widetilde\rho_\eta\bigl(h_{\kappa_2}(\xi)\bigr)
- h_{\kappa_2}([\xi,\eta])$ to 
\begin{equation*}
  2c_{12}(\Delta_T\xi,\Delta_T\eta) - \kappa_2^{\mathrm d}(\theta^L,[a_\xi,a_\eta]).
\end{equation*}
On the other hand $\langle\Delta_T\xi, h_{\kappa_2}(\eta)\rangle = \kappa_2^{\mathrm d}(a_\eta,b_\xi) $, so we have 
\begin{equation*}
  \frac{1}{2}\dd (\langle\Delta_T\xi, h_{\kappa_2}(\eta)\rangle + \langle\Delta_T\eta, h_{\kappa_2}(\xi)\rangle ) = c_{12}(\Delta_T\xi,\Delta_T\eta) - \kappa_2^{\mathrm d}(\theta^L,[a_\xi,a_\eta])
\end{equation*}
Combining the above two equations proves \eqref{eq:level-add-c-id}.

(2) For $\xi\in\mathfrak p_+$, the left Maurer--Cartan form on $\mathsf P_-\times\mathsf P_-$ satisfies $\theta^L_{g_1g_2} =\Ad_{g_2^{-1}}\theta_1^L+\theta_2^L$. Contracting with $M_{g_1g_2}\xi$ and using the $\Ad$-invariance of $\kappa^{\mathrm d}$ yields the identity
\begin{equation}
m^*\beta_{\kappa}(\xi)
=\operatorname{pr}_1^*\beta_{\kappa}(\xi)
+\operatorname{pr}_2^*\beta_{\kappa}(M_{g_1}\xi).
\label{eq:h-mult-cocycle}
\end{equation}

Now consider the iterated composites in the coassociativity identity. By Prop.~\ref{prop:mult-factor-alg}, their underlying extended Lie algebroid morphisms agree. Let $m_{12}(g_1,g_2,g_3)=g_1g_2$. By the composition formula~\eqref{eq:v_algebroid_comp}, for $\xi\in\mathfrak p_+$, the one-form components of the left- and right-associated composites are, respectively,
\begin{align*}
h_{(\Delta_{\kappa_1,\kappa_2}\boxplus\id)\circ\Delta_{\kappa_1+\kappa_2,\kappa_3}}(\xi)
&=
\operatorname{pr}_1^*\beta_{\kappa_2}(\xi)
+m_{12}^*\beta_{\kappa_3}(\xi),
\notag\\
h_{(\id\boxplus\Delta_{\kappa_2,\kappa_3})\circ\Delta_{\kappa_1,\kappa_2+\kappa_3}}(\xi)
&=
\operatorname{pr}_1^*\beta_{\kappa_2}(\xi)
+\operatorname{pr}_1^*\beta_{\kappa_3}(\xi)
+\operatorname{pr}_2^*
  \beta_{\kappa_3}(M_{g_1}\xi).
\end{align*}
These two expressions agree by the identity~\eqref{eq:h-mult-cocycle}. Thus the two composites have identical underlying extended Lie algebroid morphisms and identical maps $h$ on $\mathfrak p_+$. Uniqueness in Cor.~\ref{cor:frame-mor-crit} implies the vertex-algebroid coassociativity identity.
\end{proof}

\begin{remark}
Although in our notation $\Delta_{\kappa_1,\kappa_2}$ only depends on the levels, it should be kept in mind that in our construction it also depends on the choice of isotropic extension $\kappa_2^{\mathrm d}$ of $\kappa_2$.
\end{remark}

Now we can specialize this construction to $\kappa_0=\kappa_1=\kappa$ and $\kappa_2=0$, with $0^{\mathrm d}=0$. In this case we have:
\begin{corollary}
\label{prop:canon-affine-coact}
For every invariant level $\kappa$, the vertex-algebra homomorphism
\begin{equation*}
\delta_{\kappa}^R:=\Delta_{\kappa,0}:
V_{\kappa}\longrightarrow V_{\kappa}\otimes V.
\end{equation*}
equips $V_{\kappa}$ as a right comodule vertex algebra over the zero-level vertex bialgebra $V$.
\end{corollary}

For simple $\g$ over $\mathbb C$ with Killing form $\kappa$, every invariant symmetric bilinear form on $\g[u]$ has the form
\begin{equation*}
\widehat\kappa_\varphi(xf,yg)
=
\kappa(x,y)\varphi(fg),
\qquad x,y\in\g,\quad f,g\in\kk[u],
\end{equation*}
for a unique linear functional $\varphi\in(\kk[u])^\vee$ (cf.~\cite[Thm.~2]{ZusmanovichCurrentForms}). 

We extend this form to $\g((u^{-1}))$ by
\begin{equation*}
\widehat\kappa_\varphi^{\mathrm d}(X,Y)
=\varphi\!\left(\bigl(\kappa(X(u),Y(u))\bigr)_+\right),
\qquad
X,Y\in\g((u^{-1})).
\end{equation*}
This symmetric bilinear form is $\g((u^{-1}))$- and $G_-$-invariant because the ad-invariance identity $\kappa([Z,X],Y)+\kappa(X,[Z,Y])=0$ holds pointwise in $u$ as Laurent series. It restricts to $\widehat\kappa_\varphi$ on $\g[u]$ and vanishes identically on $u^{-1}\g[[u^{-1}]]\times u^{-1}\g[[u^{-1}]]$. Thus, when $\g$ is simple, every invariant level on $\g[u]$ is canonically isotropically extendable.

Under this isotropic extension, the one-form component $h_{\whkappa_\varphi}(\xi)$ of Thm.~\ref{thm:level-add-coprod} is
\begin{equation*}
h_{\whkappa_\varphi}(\xi)
=
-\widehat\kappa_\varphi^{\mathrm d}\bigl(M_{g_1}\xi, \theta^L_{g_1}\bigr).
\end{equation*}
Because $\deg_u(M_{g_1}\xi) \leq \deg_u\xi$, we can check that the expression is a finite sum, and $h_{\whkappa_\varphi}(\xi)$ belongs to the algebraic K\"ahler module.

If $\g=\bigoplus_{i=1}^r\g_i$ is semisimple, every invariant symmetric bilinear form on $\g[u]$ decomposes as an orthogonal direct sum $\sum_{i=1}^r\widehat\kappa_{\varphi_i}$ over the simple ideals, and $\sum_{i=1}^r\widehat\kappa_{\varphi_i}^{\mathrm d}$ supplies an isotropic extension. If $\g=\g_{ss}\oplus\mathfrak z$ is reductive (such as $\mathfrak{gl}_N$), invariance forces $\widehat\kappa(\g_{ss}[u],\mathfrak z[u])=0$. On the center $\mathfrak z[u]$, every symmetric bilinear form $\widehat\kappa_{\mathfrak z}$ is invariant since $[\mathfrak z,\cdot]=0$. Setting $(\widehat\kappa_{\mathfrak z})^{\mathrm d}(Z,W):=\widehat\kappa_{\mathfrak z}(Z_+,W_+)$ gives an extension to $\mathfrak z((u^{-1}))$ that is trivially $\g((u^{-1}))$- and $G_-$-invariant and vanishes on $u^{-1}\mathfrak z[[u^{-1}]]\times u^{-1}\mathfrak z[[u^{-1}]]$. Thus every invariant level on a reductive Lie algebra is isotropically extendable. Consequently, because the assignment $\whkappa \mapsto \whkappa^{\mathrm d}$ is linear, the additivity condition $(\whkappa_2+\whkappa_3)^{\mathrm d} = \whkappa_2^{\mathrm d} + \whkappa_3^{\mathrm d}$ holds automatically. Thus, for any reductive Lie algebra $\g$, the chiral cotangent Yangians admit a coassociative level-colored coproduct
\begin{equation*}
\Delta_{\whkappa_1,\whkappa_2}:
\Ych_{\whkappa_1+\whkappa_2}(T^*\g)
\longrightarrow
\Ych_{\whkappa_1}(T^*\g)\otimes\Ych_{\whkappa_2}(T^*\g),
\end{equation*}
equipping $\Ych_{\whkappa}(T^*\g)$ with a canonical right coaction $\delta_{\whkappa}^R := \Delta_{\whkappa,0}: \Ych_{\whkappa}(T^*\g) \to \Ych_{\whkappa}(T^*\g) \otimes \Ych(T^*\g)$ over the level-zero chiral cotangent Yangian.

\subsection{Zhu reduction and the cotangent Yangian}
\label{sec:ord-Zhu-bialg}

In this subsection we compute the Zhu algebra of $\Ych_{\whkappa}(T^*\g)$, and compare it with the cotangent Yangian of Abedin--Niu \cite{AbedinNiu}. We first recall the definition of Zhu algebra of a vertex algebra \cite{DeSoleKacZhu}. Let $V=\bigoplus_{n\geq0}V_n$ be a nonnegatively graded vertex algebra. For a homogeneous $v\in V$ of weight
$\Delta_v$, define
\begin{equation*}
  \begin{aligned}
    v*w
&:=\operatorname{Res}_z
\frac{(1+z)^{\Delta_v}}{z}Y(v,z)w
=\sum_{j\geq0}\binom{\Delta_v}{j}v_{(j-1)}w,
\\
v\circ w
&:=\operatorname{Res}_z
\frac{(1+z)^{\Delta_v}}{z^2}Y(v,z)w
=\sum_{j\geq0}\binom{\Delta_v}{j}v_{(j-2)}w.
  \end{aligned}
\end{equation*}
Let $O(V)$ be the span of the $\circ$ products and set
\begin{equation*}
\Zhu(V):=V/O(V).
\end{equation*}
It is equipped with the associative product induced by the $*$-product. We write $[v]$ for the class of $v$. The standard Zhu identities give $ [\partial v]=-\Delta_v[v]$.

Recall that the Lie algebroid enveloping algebra $U_A(T)$ of a Lie $A$-algebroid $T$ is generated by $A$ and $T$, with relations $\tau a-a\tau=\rho_\tau(a)$, $\tau\nu-\nu\tau=[\tau,\nu]_T$, and $a\tau=(a\tau)_T$, for $a\in A$ and $\tau,\nu\in T$. The following result of Arakawa--Chebotarov--Malikov \cite[Thm.~3.1(2)]{ArakawaChebotarovMalikov} allows us to identify the Zhu algebra of a vertex algebroid with the enveloping algebra of its underlying Lie algebroid
\begin{theorem}\label{thm:Zhu-GMS-LR}
Let $\cA=(\mathcal T,\gamma,\chi,c)$ be a vertex algebroid over a field $\kk$ of
characteristic zero.
Assume that $\Omega=A\,\partial A$ and that $T$ and $\Omega$ are free
$A$-modules.  Then there is a canonical algebra isomorphism
\begin{equation*}
{
U_A(T)
\xrightarrow{\sim}
\Zhu\bigl(\Uenv(\cA)\bigr).}
\end{equation*}
\end{theorem}


Let $\mathfrak a$ act on $A$ through $\rho:\mathfrak a\to\Der(A)$. For the action Lie algebroid $T=A\otimes_{\kk}\mathfrak a$, one has
\begin{equation*}
U_A(T)\cong A\#_\rho U(\mathfrak a).
\end{equation*}
The smash product has underlying vector space $A\otimes U(\mathfrak a)$, and the defining cross relation is $[1\#\xi,a\#1]=\rho_\xi(a)\#1$, for $a\in A$, $\xi\in\mathfrak a$. For $\mathfrak a=\dplus=\g[u]$ and $A=\cO(G_-)$, with the matched-pair action $\rho$, this gives:
\begin{corollary}
    We have an isomorphism of associative algebras:
    \begin{equation}
\Zhu(\Ych_{\whkappa}(T^*\g))\cong\cO(G_-)\#_\rho U(\g[u])
\label{eq:Zhu-action-smash}
\end{equation}
\end{corollary}

Zhu reduction passes from a nonnegatively graded vertex algebra to an associative algebra.  It is functorial for grading-preserving vertex-algebra morphisms, and $\Zhu(V\otimes W)\cong\Zhu(V)\otimes\Zhu(W)$ as associative algebras \cite[Thm.~4.1]{MilasTensorProduct}.  Consequently, a graded local vertex bialgebra $(V,\Delta,\epsilon)$ gives an ordinary bialgebra $\Zhu(V)$: the maps induced by $\Delta$ and $\epsilon$ are its coproduct and counit, and their identities follow from the corresponding vertex-bialgebra identities.

We finally compare this Zhu algebra with the cotangent Yangian of Abedin--Niu.  Assume $\kk=\bbC$ and that $\g$ is complex simple.  For Yang's rational $r$-matrix, their formal cotangent Yangian is the topological smash product
\begin{equation*}
Y_\hbar(T^*\g) := U(\g[[u]])[[\hbar]] \mathbin{\#_{\bbC[[\hbar]]}} \Sym(\g^*[[u]])[[\hbar]].
\end{equation*}
The smash product action comes from the same matched pair as we used in Sec.~\ref{sec:bicrossed-action-algs}, and in this notation the two factors are loop-adically completed:
\begin{equation*}
U(\g[[u]]) :=\varprojlim_N U\bigl(\g[u]/u^N\g[u]\bigr),\qquad\Sym(\g^*[[u]]) :=\varprojlim_N\Sym\bigl(\g^*[u]/u^N\g^*[u]\bigr).
\end{equation*}
The tensor product in $Y_\hbar(T^*\g)$ is completed for this loop topology, and $[[\hbar]]$ denotes the separate $\hbar$-adic completion \cite[Eqs.~(2.15)--(2.17), Thm.~3.5, and Sec.~3.2.1]{AbedinNiu}. 

There is a dense subalgebra $Y_\hbar^\circ(T^*\g)\subset Y_\hbar(T^*\g)$, generated over $\bbC[[\hbar]]$ by $(T^*\g)[u]$.  It is a Hopf subalgebra with generator formulas polynomial in $\hbar$
\cite[Prop.~3.9]{AbedinNiu}.  Their graded evaluation defines $Y_1^\circ(T^*\g)$ \cite[Secs.~3.1.4 and 3.2.1]{AbedinNiu}. Since both smash products come from the same matched pair, \eqref{eq:Zhu-action-smash} gives, at level zero, the algebra isomorphism
\begin{equation*}
\Phi_1: \Zhu\bigl(\Ych(T^*\g)\bigr) \xrightarrow{\sim} \cO(G_-)\#_\rho U(\g[u]) \xrightarrow{\sim} Y_1^\circ(T^*\g).
\end{equation*}

\begin{corollary}
\label{cor:Zhu-dense-bialg-comp}
Assume that $\kk=\bbC$ and that $\g$ is complex simple.  At level zero, the algebra isomorphism $\Phi_1$ is an isomorphism of bialgebras
\begin{equation*}
\Phi_1: \Zhu\bigl(\Ych(T^*\g)\bigr) \xrightarrow{\sim} Y_1^\circ(T^*\g).
\end{equation*}
\end{corollary}

\begin{proof}
Applying the preceding observation to the coproduct and counit of Cor.~\ref{cor:zero-bialg} gives $\Zhu(\Ych(T^*\g))$ a bialgebra structure.  For $f\in A$ and $\xi\in\g[u]$, its coproduct and counit are
\begin{equation*}
  \Delta_{\mathrm{Zhu}}(f)=m^*f, \quad \Delta_{\mathrm{Zhu}}(\xi) =\xi^{(1)}+\bigl(M_{g_1}\xi\bigr)^{(2)},\quad  \epsilon_{\mathrm{Zhu}}(f)=f(e),\quad \epsilon_{\mathrm{Zhu}}(\xi)=0.
\end{equation*}
These agree on generators with the matched-pair coproduct and counit defining the bialgebra structure on $Y_1^\circ(T^*\g)$ in \cite[Thm.~3.5, Prop.~3.9]{AbedinNiu}. Hence $\Phi_1$ is an isomorphism of bialgebras.
\end{proof}

\subsection{The matrix vertex bialgebra for \texorpdfstring{$G=GL_N$}{G = GL(N)}}
\label{sec:matrix-vbialg}

Continuing Sec.~\ref{sec:matrix-rat-realization}, when $\g = \mathfrak{gl}_N$, the vertex algebra and coalgebra structures of the chiral cotangent Yangian $\Ych(T^*\mathfrak{gl}_N)$ specialize to concrete matrix formulas. 

The strong generators of $\Ych(T^*\mathfrak{gl}_N)$ consist of the polynomial coordinate functions $t_{ij}^{(r)} \in \mathcal O(G_-)$ ($1\leq i,j\leq N$, $r\geq1$) organized into the matrix series
\begin{equation*}
T(u,z) = \bigl(t_{ij}(u,z)\bigr)_{i,j=1}^N,
\qquad
t_{ij}(u,z) = \delta_{ij} + \sum_{r\geq 1} t_{ij}^{(r)}(z) u^{-r},
\end{equation*}
and the Lie algebra currents $\Lcur_\xi \in Y$ for $\xi(u)\in\mathfrak{gl}_N[u]$, spanned by the basis currents $\Lcur_{E_{ij}u^r}$.

We identify $\Omega_A^1$ with its image in the vertex envelope. Under this identification, vertex translation is identified with the algebraic de Rham differential, i.e. $\partial f = \dd f$ for $f\in A$.

For a general invariant level $\whkappa$ on $\mathfrak{gl}_N[u]$, the defining OPE of $\Ych_{\whkappa}(T^*\mathfrak{gl}_N)$ read
\begin{align*}
T(u_1,z) T(u_2,w) &\sim 0, \\
\Lcur_\xi(z) T(u,w) &\sim \frac{\rho_\xi(T)(u,w)}{z-w} = \frac{T(u,w)\bigl(T(u,w)^{-1}\xi(u) T(u,w)\bigr)_+ - \xi(u) T(u,w)}{z-w}, \\
\Lcur_\xi(z) \Lcur_\eta(w) &\sim \frac{\whkappa(\xi,\eta)}{(z-w)^2} + \frac{\Lcur_{[\xi,\eta]}(w)}{z-w}. 
\end{align*}
At level zero, the double pole drops out, giving the loop Lie algebra OPE $\Lcur_\xi(z) \Lcur_\eta(w) \sim \frac{\Lcur_{[\xi,\eta]}(w)}{z-w}$.

The zero-level vertex coproduct $\Delta$ is determined on the generators by group multiplication and the dressing action. Group multiplication $T(u) \mapsto T_1(u)T_2(u)$ gives
\begin{equation*}
\Delta\bigl(t_{ij}(u,z)\bigr) = \sum_{a=1}^N t_{ia}(u,z) \otimes t_{aj}(u,z),
\end{equation*}
with counit $\epsilon\bigl(T(u,z)\bigr) = I$. On the currents $\Lcur_\xi$, the matrix dressing action $M_T(\xi) = (T^{-1}\xi T)_+$ gives
\begin{equation*}
\Delta(\Lcur_\xi) = \Lcur_\xi \otimes 1 + \Lcur_{(T_1^{-1}\xi T_1)_+}^{(2)},
\qquad
\epsilon(\Lcur_\xi) = 0.
\end{equation*}
Expanding $(T^{-1}\xi T)_+ = \sum_{i,j=1}^N \sum_{m\geq 0} C_{ij,m}(T) E_{ij}u^m \in \mathcal O(G_-) \otimes \mathfrak{gl}_N[u]$, the term $\Lcur_{(T_1^{-1}\xi T_1)_+}^{(2)}$ can be written as $\sum_{i,j,m} C_{ij,m}(T) \otimes \Lcur_{E_{ij}u^m}$. It is primitive for constant $\xi$, while for the basis elements $\xi = E_{ij}u$ of spectral degree one, one has $(T^{-1}E_{ij}uT)_+ = E_{ij}u + [E_{ij}, T^{(1)}]$, yielding
\begin{equation*}
\Delta(\Lcur_{E_{ij}u}) = \Lcur_{E_{ij}u} \otimes 1 + 1 \otimes \Lcur_{E_{ij}u} + \sum_{k=1}^N \bigl(t_{jk}^{(1)} \otimes \Lcur_{E_{ik}} - t_{ki}^{(1)} \otimes \Lcur_{E_{kj}}\bigr).
\end{equation*}

Alternatively, we use the basis $\mathsf K_\xi:=\Lcur_{(T\xi)_+}$ that corresponds to the coordinate one-forms of Sec.~\ref{sec:matrix-rat-realization}. Their $T$-dependent coefficients give additional terms in the OPEs. Writing
\begin{equation*}
(T\xi)_+=\sum_{i,j=1}^N\sum_{m\geq0}B^\xi_{ij,m}(T)E_{ij}u^m,
\end{equation*}
and abbreviating $\alpha=(i,j,m)$, $e_\alpha=E_{ij}u^m$, $B^\xi_\alpha=B^\xi_{ij,m}(T)$, and $\rho_\alpha=\rho_{e_\alpha}$, we have
\begin{equation*}
\mathsf K_\xi(z)=\sum_\alpha :B^\xi_\alpha(T(z))\Lcur_{e_\alpha}(z):.
\end{equation*}
Using \eqref{eq:GMS-OPE-dictionary}, the singular OPEs read
\begin{align*}
\mathsf{K}_\xi(z)\,T(u,w) &\sim \frac{T(u,w)\bigl(\xi(u) T(u,w)\bigr)_+ - \bigl(T(u,w)\xi(u)\bigr)_+ T(u,w)}{z-w},
\\
\mathsf{K}_\xi(z)\,\mathsf{K}_\eta(w) &\sim \frac{\chi(\mathsf{K}_\xi,\mathsf{K}_\eta)(w)}{(z-w)^2}+\frac{\bigl([\mathsf K_\xi,\mathsf K_\eta]-c(\mathsf K_\xi,\mathsf K_\eta)+\frac12\partial\chi(\mathsf K_\xi,\mathsf K_\eta)\bigr)(w)}{z-w}.
\end{align*}
The Lie algebroid bracket $[\mathsf{K}_\xi,\mathsf{K}_\eta]$ corresponds to the Poisson cotangent bracket $[\mathsf{K}_\xi,\mathsf{K}_\eta]_{\pi_{\mathrm{rat}}} = \dd\{f_\xi,f_\eta\}_{\pi_{\mathrm{rat}}}$, where $f_\xi(T)=\Res_u\Tr(\xi(u)T(u))$ and $\{f_\xi,f_\eta\}_{\pi_{\mathrm{rat}}}$ is the classical RTT bracket \eqref{eq:rat-RTT}. Applying~\eqref{eq:action-ext} to the expansion of $(T\xi)_+$, or directly computing the OPE in the $\Lcur$ basis, gives the following correction terms\footnote{These non-trivial corrections are specific to the chiral setting. In the cotangent Yangian of Abedin and Niu \cite{AbedinNiu}, the change of basis produces no such terms, and the cotangent generators satisfy the exact Koszul bracket $[\mathsf{K}_\xi,\mathsf{K}_\eta] = \dd\{f_\xi,f_\eta\}_{\pi_{\mathrm{rat}}}$.}:
\begin{equation*}
  \chi(\mathsf K_\xi,\mathsf K_\eta)
=\whkappa\bigl((T\xi)_+,(T\eta)_+\bigr)-\sum_{\alpha,\beta}\Bigl[
\rho_\alpha(B^\eta_\beta)\rho_\beta(B^\xi_\alpha)
+B^\xi_\alpha\,\rho_\beta\rho_\alpha(B^\eta_\beta)
+B^\eta_\beta\,\rho_\alpha\rho_\beta(B^\xi_\alpha)\Bigr].
\end{equation*}
\begin{equation*}
  \begin{aligned}
    c(\mathsf K_\xi,\mathsf K_\eta)
={}\frac{1}{2}&\Bigl\{\whkappa\bigl((T\xi)_+,\partial(T\eta)_+\bigr)
+\sum_{\alpha,\beta}\Bigl[
\rho_\alpha(B^\eta_\beta)\,\partial\rho_\beta(B^\xi_\alpha)\\
&\qquad+\partial\!\left(B^\eta_\beta\,\rho_\alpha\rho_\beta(B^\xi_\alpha)\right)
-2\rho_{[e_\alpha,e_\beta]}(B^\xi_\alpha)\,\partial B^\eta_\beta
\Bigr]\Bigr\}-(\xi\leftrightarrow\eta).
  \end{aligned}
\end{equation*}
All sums are finite for polynomial $\xi,\eta$.

For example, at level zero we have
\begin{align*}
\chi(\mathsf K_{E_{ij}u},\mathsf K_{E_{kl}})
&=\delta_{il}\!\left(Nt_{jk}^{(1)}-\delta_{jk}\Tr T^{(1)}\right)
+\delta_{ij}t_{lk}^{(1)}-\delta_{kl}t_{ji}^{(1)},\\
c(\mathsf K_{E_{ij}u},\mathsf K_{E_{kl}})
&=-\frac12\partial\chi(\mathsf K_{E_{ij}u},\mathsf K_{E_{kl}}).
\end{align*}

The same $T$-dependent change of basis also affects the zero-level coproduct. The $h$ component of this coproduct satisfies $h(\Lcur_{\xi})= 0 $, but does not vanish on $\mathsf{K}_{\xi}$. We have
\begin{equation*}
\Delta(\mathsf K_\xi)
=\mathsf K^{(1)}_{(T_2\xi)_+}
+\mathsf K^{(2)}_{(\xi T_1)_+}
+h\bigl((T\xi)_+\bigr),
\qquad
\epsilon(\mathsf K_\xi)=0.
\end{equation*}
Using \eqref{eq:frame-h-ext} we can write $h\bigl((T\xi)_+\bigr) =-\sum_\alpha\gamma_{12}\!\left(m^*B_\alpha^\xi,\Delta_T e_\alpha\right)$. Put $K^{(r)}=(\mathsf K_{E_{ji}u^{r-1}})_{i,j=1}^N$ and $K(u,z)=\sum_{r\geq1}K^{(r)}(z)u^{-r}$. We have
\begin{equation*}
\Delta K(u)=K_1(u)T_2(u)+T_1(u)K_2(u)-\mathcal E(u),
\qquad
\epsilon(K(u))=0,
\end{equation*}
where $\mathcal E(u)$ can be computed by \eqref{eq:ext-prod-gamma} and give, for $\alpha=(k,l,m)$,
\begin{equation*}
\mathcal E(u)=
\sum_{\alpha=(k,l,m)}u^{-m-1}E_{lk}
\left[
\bigl(\rho_\alpha T_1(u)\bigr)\partial T_2(u)
+\partial T_1(u)\bigl(\rho_{M_{T_1}(e_\alpha)}T_2(u)\bigr)
\right].
\end{equation*}

\section{Sphere conformal blocks}
\label{sec:sphere-blocks}

Let $X$ be a smooth curve and let $\mathrm{Aut}_X\to X$ be the bundle of formal local coordinate charts of $X$. For a vertex algebra $V$, a quasi-conformal structure is exactly the coordinate-change datum needed to globalize $V$. The integrated $\operatorname{Aut}\cO$-action on a quasi-conformal vertex algebra $V$ defines the associated \emph{vertex-algebra bundle}
\begin{equation*}
\mathscr{V} :=\mathrm{Aut}_X\mathbin{\times_{\operatorname{Aut}\cO}}V,
\end{equation*}
which carries a canonical flat connection $\nabla$
\cite[Secs.~6.2--6.6]{FrenkelBenZvi}. 

Choose distinct marked points $\mathbf x=(x_1,\ldots,x_n)$ with $n\geq1$, set $\Sigma:=X\setminus\{x_1,\ldots,x_n\}$, and let $M_i$ be $V$-modules whose quasi-conformal actions integrate to $\operatorname{Aut}\cO$.  We consider the Lie algebra
\begin{equation*}
U_\Sigma(\mathscr{V} ) := \frac{\Gamma(\Sigma,\mathscr{V} \otimes\Omega_X^1)}{\nabla\Gamma(\Sigma,\mathscr{V})}.
\end{equation*}
If a section representing a class in $U_\Sigma(\mathscr V)$ has local Laurent expansion $f(z_i)a dz_i$ at $x_i$, its action on the $i$th tensor factor is
\begin{equation}
\operatorname{Res}_{z_i=0}
f(z_i)Y_{M_i}(a,z_i)\,dz_i.
\label{eq:global-sec-local-action}
\end{equation}
where $Y_{M_i}$ is the module vertex operator.  This expression can be easily seen to factor through $U_\Sigma(\mathscr V)$.  Summing over the marked points gives the diagonal action on
$M_{\mathbf x}:=M_1\otimes\cdots\otimes M_n$
\cite[Thm.~9.2.6 and Def.~10.1.6]{FrenkelBenZvi}. We follow notation of \cite[Def.~10.1.6]{FrenkelBenZvi} and define the coinvariants and conformal blocks by
\begin{align*}
H_V(X,\mathbf x,M_1,\ldots,M_n)
&:=
M_{\mathbf x}/U_\Sigma(\mathscr{V} )M_{\mathbf x}.
\\
C_V(X,\mathbf x,M_1,\ldots,M_n)
&:=
\operatorname{Hom}_{U_\Sigma(\mathscr{V})}(M_{\mathbf x},\kk).
\end{align*}
Thus a block is precisely a linear functional vanishing on $U_\Sigma(\mathscr V)M_{\mathbf x}$.

In this section we specialize to the case $X=\mathbb P^1$, with one marked point placed at $0$,
and the insertion $M_1=V$. Put $\mathfrak g_{\mathrm{out}}(V) := U_{\mathbb P^1\setminus\{0\}}(\mathscr{V})$. Then we can write sphere one-point coinvariants and conformal blocks as
\begin{equation*}
H(\mathbb P^1,0,V) = V/\mathfrak g_{\mathrm{out}}(V)V, \quad C(\mathbb P^1,0,V) =\operatorname{Hom}_{\mathfrak g_{\mathrm{out}}(V)}(V,\kk).
\end{equation*}
For chiral cotangent Yangian, we will study its conformal blocks for both of the quasi-conformal structures constructed in Sec.~\ref{sec:qc}.

\subsection{Sphere one-point coinvariants}
\label{sec:1pt-coinvariants}

We first give a general description of sphere one-point coinvariants for a nonnegatively graded quasi-conformal vertex algebra. For conformal vertex algebras with finite-dimensional graded pieces, this description follows from propagation of vacua and the identification of two-point blocks with invariant bilinear forms \cite[Thm.~10.3.1 and Sec.~10.4]{FrenkelBenZvi}; the latter are classified in \cite[Thm.~1]{RoitmanInvariantForms}. We give a direct proof without assuming a conformal vector or finite-dimensional graded pieces.

\begin{proposition}
\label{prop:1pt-coinvariants}
Let $V=\bigoplus_{n\geq0}V_n$ be a quasi-conformal vertex algebra, and let $q_V:V_0\to H(\mathbb P^1,0,V)$ be inclusion followed by the $\mathfrak g_{\mathrm{out}}(V)V$ quotient.  Then $q_V$ induces a canonical isomorphism
\begin{equation*}
\overline q_V:
V_0/\mathscr L_1V_1
\xrightarrow{\sim}
H(\mathbb P^1,0,V)
\end{equation*}
\end{proposition}

\begin{proof}
Put $N:=\mathfrak g_{\mathrm{out}}(V)V$.  It suffices to show that $N$ is graded, contains every $V_r$ with $r>0$, and satisfies $N\cap V_0=\mathscr L_1V_1$.

We first describe $\mathfrak g_{\mathrm{out}}(V)$. The coordinate $w=z^{-1}$ trivializes $\mathscr V$ on $\Sigma=\mathbb P^1\setminus\{0\}=\mathbb A^1_w$, so $\Gamma(\Sigma,\mathscr V\otimes\Omega_\Sigma^1)
=V\otimes\kk[w]\,dw$. Thus the classes of $a\,w^m\,dw$, with $a\in V_d$ and $m\geq0$, span the outgoing Lie algebra $\mathfrak{g}_{\mathrm{out}}(V)$. Under coordinate inversion, this section has local expansion
\begin{equation*}
(-z^2)^{\mathscr L_0}e^{-z\mathscr L_1}a z^{-m}(-z^{-2})\dd z = (-1)^{d+1}\sum_{j=0}^{d} \frac{(\mathscr L_1^ja)z^{2d-m-j-2}}{j!}\dd z;
\end{equation*}
see \cite[Eq.~(10.4.5) and p.~179, Eq.~(10.4.8)]{FrenkelBenZvi}. By the residue action \eqref{eq:global-sec-local-action}, the section therefore acts as $-a_m^o$, where
\begin{equation}
a_m^o:=(-1)^d\sum_{j=0}^{d}
\frac{(\mathscr L_1^ja)_{(2d-m-j-2)}}{j!},
\qquad a\in V_d,\quad m\geq0,
\label{eq:1pt-outgoing-op}
\end{equation}
is the mode of the opposite vertex operator as in \cite[Eq.~(2.57)]{HuangLepowskyZhangI}.
Consequently,
\begin{equation*}
N=\operatorname{span}_{\kk}
\{a_m^ov:a,v\text{ homogeneous},\ m\geq0\},
\qquad
a_m^oV_s\subseteq V_{s+m+1-d}.
\end{equation*}
The weight formula follows termwise from
\eqref{eq:qc-covariance}; in particular, $N$ is graded.

For $a\in V_r$ with $r>0$, the definition gives
\begin{equation*}
a_{(-1)}
=(-1)^r\sum_{j=0}^{r}\frac{1}{j!}
(\mathscr L_1^ja)_{2r-j-1}^o.
\end{equation*}
Indeed, after expanding the right-hand side, the coefficient of $(\mathscr L_1^ta)_{(-t-1)}$ is $1$ for $t=0$ and $0$ otherwise. Applying the above expression for $a_{(-1)}$ to the vacuum therefore proves that $a\in N$, and hence $V_r\subseteq N$ for $r>0$.

It remains to compute $N\cap V_0$.  The commutator identity~\eqref{eq:qc-covariance} gives
\begin{equation*}
[\mathscr L_1,b_{(k)}] =(\mathscr L_1b)_{(k)}+(2\delta-k-2)b_{(k+1)}, \qquad b\in V_\delta.
\end{equation*}
Substituting this in \eqref{eq:1pt-outgoing-op} yields
\begin{equation*}
[\mathscr L_1,a_{m+1}^o]=(m+1)a_m^o.
\end{equation*}
If $v\in V_s$ and $a_m^ov$ has weight zero, then $a_{m+1}^ov\in V_1$, so
\begin{equation*}
a_m^ov
=\frac{\mathscr L_1\bigl(a_{m+1}^ov\bigr)}{m+1}
-\frac{a_{m+1}^o\mathscr L_1v}{m+1}
\equiv-\frac{a_{m+1}^o\mathscr L_1v}{m+1}
\pmod{\mathscr L_1V_1}.
\end{equation*}
Induction on $s$, starting with $\mathscr L_1V_0=0$, shows that every weight-zero generator of $N$ belongs to $\mathscr L_1V_1$. Conversely, for $a\in V_1$ the vacuum identities give $a_0^o\mathbf1=-\mathscr L_1a$, so $\mathscr L_1V_1\subseteq N$. Thus $N\cap V_0=\mathscr L_1V_1$. Therefore we proved $\overline q_V$ is an isomorphism.
\end{proof}

Let $\rho:\mathfrak a\to\Der(A)$ be an action as in Subsection~\ref{sec:action-levels-3form}, and fix $\lambda\in Z^1(\mathfrak a,A)$. For $\theta\in\kk$, let $A_\theta$ denote $A$ with the right $T=A\otimes\mathfrak a$-connection
$$
f\mathbin{\cdot}(a\xi) :=-\rho_\xi(af)+\theta af\lambda_\xi,\qquad f,a\in A,\quad\xi\in\mathfrak a,
$$
where the cocycle identity of $\lambda$ implies flatness. We write $A_\theta\mathbin{\cdot}\mathfrak a$:
$$
A_\theta\mathbin{\cdot}\mathfrak a
=\operatorname{span}_{\kk}
\{\rho_\xi(f)-\theta f\lambda_\xi:\xi\in\mathfrak a,\ f\in A\}.
$$

We use $H(...)_\theta$ and $C(..)_\theta$ for the coinvariants and blocks under the quasi-conformal structure $\{\mathscr{L}_n^{\theta}\}$, and omit $\theta$ when $\theta=0$.

\begin{proposition}
\label{prop:1pt-action-coinv}
For every $\kappa\in\operatorname{Hom}_{\kk}(\Sym^2_{\kk}\mathfrak a,\kk)^{\mathfrak a}$ and $\mathcal H\in Z^3_{\mathrm{DR}}(\cT(A, \mathfrak a))$, put $V_{\kappa,\mathcal H}:=V_{\kappa,\mathcal H}(A, \mathfrak a)$.  We have canonical identifications
$$
H\bigl(\mathbb P^1,0,V_{\kappa,\mathcal H}\bigr)_\theta
\cong
A_\theta/(A_\theta\mathbin{\cdot}\mathfrak a)
$$
and
$$
C\bigl(\mathbb P^1,0,V_{\kappa,\mathcal H}\bigr)_\theta
\cong
\left\{
\Lambda\in A^\vee:
\Lambda(\rho_\xi f-\theta f\lambda_\xi)=0
\text{ for all }\xi\in\mathfrak a,\ f\in A
\right\}.
$$
In particular, for fixed $\lambda$ and $\theta$, these vector spaces are independent of $\kappa$ and
$\mathcal H$.
\end{proposition}

\begin{proof}
By Sec.~\ref{sec:valg-truncation},
\begin{equation*}
(V_{\kappa,\mathcal H})_0=A,
\qquad
(V_{\kappa,\mathcal H})_1
=(A\otimes\mathfrak a)\oplus\Omega_A^1,
\end{equation*}
By Cor.~\ref{prop:Yangian-qc} and Prop.~\ref{prop:twisted-qc},
$$
\mathscr L_1^\theta\left(\sum_i a_i\xi_i+\omega\right) =\sum_i\left(\rho_{\xi_i}(a_i)-\theta a_i\lambda_{\xi_i}\right).
$$
Consequently,
$\mathscr L_1^\theta(V_{\kappa,\mathcal H})_1=A_\theta\mathbin{\cdot}\mathfrak a$.
The two assertions now follow from Prop.~\ref{prop:1pt-coinvariants}.
\end{proof}

For an algebraically integrated matched pair $\mathfrak p=\mathfrak p_+\bowtie\mathfrak p_-$ with $\mathsf P_-$ affine, the proposition applies to $A=\cO(\mathsf P_-)$, $\mathfrak a=\mathfrak p_+$. For the cotangent Yangian action of Lem.~\ref{lem:rat-factor-poly-closure}, take $A=\cO(G_-)$, $\mathfrak a=\dplus=\g[u]$, and the cocycle $\lambda$ of Sec.~\ref{sec:twisted-qc}. We use the abbreviations
$$
\mathscr Q_{\whkappa,\theta}
:=
H(\mathbb P^1,0,\Ych_{\whkappa}(T^*\g))_\theta,
\qquad
\mathscr B_{\whkappa,\theta}
:=
C(\mathbb P^1,0,\Ych_{\whkappa}(T^*\g))_\theta
=(\mathscr Q_{\whkappa,\theta})^\vee.
$$

\begin{proposition}
\label{prop:twisted-1pt-affine-center}
For every level $\widehat\kappa$ and every $\theta\in\kk$, there is a canonical identification
$$
\mathscr B_{\whkappa,\theta}
\cong\widehat{Z}\bigl(V_{-\theta\kappa}(\g)\bigr),
$$
where $V_{-\theta\kappa}(\g)$ is the affine vertex algebra at level $-\theta\kappa$, and $\widehat Z$ denotes the completion of the center with respect to the standard nonnegative grading of the affine vacuum module.
\end{proposition}

\begin{proof}
By Prop.~\ref{prop:1pt-action-coinv}, the coinvariants are the quotient of $A$ by the span of $\rho_\xi f-\theta f\lambda_\xi$. We analyze this quotient by passing to its graded dual. Using the spectral grading, PBW identifies $A=\mathcal O(G_-)$ with the graded restricted dual of $U(u^{-1}\mathfrak g[u^{-1}])$, giving a nondegenerate pairing
$$
\langle-,-\rangle:
A\otimes U(u^{-1}\g[u^{-1}])\longrightarrow\kk.
$$
Write $\xi\mathbin{\cdot}_{-\theta\kappa}Y$ for the $\g[u]$-action on $U(u^{-1}\g[u^{-1}])$ under the identification $V_{-\theta\kappa}(\g)\cong U(u^{-1}\g[u^{-1}])$, and $\xi\mathbin{\cdot}Y$ for the action at $\theta=0$.

For $X,Z\in\g[u,u^{-1}]$, their actions on $V_{-\theta\kappa}(\g)$ satisfy
$$
XZ-ZX = [X,Z]-\theta\Res_u\kappa(\Tu X,Z)\id,
$$
where $[X,Z]$ is the bracket in $\g[u,u^{-1}]$. For $B\in u^{-1}\g[u^{-1}]$, put $g=\exp B$. The pairing $\langle-,-\rangle$ extends to these exponentials and satisfies $\langle f,e^B\rangle=f(g)$. For $\xi \in \g[u]$, commuting $\xi$ past the powers of $B$ as operators on $V_{-\theta\kappa}(\g)$ gives
$$
\xi e^B = e^B\Ad_{g^{-1}}\xi + \theta e^B \sum_{n\geq0}\frac{(-1)^n}{(n+1)!} \Res_u\kappa\bigl(\Tu B,(\ad_B)^n\xi\bigr).
$$
Here $\ad_B$ denotes the bracket with $B$ in $\g[u,u^{-1}]$.
By invariance of $\kappa$,
$$
\sum_{n\geq0}\frac{(-1)^n}{(n+1)!} \Res_u\kappa\bigl(\Tu B,(\ad_B)^n\xi\bigr) = \Res_u\kappa\bigl(\xi,(\Tu g)g^{-1}\bigr) =\lambda_\xi(g),
$$
Applying the resulting identity to the vacuum of $V_{-\theta\kappa}(\g)$ gives:
$$
\xi\mathbin{\cdot}_{-\theta\kappa}e^B = e^B (\Ad_{g^{-1}}\xi)_- +\theta\lambda_\xi(g)e^B.
$$
The pairing satisfies $ \langle f,e^B(\Ad_{g^{-1}}\xi)_-\rangle = -(\rho_\xi f)(g)$, where we used the defining formula $\rho_\xi(g)=-(L_g)_*(\Ad_{g^{-1}}\xi)_-$ for the dressing action. Consequently,
$$
\langle f,\xi\mathbin{\cdot}_{-\theta\kappa}e^B\rangle = -\langle\rho_\xi f-\theta f\lambda_\xi,e^B\rangle.
$$
Replacing $B$ by $tB$ and compare the power of $t$ proves the identity for arbitary power of $B$, and thus arbitrary symmetrized products of $y_i\in u^{-1}\g[u^{-1}]$. These span the enveloping algebra $U(u^{-1}\g[u^{-1}])$, thus proves
\begin{equation}
\langle\rho_\xi f-\theta f\lambda_\xi,Y\rangle
=
-\langle f,\xi\mathbin{\cdot}_{-\theta\kappa}Y\rangle.
\label{eq:rat-dressing-transpose}
\end{equation}
Consequently, the restricted dual of $\mathscr Q_{\whkappa,\theta}$ identifies with $ V_{-\theta\kappa}(\g)^{\g[u]} = Z\bigl(V_{-\theta\kappa}(\g)\bigr)$. Taking the product over spectral degrees identifies the block space $\mathscr B_{\whkappa,\theta}$ with the completed center $\widehat{Z}\bigl(V_{-\theta\kappa}(\g)\bigr)$.
\end{proof}

Two special cases of Prop.~\ref{prop:twisted-1pt-affine-center} are interesting. Suppose $\kk = \bbC$. Suppose first that $\g$ is semisimple and $ -\theta\kappa=-\frac12\kappa_{\mathrm{ad}}$, where $\kappa_{\mathrm{ad}}$ is the Killing form. The affine vertex algebra $V_{-\frac12\kappa_{\mathrm{ad}}}(\g)$ is then at critical level. The Feigin--Frenkel theorem \cite{FeiginFrenkel1992} (see also \cite[Thm.~4.3.2]{FrenkelLoopGroups}) identifies the center with functions on opers on the formal disc for the Langlands-dual group. More explicitly, let $d_1,\ldots,d_\ell$ be the exponents of $\g$, where $\ell=\operatorname{rk}\g$. There are homogeneous algebraically independent generators $S_{i,n}$ such that $
Z\bigl(V_{-\theta\kappa}(\g)\bigr) = \bbC[S_{i,n}:1\leq i\leq\ell,\ n\geq0]$. Thus $\mathscr B_{\whkappa,\theta}$ identifies with the graded completion of this polynomial algebra.

Another interesting case is with the specialization $\theta=0$ and with $\g$ be reductive. Let $\mathfrak z:=\mathfrak z(\g)$ be its center, and define the central subgroup
$$
Z_-:=\exp\bigl(u^{-1}\mathfrak z[[u^{-1}]]\bigr)\subset G_-.
$$
The noncritical-center result \cite[Prop.~3.3.3]{FrenkelLoopGroups}, applied to the simple factors of $\g$, together with the trivial action of $\mathfrak z[u]$ at level zero, gives
$$
Z\bigl(V_0(\g)\bigr) = U\bigl(u^{-1}\mathfrak z[u^{-1}]\bigr).
$$
Under the PBW pairing, the inclusion of this subspace is the transpose of restriction of functions from $G_-$ to $Z_-$. We therefore have canonical identifications
$$
\mathscr Q_{\whkappa,0}\cong\cO(Z_-),
\qquad
\mathscr B_{\whkappa,0}\cong\cO(Z_-)^\vee.
$$
In particular, if $\g$ is semisimple, the $\theta = 0$ block space is $\bbC$, generated by evaluation at the identity.

Finally, we provide a geometric description of the coinvariants in terms of the Poisson geometry on $G_-$. As discussed before, the Poisson tensor $\pi_{\mathrm{rat}}$ makes $\Omega^1_{G_-}$ a Lie algebroid. Via the isomorphism $\jmath$, we have a right module $A_\theta$ over this Lie algebroid, equipped with the flat right connection by
$$
f\mathbin{\cdot}\bigl(a\jmath(\xi)\bigr) = -\rho_\xi(af)+\theta af\lambda_\xi,
\qquad
f,a\in A,\quad\xi\in\dplus.
$$
We write $H_\bullet^{\mathrm{LR}}(\Omega^1_{G_-};A_\theta)$ for the Lie algebroid homology (also called Lie--Rinehart homology) with this right-module structure \cite{HuebschmannBV}.

\begin{corollary}
\label{cor:1pt-coinv-LR-homology}
For every level and every $\theta\in\kk$, there is a canonical identification
$$
\mathscr Q_{\whkappa,\theta}
\cong
H_0^{\mathrm{LR}}(\Omega^1_{G_-};A_\theta).
$$
\end{corollary}

\begin{proof}
For $\xi\in\dplus$, one has $\pi_{\mathrm{rat}}^\sharp(\jmath(\xi))=\rho_\xi$.  Thus the degree-one differential satisfies
$$
d_1\bigl(a\otimes\jmath(\xi)\bigr)
=a\mathbin{\cdot}\jmath(\xi)
=-\rho_\xi(a)+\theta a\lambda_\xi.
$$
Its image is therefore $A_\theta\mathbin{\cdot}\dplus$, and hence
$$
H_0^{\mathrm{LR}}(\Omega^1_{G_-};A_\theta)
\cong A_\theta/(A_\theta\mathbin{\cdot}\dplus)
\cong\mathscr Q_{\whkappa,\theta}.
$$
\end{proof}

\subsection{Convolution products on one-point blocks}
\label{sec:block-conv}
In this subsection, we examine the algebraic consequences of the coproduct for coinvariants and conformal blocks. We first show that sphere one-point coinvariants commute with tensor products. Consequently, any quasi-conformally equivariant coproduct descends to a coalgebra structure on the coinvariants, inducing an associative convolution product on the space of conformal blocks.

\begin{lemma}
\label{lem:block-ext-prod-descent}
Let $V,W$ be quasi-conformal vertex algebras, and equip $V\otimes W$ with the diagonal quasi-conformal structure. The quotient maps induce a canonical isomorphism
\begin{equation*}
H(\mathbb P^1,0,V\otimes W)
\cong H(\mathbb P^1,0,V)\otimes H(\mathbb P^1,0,W).
\end{equation*}
\end{lemma}

\begin{proof}
Write $N_V:=\mathfrak g_{\mathrm{out}}(V)V$, and similarly for $W$ and $V\otimes W$. For homogeneous $a\in V_d$, extend the opposite modes in~\eqref{eq:1pt-outgoing-op} to all integer indices by
\begin{equation*}
a_r^o:=(-1)^d\sum_{j\geq0}
\frac{(\mathscr L_1^ja)_{(2d-r-j-2)}}{j!},
\qquad r\in\mathbb Z.
\end{equation*}
The sum is finite by local nilpotence of $\mathscr L_1$. Since the $\mathscr{L}_1$ action on $V\otimes W$ is diagonal, for homogeneous $a\in V$, $b\in W$ one has
\begin{equation*}
(a\otimes b)_m^o(v\otimes w)
=\sum_{\substack{p,q\in\mathbb Z\\p+q=m-1}}
a_p^ov\otimes b_q^ow,
\qquad m\geq0.
\end{equation*}
The vertex-algebra axiom: $c_{(n)}v=0$ for sufficiently large $n$ (with $c,v$ fixed) implies $a_p^ov=0$ for sufficiently negative $p$, and similarly for $b_q^ow$, so the sum is finite. Although either component index may be negative, $p+q=m-1\geq-1$ implies $p\geq0$ or $q\geq0$. Hence $N_{V\otimes W}\subseteq N_V\otimes W+V\otimes N_W$. For the reverse inclusion, the vacuum identities give $\mathbf1_{-1}^o=\id$ and $\mathbf1_r^o=0$ for $r\ne-1$; taking $b=\mathbf1$ or $a=\mathbf1$ in the tensor-product identity therefore gives
\begin{equation*}
N_{V\otimes W}=N_V\otimes W+V\otimes N_W,
\end{equation*}
which gives the asserted quotient isomorphism.
\end{proof}

\begin{proposition}
\label{prop:block-ord-descent}
Let $(V,\Delta,\epsilon)$ be a vertex bialgebra equipped with a quasi-conformal structure such that
\begin{equation*}
\Delta\mathscr L_n
=(\mathscr L_n\otimes\id+\id\otimes\mathscr L_n)\Delta,
\qquad
\epsilon\mathscr L_n=0,
\qquad n\geq-1,
\end{equation*}
Then $\Delta$ and $\epsilon$ induce a counital coalgebra structure on $H(\mathbb P^1,0,V)$. Its full dual $C(\mathbb P^1,0,V)$ is an associative unital algebra under
\begin{equation*}
(\lambda\star\mu)(v)
=(\lambda\otimes\mu)\Delta(v),
\qquad \mathbf1=\epsilon,
\end{equation*}
where blocks are viewed as functionals on $V$ annihilating $\mathfrak g_{\mathrm{out}}(V)V$.
\end{proposition}

\begin{proof}
Put $N:=\mathfrak g_{\mathrm{out}}(V)V$. Vertex bialgebra axiom and Lem.~\ref{lem:block-ext-prod-descent} give
\begin{equation*}
\Delta(a_m^ov)
=(\Delta(a))_m^o\Delta(v)
\in N\otimes V+V\otimes N,
\qquad m\geq0.
\end{equation*}
Since the outgoing operators on $\kk$ vanish, we obtain the coideal conditions
\begin{equation*}
\Delta(N)\subseteq N\otimes V+V\otimes N,
\qquad \epsilon(N)=0.
\end{equation*}
Thus $V/N$ inherits the coproduct and counit. Dualizing gives the associative product and unit on $C(\mathbb P^1,0,V)$.
\end{proof}

More generally, the same argument associates to any vertex-algebra morphism $f:U\to V\otimes W$ that is compatible with the quasi-conformal structures a bilinear map
\begin{equation*}
C(\mathbb P^1,0,V)\otimes C(\mathbb P^1,0,W) \longrightarrow C(\mathbb P^1,0,U), \qquad \lambda\otimes\mu\longmapsto(\lambda\otimes\mu)\circ f.
\end{equation*}

Assume $\kk=\bbC$. We now apply Prop.~\ref{prop:block-ord-descent} to the twisted standard quasi-conformal structures on $\Ych(T^*\g)$, using Thm.~\ref{thm:zero-coprod-qc}. Thus each $\mathscr B_{0,\theta}$ is an associative unital algebra under convolution.

\begin{corollary}
\label{prop:ord-block-conv}
For every $\theta\in\kk$, the identification of Prop.~\ref{prop:twisted-1pt-affine-center} gives an isomorphism of commutative algebras
$$
\mathscr B_{0,\theta}
\cong
\widehat Z\bigl(V_{-\theta\kappa}(\g)\bigr),
$$
where the center carries its product induced from the $(-1)$-vertex product.
\end{corollary}

\begin{proof}
Regard $v,w\in Z(V_{-\theta\kappa}(\g))$ as elements of $U(u^{-1}\g[u^{-1}])$. Since $\Delta|_A=m^*$, the pairing satisfies
$$
\langle m^*f,v\otimes w\rangle = \langle f,vw\rangle,
\qquad f\in A.
$$
Thus convolution corresponds to multiplication in the enveloping algebra. Since $w$ is central, $w_{(-1)}$ commutes with every mode and for any $v$ we have $ v_{(-1)}w = w_{(-1)}v = vw$. Hence convolution agrees with the product on the center, and the counit corresponds to the vacuum.
\end{proof}

For $\theta=0$ and reductive $\g$, this recovers the commutative convolution algebra $\mathscr B_0\cong\cO(Z_-)^\vee$. Under $\mathscr Q_0\cong\cO(Z_-)$, the descended coproduct and
counit are multiplication on $Z_-$ and evaluation at its identity.

\subsection{The algebra of mini-twistor blocks}
\label{sec:mt-blocks}

We abbreviate the level-zero chiral cotangent Yangian by $\mathcal Y:=\Ych(T^*\g)$. Under the cocycle-twisted mini-twistor quasi-conformal structure $\{\mathscr L_n^{\mathrm{MT},\theta}\}$, we study the algebra of sphere one-point blocks. Write
$$
\mathscr B_\theta^{\mathrm{MT}} :=C(\mathbb P^1,0,\mathcal Y)_\theta =\left(\mathcal Y/\mathfrak g_{\mathrm{out}}(\mathcal Y)\mathcal Y\right)^\vee,
$$
where the outgoing algebra is defined using $\{\mathscr L_n^{\mathrm{MT},\theta}\}$. We omit $\theta$ when
$\theta=0$. By Thm.~\ref{thm:zero-coprod-qc} and Prop.~\ref{prop:block-ord-descent}, the coproduct defines an associative unital product on $\mathscr B_\theta^{\mathrm{MT}}$ by convolution.

Recall that the minitwistor space $\mathbb{MT}$ is defined to be $\operatorname{Tot}_{\mathbb P^1}\cO(2)$. We denote the projection map $p:\mathbb{MT} \to \mathbb P^1$, and $\Sigma_0 \cong \mathbb{P}^1$ the zero section. Let $U_0,U_\infty$ be the standard affine charts of $\mathbb P^1$, we define $\widetilde{U}_i=p^{-1}(U_i)$ for $i=0,\infty$ and $\widetilde{U}_{01}=\widetilde{U}_0\cap \widetilde{U}_\infty$. Write $\underline{\g}=\g\otimes\cO_{\mathrm{MT}}$, we denote sections of $\underline{\g}$ on the overlap and on the two charts as
\begin{equation*}
\mathfrak D_+=\Gamma(\widetilde U_{01},\underline\g),
\qquad
\mathfrak D_+^i=\Gamma(\widetilde U_i,\underline\g),
\quad i=0,\infty.
\end{equation*}
Let $ \mathcal P_{\Sigma_0} :=\cO_{\mathrm{MT}}(*\Sigma_0)/\cO_{\mathrm{MT}} $ be the sheaf of principal parts along $\Sigma_0$. Then define
\begin{equation*}
\mathfrak D_- :=\Gamma(\widetilde U_{01},\underline\g\otimes\mathcal P_{\Sigma_0}),
\qquad
\mathfrak D_-^i :=\Gamma(\widetilde U_i,\underline\g\otimes\mathcal P_{\Sigma_0}),
\quad i=0,\infty,
\end{equation*}
We can choose coordinates $(t,u)$ and $(z,w)$ on $\widetilde{U}_0$ and $\widetilde{U}_\infty$, related by $z=t^{-1}, w=ut^{-2}$. In these coordinates we can identify  
\begin{equation*}
\begin{aligned}
&\mathfrak D_+=\g[u,t,t^{-1}], \quad  \mathfrak D_+^{0}=\g[u,t], \quad \mathfrak D_+^{\infty}=\g[w,z]\\
&\mathfrak D_-=u^{-1}\g[u^{-1},t,t^{-1}],  \quad \mathfrak D_-^{0}=u^{-1}\g[u^{-1},t], \quad \mathfrak D_-^{\infty}=w^{-1}\g[w^{-1},z].
\end{aligned}
\end{equation*}

We also define $\mathfrak D :=\Gamma(\widetilde U_{01}\setminus\Sigma_0,\underline\g) =\g[u,u^{-1},t,t^{-1}] =\mathfrak D_+\oplus\mathfrak D_-$.
Give $\mathfrak D$ the bidegree
\begin{equation*}
\operatorname{bideg}(xu^mt^k):=(m,-k),
\qquad m,k\in\mathbb Z,
\end{equation*}
which extend additively to enveloping algebras. A nonzero homogeneous element of $\mathfrak D_\pm$ of bidegree $(a,b)$ belongs to $\mathfrak D_\pm^0$ exactly when $b\leq0$, and to $\mathfrak D_\pm^\infty$ exactly when $b\geq2a$, as $xu^at^{-b}=xw^az^{b-2a}$. 

When writing these degrees in the subscript, we let $(-)_{[q]}$ denote the spectral degree. This is compatible with our previous notation for the vertex algebra, where we denote
\begin{equation*}
\mathcal Y_{[q],h}:=\mathcal Y_{[q]}\cap\mathcal Y_h,
\qquad
\mathcal Y=\bigoplus_{q,h\geq0}\mathcal Y_{[q],h},
\end{equation*}
For a nonzero homogeneous $v\in\mathcal Y_{[q],h}$, one has $\operatorname{wt}_{\mathrm{MT}}(v)=h-q$. The identity $\lambda_{xu^m}\in A_{[m]}$ implies that $\g_{\mathrm{out}}(\mathcal Y)$ are homogeneous in spectral degree. Thus a block supported on $\mathcal Y_{[q]}$ has spectral degree $-q$, and we define
$$
\mathscr B_{\theta,[-q]}^{\mathrm{MT}}
:=\left\{
\lambda\in\mathscr B_\theta^{\mathrm{MT}}:
\lambda(\mathcal Y_{[q']})=0
\text{ for }q'\ne q
\right\},
\qquad q\geq0.
$$
The full block space is therefore
$$
\mathscr B_\theta^{\mathrm{MT}}
=\prod_{q\geq0}\mathscr B_{\theta,[-q]}^{\mathrm{MT}}.
$$
We give it the spectral topology, whose neighborhoods of zero are the subspaces $\prod_{q\geq N}\mathscr B_{\theta,[-q]}^{\mathrm{MT}}$. Since $\Delta$ preserves spectral degree, the convolution product satisfies
$$
(\lambda\star\mu)_{[-q]}
=\sum_{i+j=q}\lambda_{[-i]}\star\mu_{[-j]}.
$$
Each component is a finite sum, so the product is continuous in the spectral topology.

The decomposition $\mathfrak D=\mathfrak D_+\oplus\mathfrak D_-$ defines a matched pair of Lie algebras. Write $U(-)$ for enveloping algebras, with primitive coproduct $\Delta$, counit $\epsilon$, and antipode $S$. The two PBW projections $p_\pm:U(\mathfrak D)\to U(\mathfrak D_\pm)$ are characterized by $p_+(XY)=X\epsilon(Y)$ and $p_-(YX)=Y\epsilon(X)$ for $X\in U(\mathfrak D_+)$ and $Y\in U(\mathfrak D_-)$. For $U(\mathfrak D^i_-)$, we write $\widehat{(-)}$ for completion in spectral degree: $\widehat{U(\mathfrak D_-^i)} = \prod_{d\geq0}U(\mathfrak D_-^i)_{[-d]}$, and similarly for $\mathfrak D_-$.

We use the construction of Koppinen smash product for a matched pair of Lie algebras, which is reviewed in Appendix~\ref{app:kop} (Prop.~\ref{prop:kop-matched-pair}).
\begin{definition}
\label{def:alg-K}
For the matched pair $\mathfrak D = \mathfrak D_+ \oplus \mathfrak D_-$, the algebra $\mathcal K$ is the completion
\begin{equation*}
\mathcal K := \prod_{q\geq0}\mathcal K_{[-q]}
= \operatorname{Hom}_{\kk}\bigl(U(\mathfrak D_+),\widehat{U(\mathfrak D_-)}\bigr),
\qquad
\mathcal K_{[-q]} := \bigoplus_{p+d=q} \operatorname{Hom}_{\kk}\bigl(U(\mathfrak D_+)_{[p]},U(\mathfrak D_-)_{[-d]}\bigr),
\end{equation*}
equipped with the product $(F\diamond G)_{[-q]} := \sum_{i+j=q}F_{[-i]}\diamond G_{[-j]}$, where for homogeneous $F,G$ and $X\in U(\mathfrak D_+)$,
\begin{equation}
(F\diamond G)(X) := \sum_{(X)} Y_{(2)} G\bigl(X_{(2)}\triangleleft Y_{(1)}\bigr),
\qquad
Y=F(X_{(1)}),\quad X\triangleleft Y:=p_+(S(Y)X).
\label{eq:mt-kop-prod}
\end{equation}
\end{definition}

Since $U(\mathfrak D_+)_{[p]}\triangleleft U(\mathfrak D_-)_{[-e]}\subseteq U(\mathfrak D_+)_{[p-e]}$, the product is graded: $\mathcal K_{[-q]}\diamond\mathcal K_{[-q']}\subseteq\mathcal K_{[-(q+q')]}$, so each $(F\diamond G)_{[-q]}$ is a finite sum, making $\mathcal K$ an associative topological algebra with unit $X\mapsto\epsilon(X)1$.

The opposite PBW ordering defines for each $F\in\mathcal K$ the reordered map $F^R\in\mathcal K$,
\begin{equation}
F^R(X) := \sum_{(X)}p_-\bigl(S(X_{(1)})F(X_{(2)})\bigr), \qquad X\in U(\mathfrak D_+).
\label{eq:mt-reorder-map}
\end{equation}
and a reordered Koppinen product $(F\diamond_R G)(X) := \sum_{(X)} Y_{(2)} G\bigl(X_{(2)}\triangleleft_R Y_{(1)}\bigr)$, where $X\triangleleft_R Y:=p_+^R(XY)$. Here $p_+^R:U(\mathfrak D)\to U(\mathfrak D_+)$ is the opposite PBW projection, characterized by $p_+^R(YX)=\epsilon(Y)X$. By Lem.~\ref{lem:kop-opp-PBW}, the assignment $F\mapsto F^R$ is an algebra isomorphism preserving spectral degree with $(F\diamond G)^R = F^R\diamond_R G^R$ and inverse $F(X)=\sum_{(X)}p_-\bigl(X_{(1)}F^R(X_{(2)})\bigr)$.

\begin{propositiondefinition}
\label{prop:mt-1chart-algs}
Set
\begin{equation}
\begin{aligned}
\mathcal K(U_\infty)
&:=\{F\in\mathcal K:
\operatorname{im}F\subseteq\widehat{U(\mathfrak D_-^\infty)}\},\\
\mathcal K(U_0)
&:=\{F\in\mathcal K:
\operatorname{im}F^R\subseteq\widehat{U(\mathfrak D_-^0)}\},\\
\mathcal K(U_0,U_\infty)
&:=\mathcal K(U_0)\cap\mathcal K(U_\infty).
\end{aligned}
\label{eq:mt-1chart-algs}
\end{equation}
These are closed unital subalgebras of $\mathcal K$.
\end{propositiondefinition}

\begin{proof}
If $F,G\in\mathcal K(U_\infty)$, then in \eqref{eq:mt-kop-prod}, $Y=F(X_{(1)})\in\widehat{U(\mathfrak D_-^\infty)}$ implies $Y_{(2)}\in\widehat{U(\mathfrak D_-^\infty)}$, while $G(\dots)\in\operatorname{im}G\subseteq\widehat{U(\mathfrak D_-^\infty)}$ regardless of input. We have $Y_{(2)} G(X_{(2)}\triangleleft Y_{(1)}) \in \widehat{U(\mathfrak D_-^\infty)}$ so $F\diamond G\in\mathcal K(U_\infty)$.  Similarly, for $F,G\in\mathcal K(U_0)$, applying the same reasoning to $(F\diamond G)^R = F^R\diamond_R G^R$ shows that $\operatorname{im}(F\diamond G)^R\subseteq\widehat{U(\mathfrak D_-^0)}$, so $F\diamond G\in\mathcal K(U_0)$. The unit $\mathbf 1(X)=\epsilon(X)1$ satisfies $\mathbf 1^R=\mathbf 1$ with image $\kk \cdot 1$, and both image conditions are degree-wise closed in the spectral topology.
\end{proof}

Using the pairing $A\otimes U(\dminus)\to\kk$ from Prop.~\ref{prop:twisted-1pt-affine-center}, define, for $P\in A$, the linear functional $\sigma_n(P)$ on $U(\mathfrak D_-)$ by
$$
\sigma_n(P)(y_1t^{r_1}\cdots y_st^{r_s}) := \delta_{n+\sum_jr_j,-1}\langle P,y_1\cdots y_s\rangle, \quad y_j\in\dminus.
$$
For a linear functional $f\in U(\mathfrak D_-)^*$, set $J_fY:=\sum_{(Y)}f(Y_{(1)})Y_{(2)}$. For homogeneous $f$, this extends continuously to the spectral completion of $U(\mathfrak D_-)$. We write $(J_fF)(X):=J_f(F(X))$.

For $\xi_0\in\mathfrak D_+^0$ and $\xi_\infty=\xi t^k\in\mathfrak D_+^\infty$, where $\xi\in\g u^m$, $m\geq0$ and $k\leq-2m$, define
\begin{equation}
\begin{aligned}
(\xi_0\cdot_0F)(X) &:=F(X\xi_0),\\
(\xi_\infty\cdot_{\infty,\theta}F)(X) &:=-F(\xi_\infty X) +p_-\bigl(\xi_\infty F(X)\bigr) +\theta J_{\sigma_{k-1}(\lambda_\xi)}F(X).
\end{aligned}
\label{eq:mt-chart-actions}
\end{equation}
Extend the second formula linearly in $\xi_\infty$, and write $\cdot_\infty:=\cdot_{\infty,0}$.

\begin{propositiondefinition}
\label{prop:mt-kop-subspace}
The assignments \eqref{eq:mt-chart-actions} define commuting continuous actions of $\mathfrak D_+^0$ and
$\mathfrak D_+^\infty$ on $\mathcal K(U_0,U_\infty)$. The invariant subspace for these actions,
$$
\mathcal R_\theta^{\mathrm{MT}} :=\mathcal K(U_0,U_\infty)^{ \mathfrak D_+^0\oplus\mathfrak D_+^\infty},
$$
is a closed subspace of $\mathcal K$.
\end{propositiondefinition}

\begin{proof}
It is easy to check that $\cdot_0$ and $\cdot_\infty$ define Lie algebra actions. To verify the same for $\cdot_{\infty,\theta}$, let $Y=(y_1t^{r_1})\cdots(y_st^{r_s})$. Applying \eqref{eq:rat-dressing-transpose} with $\theta=0$ gives
$$
\sigma_n(P)\bigl(p_-(\xi t^kY)\bigr) =-\sigma_{n+k}(\rho_\xi P)(Y).
$$
Write $\xi_\infty=\xi t^k$ and $\xi'_\infty=\xi't^{k'}$. Using the identity $[p_-(\xi_\infty\,\cdot),J_f]=-J_{f\circ p_-(\xi_\infty\,\cdot)}$ together with the cocycle identity, we obtain
$$
[\xi_\infty\cdot_\infty-,J_{\sigma_{k'-1}(\lambda_{\xi'})}] -[\xi'_\infty\cdot_\infty-,J_{\sigma_{k-1}(\lambda_\xi)}] =J_{\sigma_{k+k'-1}(\lambda_{[\xi,\xi']})}.
$$
Together with $[J_f,J_g]=0$, this shows that $\cdot_{\infty,\theta}$ defines a Lie algebra action. The two actions commute because applying the two actions in either order yields the same expression $-F(\xi_\infty X \xi_0) + p_-\bigl(\xi_\infty F(X\xi_0)\bigr) +\theta J_{\sigma_{k-1}(\lambda_\xi)}F(X\xi_0)$. A homogeneous $\xi_i$ of spectral degree $m$ maps $\mathcal K_{[-q]}$ to $\mathcal K_{[-(q-m)]}$ (or to zero if $m>q$), so both actions are continuous in the spectral product topology.

To see that the actions preserve $\mathcal K(U_0,U_\infty)$, express each action in the opposite PBW ordering:
$$
(\xi_0\cdot_0F)^R(X) =F^R(X\xi_0)+p_-\bigl(\xi_0F^R(X)\bigr), \quad (\xi_\infty\cdot_{\infty}F)^R(X) =-F^R(\xi_\infty X).
$$

For $X\in U(\mathfrak D_+)$ and $Y\in U(\mathfrak D_-)$, by the coalgebra property of $p_-$, and the identity $ \sum_{(X)}X_{(1)}S(X_{(2)})\otimes X_{(3)} = 1\otimes X$, we have
$$
p_-\bigl(S(X)J_fY\bigr) = \sum_{(X)} J_{f\circ p_-(X_{(1)}\,\cdot)}\bigl(p_-(S(X_{(2)})Y)\bigr).
$$
Substituting into the definition of $F^R$ yields
$$
(J_fF)^R(X) =\sum_{(X)} J_{f\circ p_-(X_{(1)}\,\cdot)}\bigl(F^R(X_{(2)})\bigr).
$$
Since $\mathfrak D_+^i\oplus\mathfrak D_-^i$ is a Lie subalgebra of $\mathfrak D$ and $U(\mathfrak D_-^i)$ is a subcoalgebra of
$U(\mathfrak D_-)$, we have $ p_-\bigl(\xi_iU(\mathfrak D_-^i)\bigr)
\subseteq U(\mathfrak D_-^i)$, $J_fU(\mathfrak D_-^i)\subseteq U(\mathfrak D_-^i)$. The above identities therefore show that both actions preserve the two image conditions.
\end{proof}

For $\theta=0$, write $\mathcal R^{\mathrm{MT}}:=\mathcal R_0^{\mathrm{MT}}$. A map $F\in\mathcal K(U_0,U_\infty)$ belongs to $\mathcal R^{\mathrm{MT}}$ precisely when
\begin{equation*}
F\bigl(U(\mathfrak D_+)\mathfrak D_+^0\bigr)=0,
\qquad
F^R\bigl(\mathfrak D_+^\infty U(\mathfrak D_+)\bigr)=0.
\end{equation*}

\begin{theorem}
\label{thm:mt-kop-comp}
The coproduct $\Delta$ and counit $\epsilon$ on $\mathcal Y$ make the space of sphere one-point conformal blocks $\mathscr B_\theta^{\mathrm{MT}}$ an associative topological algebra with product and unit
$$
(\lambda\star\mu)(v) :=(\lambda\otimes\mu)\Delta(v), \qquad \mathbf1_{\mathscr B}=\epsilon.
$$
Moreover, the invariant subspace $\mathcal R_\theta^{\mathrm{MT}}$ is a closed unital subalgebra of $(\mathcal K,\diamond)$, and there is a canonical isomorphism of topological algebras
$$
\mathscr B_\theta^{\mathrm{MT}} \xrightarrow{\sim}\mathcal R_\theta^{\mathrm{MT}}, \qquad \lambda\longmapsto F_\lambda, \qquad F_{\lambda\star\mu}=F_\lambda\diamond F_\mu.
$$
\end{theorem}

\subsection{Proof of the main theorem}
\label{sec:mt-comp-proof}

This subsection is devoted to the proof of Thm.~\ref{thm:mt-kop-comp}. The statement of the theorem and the strategy of its proof are inspired by the computations of Costello and Paquette \cite{CostelloPaquetteCelestial} for a different chiral algebra, where conformal blocks are computed via two-chart \v{C}ech descriptions on twistor space. Here, we adapt their two-chart description to the present vertex algebra using its associative mode algebra. This lets us express the coordinate-invariance conditions and compute the convolution product.

First, we realize the space of sphere one-point conformal blocks $\mathscr B_\theta^{\mathrm{MT}}$ as the dual of a quotient of the completed associative mode algebra $\widetilde U(\mathcal Y)$ by vacuum and outgoing relations (Prop.~\ref{prop:mt-vac-quotient}). Second, using the PBW basis of $\mathcal Y$, we represent each block $\lambda \in \mathscr B_\theta^{\mathrm{MT}}$ as a map $F_\lambda \in \mathcal K = \operatorname{Hom}_{\mathbb C}(U(\mathfrak D_+), \widehat{U(\mathfrak D_-)})$, and show that the two-chart coordinate invariance corresponds precisely to the invariant subspace $\mathcal R_\theta^{\mathrm{MT}} = \mathcal K(U_0, U_\infty)^{\mathfrak D_+^0 \oplus \mathfrak D_+^\infty}$. Finally, we compare the vertex coproduct $\Delta$ with the Koppinen smash product $\diamond$, establishing that $F_{\lambda \star \mu} = F_\lambda \diamond F_\mu$, which proves that $\mathcal R_\theta^{\mathrm{MT}}$ is a closed subalgebra of $(\mathcal K, \diamond)$ and that $\lambda \mapsto F_\lambda$ is an isomorphism of topological algebras.

We begin with the mode-algebra of a vertex algebra. Following \cite[Secs.~4.1--4.3]{FrenkelBenZvi}, the Lie algebra of modes of $\mathcal Y$ is the quotient $\operatorname{Lie}(\mathcal Y) := (\mathcal Y \otimes \kk [t^{\pm}]/\operatorname{Im}(\partial\otimes \mathrm{id} + \mathrm{id} \otimes \partial_t))$. The class of $a\otimes t^k$ is denoted $[a\otimes t^k]$, which acts on $\mathcal Y$ as the $k$-th Fourier mode $a_{(k)}$ of the field $Y(a,z) = \sum_{k\in\mathbb Z}a_{(k)}z^{-k-1}$.

The completed associative algebra of modes $\widetilde U(\mathcal Y)$ is the completion of the universal enveloping algebra $U(\operatorname{Lie}(\mathcal Y))$ (with neighborhoods of zero given by left ideals generated by modes of sufficiently high degree), quotiented by the relations ensuring that the generating fields $\mathsf Y[a,z] := \sum_{k\in\mathbb Z}[a\otimes t^k]z^{-k-1}$ satisfy the vertex-algebra axioms. These include
\begin{equation*}
[\mathbf1\otimes t^{-1}]=1,
\qquad
\mathsf Y[a_{(-1)}b,z] =:\!\mathsf Y[a,z]\mathsf Y[b,z]\!:,
\quad a,b\in\mathcal Y.
\end{equation*}
The first identifies the $(-1)$-mode of the vacuum with the algebra unit, and the second expresses each mode of $a_{(-1)}b$ as a sum of products of modes of $a$ and $b$.

We fix $\theta$ and use $\mathfrak g_{\mathrm{out}}(\mathcal Y)$ for the outgoing Lie algebra of $\{\mathscr L_n^{\mathrm{MT},\theta}\}$. Combining the vacuum-module realization $\mathcal Y \cong \widetilde U(\mathcal Y)/I_{\mathrm{vac}}$ \cite[Sec.~4.3]{FrenkelBenZvi} with the definition of coinvariants yields the presentation of one-point conformal blocks as a quotient of the mode algebra:

\begin{proposition}
\label{prop:mt-vac-quotient}
Let $\operatorname{ev}_{\mathbf1}:\widetilde U(\mathcal Y) \longrightarrow\mathcal Y$, $P\longmapsto P\mathbf1$, be the vacuum-evaluation map, and set $I_{\mathrm{vac}}:=\ker(\operatorname{ev}_{\mathbf1})$. Then $\operatorname{ev}_{\mathbf1}$ induces a vector-space isomorphism
\begin{equation}
\frac{\widetilde U(\mathcal Y)}
{I_{\mathrm{vac}}+
\mathfrak g_{\mathrm{out}}(\mathcal Y)\widetilde U(\mathcal Y)}
\xrightarrow{\sim}
\frac{\mathcal Y}{\mathfrak g_{\mathrm{out}}(\mathcal Y)\mathcal Y},
\label{eq:mt-vac-quotient}
\end{equation}
and hence
\begin{equation}
\mathscr B_\theta^{\mathrm{MT}}
\cong
\left(
\frac{\widetilde U(\mathcal Y)}
{I_{\mathrm{vac}}+
\mathfrak g_{\mathrm{out}}(\mathcal Y)\widetilde U(\mathcal Y)}
\right)^\vee.
\label{eq:mt-mode-block-quotient}
\end{equation}
Under this identification, a block $\lambda\in\mathscr B_\theta^{\mathrm{MT}}$ corresponds to the linear functional on $\widetilde U(\mathcal Y)$ given by $P\longmapsto\lambda(P\mathbf1)$, $P\in\widetilde U(\mathcal Y)$.
\end{proposition}

\begin{proof}
By the creation axiom, $\operatorname{ev}_{\mathbf1}([a\otimes t^{-1}]) = a_{(-1)}\mathbf1 = a$, so $\operatorname{ev}_{\mathbf1}$ is surjective. Because $\operatorname{ev}_{\mathbf1}(xQ) = x(Q\mathbf1) = x\operatorname{ev}_{\mathbf1}(Q)$ for $x\in\mathfrak g_{\mathrm{out}}(\mathcal Y)$ and $Q\in\widetilde U(\mathcal Y)$, the preimage of $\mathfrak g_{\mathrm{out}}(\mathcal Y)\mathcal Y$ is $I_{\mathrm{vac}} + \mathfrak g_{\mathrm{out}}(\mathcal Y)\widetilde U(\mathcal Y)$. The quotient isomorphism \eqref{eq:mt-vac-quotient} follows, and dualizing yields \eqref{eq:mt-mode-block-quotient}.
\end{proof}

\begin{remark}
Although $\widetilde U(\mathcal Y)$ is a topological algebra whose natural dual is the continuous dual, the algebraic dual $(-)^\vee$ in \eqref{eq:mt-mode-block-quotient} requires no topological restriction. Indeed, the left ideal generated by the modes that annihilate the vacuum vector $\mathbf1$ is open in the mode topology \cite[Sec.~3.2.3]{FrenkelLoopGroups} and contained in $I_{\mathrm{vac}}$. Consequently, the quotient \eqref{eq:mt-vac-quotient} inherits the discrete topology, so its continuous dual coincides with its full algebraic dual.
\end{remark}

For a (bi)graded vector space $V$, we will use the notation $V^*$ for the graded dual:
\begin{equation*}
V^*:=\bigoplus_{p,r}(V_{[-p],-r})^\vee,
\qquad
(V^*)_{[p],r}=(V_{[-p],-r})^\vee;
\end{equation*}
while for a singly graded space, we omit the second grading in the above formula. The residue pairing identifies $\dminus^*\cong\g^\vee[u]\,du$ so that $(\g u^{-d})^\vee$ is identified with $\g^\vee u^{d-1}du$. We extend the residue pairing on $u$ to a double residue pairing involving $t$, which identifies $\mathfrak D_-^* = \dminus^*[t,t^{-1}]dt$. We give $\mathfrak D_-^*$ the dual bidegree. Equivalently, giving $du$ and $dt$ the bidegrees $(1,0)$ and $(0,-1)$ respectively so the pairing has bidegree $(0,0)$. We will also omit the symbol $du,dt$ when there is no ambiguity. 

We now describe the generators of $\widetilde U(\mathcal Y)$ explicitly. For each $\ell\in\dminus^*$, let $a_\ell \in \mathcal O(G_-) \cong \operatorname{Sym}(\dminus^*)$ be the corresponding linear function. We have algebra homomorphisms
\begin{equation*}
\begin{aligned}
\iota_-:\operatorname{Sym}(\mathfrak D_-^*) &\longrightarrow\widetilde U(\mathcal Y),&
\ell t^k &\longmapsto [a_\ell\otimes t^k],\\
\iota_+:U(\mathfrak D_+) &\longrightarrow\widetilde U(\mathcal Y),&
\xi t^k &\longmapsto [\Lcur_\xi\otimes t^k].
\end{aligned}
\end{equation*}
For $f\in\operatorname{Sym}(\mathfrak D_-^*)$ and $X\in U(\mathfrak D_+)$, we write $fX := \iota_-(f)\iota_+(X)$ and $Xf := \iota_+(X)\iota_-(f)$. Since the states $a_\ell$ and $\Lcur_\xi$ strongly generate $\mathcal Y$, the images of $\iota_\pm$ topologically generate $\widetilde U(\mathcal Y)$. On $\mathcal Y$, the classes $[a_\ell\otimes t^k]$ and $[\Lcur_\xi\otimes t^k]$ act as $(a_\ell)_{(k)}$ and $(\Lcur_\xi)_{(k)}$, respectively.

Under the evaluation map $\operatorname{ev}_{\mathbf1}$, non-negative modes act trivially on the vacuum vector: $\iota_-(\ell t^k)\mathbf1 = \iota_+(\xi t^k)\mathbf1 = 0$ for all $k\geq0$. Identifying $(\mathfrak D_-^0)^* = t^{-1}\dminus^*[t^{-1}] \subset \mathfrak D_-^*$, the vertex-envelope PBW isomorphism gives:
\begin{equation}
\operatorname{Sym}((\mathfrak D_-^0)^*)\otimes U(t^{-1}\dplus[t^{-1}]) \xrightarrow{\sim} \mathcal Y,
\qquad f\otimes X\longmapsto fX\mathbf1.
\label{eq:mt-coord-current-PBW}
\end{equation}

To formulate the vanishing conditions on conformal blocks, let $E_i := (\mathfrak{D}_-^i)^\perp \subset \mathfrak D_-^*$ be the annihilator of $\mathfrak{D}_-^i$ under the residue pairing for $i \in \{0,\infty\}$. Explicitly,
\begin{equation*}
E_0=\bigoplus_{\substack{d\geq1\\r<0}}(\mathfrak D_-^*)_{[d],r},
\qquad
E_\infty=\bigoplus_{\substack{d\geq1\\r>2d}}(\mathfrak D_-^*)_{[d],r}.
\end{equation*}
Under this decomposition, $E_0$ and $\mathfrak D_+^0$ are precisely the non-negative modes that annihilate the vacuum vector $\mathbf1$. Conversely, the space $\mathfrak g_{\mathrm{out}}(\mathcal Y)\mathcal Y$ is characterized as follows:

\begin{lemma}
\label{lem:mt-outgoing-gens}
For the cocycle-twisted mini-twistor quasi-conformal structure,
$$
\mathfrak g_{\mathrm{out}}(\mathcal Y)\mathcal Y = E_\infty\mathcal Y +\mathfrak D_+^\infty\cdot_\theta\mathcal Y,
$$
where $E_\infty$ acts via $\iota_-$, while the $\cdot_\theta$ action is defined by
$$
(\xi t^k)\cdot_\theta v := \left( (\Lcur_\xi)_{(k)} -\theta(\lambda_\xi)_{(k-1)}\right)v, \qquad \xi\in\g u^m,\quad m\geq0,\quad k\leq-2m,
$$
and extended linearly.
\end{lemma}

\begin{proof}
As in \eqref{eq:1pt-outgoing-op}, the action of $\mathfrak g_{\mathrm{out}}$ is given by the opposite modes
$$
c_n^o=(-1)^\Delta\sum_{j\geq0}\frac1{j!} \bigl((\mathscr L_1^{\mathrm{MT},\theta})^jc\bigr)_{(2\Delta-n-j-2)}, \qquad n\geq0,
$$
for homogeneous $c\in\mathcal Y$ of weight $\Delta=\operatorname{wt}_{\mathrm{MT}}(c)$. The strong generators $a_\ell$ and $\Lcur_\xi$ have weights $-d$ and $1-m$, respectively, for $\ell\in\g^\vee u^{d-1}$ and $\xi\in\g u^m$. Because $a_\ell$ remains primary, while $\mathscr L_1^{\mathrm{MT},\theta}\Lcur_\xi=-\theta\lambda_\xi$ and $\mathscr L_1^{\mathrm{MT},\theta}\lambda_\xi=0$,
the sum collapses to the $j=0$ term for $a_\ell$ and the $j=0,1$ terms for $\Lcur_\xi$. As $n$ ranges over $\mathbb Z_{\geq0}$, the resulting operators $(a_\ell)_n^o=(-1)^{-d}(a_\ell)_{(-2d-n-2)}$, $(\Lcur_\xi)_n^o =(-1)^{1-m}\left((\Lcur_\xi)_{(-2m-n)}-\theta(\lambda_\xi)_{(-2m-n-1)}\right)=(-1)^{1-m}\bigl(\xi t^{-2m-n}\bigr)\cdot_\theta$ span $E_\infty$ and $\mathfrak D_+^\infty\cdot_\theta$, respectively, proving the inclusion $ E_\infty\mathcal Y+\mathfrak D_+^\infty\cdot_\theta\mathcal Y \subseteq \mathfrak g_{\mathrm{out}}(\mathcal Y)\mathcal Y$.

For the reverse inclusion, use the opposite-mode identities \cite[Eqs.~(2.57)--(2.62)]{HuangLepowskyZhangI}, with opposite
modes extended to all integer indices as in Lemma~\ref{lem:block-ext-prod-descent}:
$$
\begin{aligned}
(\partial a)_n^o&=-na_{n-1}^o,\\
(a_{(-1)}b)_n^o
&=\sum_{j<0}b_{n-j-1}^o a_j^o
+\sum_{j\geq0}a_j^o b_{n-j-1}^o.
\end{aligned}
$$
In the expansion of $(a_{(-1)}b)_n^o$ for $n\geq0$, every term on the right-hand side has a leftmost operator with non-negative index $k\geq0$. Since $\mathcal Y$ is strongly generated by $a_\ell$ and $\Lcur_\xi$, an induction on the generation of states shows that every $c_n^o$ for $n\geq0$ acts as an operator in the right ideal generated by $E_\infty$ and $\mathfrak D_+^\infty\cdot_\theta$. This establishes $ \mathfrak g_{\mathrm{out}}(\mathcal Y)\mathcal Y \subseteq E_\infty\mathcal Y +\mathfrak D_+^\infty\cdot_\theta\mathcal Y$, completing the proof.
\end{proof}


We extend the residue pairing between $\mathfrak D_-^*$ and $\mathfrak D_-$ via PBW symmetrization $\operatorname{sym}:\operatorname{Sym}(\mathfrak D_-)\xrightarrow{\sim}U(\mathfrak D_-)$ to a pairing between $\operatorname{Sym}(\mathfrak D_-^*)$ and $U(\mathfrak D_-)$, which we continue to denote by $\langle-,-\rangle$. For $f\in\operatorname{Sym}(\mathfrak D_-^*)$ and $y_1,\ldots,y_n\in\mathfrak D_-$,
\begin{equation}
\left\langle f,\operatorname{sym}(y_1\cdots y_n)\right\rangle
=
\left.
\partial_{s_1}\cdots\partial_{s_n}
f\!\left(\sum_{i=1}^n s_i y_i\right)
\right|_{s_1=\cdots=s_n=0},
\qquad \langle f,1\rangle=f(0),
\end{equation}
where $\operatorname{sym}(y_1\cdots y_n) = \frac{1}{n!}\sum_{\sigma\in S_n}y_{\sigma(1)}\cdots y_{\sigma(n)}$ averages over all $n!$ orderings.

We also consider the graded dual $U(\mathfrak D_-)^*$ with the convolution product
\begin{equation*}
\langle fg,Y\rangle
=\langle f\otimes g,\Delta(Y)\rangle,
\qquad f,g\in U(\mathfrak D_-)^*,\quad  Y\in U(\mathfrak D_-).
\end{equation*}
For the operators $J_f$ introduced in Sec.~\ref{sec:mt-blocks}, this gives $\langle g,J_fY\rangle=\langle fg,Y\rangle$.

Since PBW symmetrization is a coalgebra isomorphism, the above pairing identifies $\operatorname{Sym}(\mathfrak D_-^*)$ with a subalgebra of $U(\mathfrak D_-)^*$. An element of $\operatorname{Sym}(\mathfrak D_-^*)$ is a finite linear combination of monomials, whereas a homogeneous element of $U(\mathfrak D_-)^*$ of bidegree $(d, R)$ is a formal sum of monomials in the commuting symbols $\ell t^k$ whose bidegrees add up to $(d, R)$.

Equip $(U(\mathfrak D_-)^*)_{[d],R} =(U(\mathfrak D_-)_{[-d],-R})^\vee$ with the topology whose neighborhoods of zero are
\begin{equation*}
W^\perp:=\{f:\langle f,W\rangle=0\},
\qquad
W\subset U(\mathfrak D_-)_{[-d],-R}
\ \text{finite-dimensional}.
\end{equation*}
The continuous dual of $(U(\mathfrak D_-)^*)_{[d],R}$ is canonically $U(\mathfrak D_-)_{[-d],-R}$. The subspace $\operatorname{Sym}(\mathfrak D_-^*)_{[d],R}$ is dense: every such $W$ is contained in the span of finitely many monomials, and the restriction of any $f\in(U(\mathfrak D_-)^*)_{[d],R}$ to that span is represented by a finite linear combination of the dual monomials.

Applying Prop.~\ref{prop:kop-matched-pair} to the matched pair $\mathfrak D=\mathfrak D_-\bowtie\mathfrak D_+$, the enveloping algebra $U(\mathfrak D_-)$ is a right $U(\mathfrak D_+)$-module coalgebra with action $Y\triangleleft X := p_-(S(X)Y)$. Its transpose defines a left action on $U(\mathfrak D_-)^*$:
\begin{equation}
\langle X\triangleright f,Y\rangle
:=\langle f,p_-(S(X)Y)\rangle,
\qquad Y\in U(\mathfrak D_-).
\label{eq:mt-coord-action}
\end{equation}
This action also makes $U(\mathfrak D_-)^*$ a left $U(\mathfrak D_+)$-module algebra:
\begin{equation}
\begin{aligned}
(XX')\triangleright f &=X\triangleright(X'\triangleright f),\\
X\triangleright(fg) &=\sum_{(X)}(X_{(1)}\triangleright f) (X_{(2)}\triangleright g).
\end{aligned}
\label{eq:mt-transpose-mod-alg}
\end{equation}
If $f\in(U(\mathfrak D_-)^*)_{[d],R}$ and $X\in U(\mathfrak D_+)_{[p],K}$, then $X\triangleright f\in(U(\mathfrak D_-)^*)_{[d+p],R+K}$. The following proposition realizes this action by commutation in $\widetilde U(\mathcal Y)$.

\begin{proposition}
\label{prop:mt-mode-realization}
The map $\iota_-$ extends uniquely to a continuous algebra homomorphism
$U(\mathfrak D_-)^*\to\widetilde U(\mathcal Y)$.
For $X\in U(\mathfrak D_+)$ and $f\in U(\mathfrak D_-)^*$, one has
\begin{equation}
\begin{aligned}
\iota_+(X)\iota_-(f)
&=\sum_{(X)}
\iota_-(X_{(1)}\triangleright f)\iota_+(X_{(2)}),\\
\iota_-(f)\iota_+(X)
&=\sum_{(X)}
\iota_+(X_{(1)})
\iota_-(S(X_{(2)})\triangleright f).
\end{aligned}
\label{eq:mt-mode-reorder}
\end{equation}
In particular, these relations hold as locally finite operator identities when acting on $\mathcal Y$.
\end{proposition}

\begin{proof}
Let $f\in(U(\mathfrak D_-)^*)_{[d],R}$, and express it as a formal sum of monomials like $e_1\cdots e_n$, where $e_j=\ell_jt^{-r_j-1}$, $\ell_j\in\g^\vee u^{d_j-1}$, and $d_j\geq1$. Then $\sum_jd_j=d$, $\sum_jr_j=R$, and $n\leq d$. For each $N>0$, there are only finitely many integer solutions to $r_j\geq-N$ and $\sum_jr_j=R$, so only finitely many of these monomials satisfy $r_j\geq-N$ for every $j$. Every other monomial contains a factor of mode degree $-r_j-1\geq N$, and hence maps into the open left ideal generated by modes of degree at least $N$. Thus $\iota_-(f)$ is defined modulo every such ideal by a finite sum, and define an element of $\widetilde U(\mathcal Y)$. The inverse image of this open ideal contains the annihilator of the finite-dimensional span of the monomials dual to the finitely many terms above. Hence $\iota_-$ is continuous for this topology, and density of $\operatorname{Sym}(\mathfrak D_-^*)$ gives its uniqueness. Continuity of convolution and of multiplication in $\widetilde U(\mathcal Y)$ shows that the extension is an algebra homomorphism.

On a state $v\in\mathcal Y_{[q],h}$, we have 
\begin{equation*}
\iota_-(e_1)\cdots\iota_-(e_n)v\ne0 \quad\Longrightarrow\quad n\leq d,\qquad -h\leq r_j\leq R+(d-1)h.
\end{equation*}
Indeed, commuting each factor next to $v$ gives $r_j\geq-h$, and $\sum_jr_j=R$ gives the upper bound. Only finitely many monomials satisfy these inequalities, so $\iota_-(f)$ acts by a finite sum and shifts bidegree
by $(d,R)$.

The maps $\sigma_n:A\to U(\mathfrak D_-)^*$ introduced in Sec.~\ref{sec:mt-blocks} have the monomial expression
\begin{equation*}
\sigma_n(a_{\ell_1}\cdots a_{\ell_N}) = \sum_{i_1+\cdots+i_N=n+1-N} (\ell_1t^{i_1})\cdots(\ell_Nt^{i_N}), \qquad \sigma_n(1)=\delta_{n,-1}.
\end{equation*}
For $P\in A_{[d]}$, this lies in
$(U(\mathfrak D_-)^*)_{[d],-n-1}$.
The defining normal-product relation gives $[P\otimes t^n] = \iota_-\!\left(\sigma_n(P)\right)$ in $\widetilde U(\mathcal Y)$. We compare \eqref{eq:mt-coord-action} with multiplication in $\widetilde U(\mathcal Y)$. By \eqref{eq:mt-coord-action} and the identity $\sigma_j(P)\circ p_-(\xi t^k\,\cdot) =-\sigma_{k+j}(\rho_\xi P)$ proved in Prop.~\ref{prop:mt-kop-subspace}, we have,
for arbitrary $k\in\mathbb Z$,
$$
(\xi t^k)\triangleright\sigma_j(P)
=\sigma_{k+j}(\rho_\xi P).
$$
Since $\sigma_j(a_\ell)=\ell t^j$, we obtain $(\xi t^k)\triangleright(\ell t^j) =\sigma_{k+j}(\rho_\xi a_\ell)$. The mode commutator formula now gives
\begin{equation*}
\bigl[\iota_+(\xi t^k),\iota_-(\ell t^j)\bigr] = [\rho_\xi a_\ell\otimes t^{k+j}] = \iota_-\!\left( (\xi t^k)\triangleright(\ell t^j) \right).
\end{equation*}
For fixed $X=\xi t^k$, \eqref{eq:mt-transpose-mod-alg} extends this commutator identity to $f\in\operatorname{Sym}(\mathfrak D_-^*)$. The map $X\triangleright-$ is continuous for the topology on $U(\mathfrak D_-)^*$: for every finite-dimensional $W\subset U(\mathfrak D_-)$, $(X\triangleright-)^{-1}(W^\perp) =p_-(S(X)W)^\perp$. Consequently, both maps
\begin{equation*}
f\longmapsto[\iota_+(X),\iota_-(f)],
\qquad
f\longmapsto\iota_-(X\triangleright f)
\end{equation*}
are continuous. They agree on the dense subspace
$\operatorname{Sym}(\mathfrak D_-^*)$, and therefore agree on
$U(\mathfrak D_-)^*$.
Induction on the number of factors in a PBW monomial $X$ gives the first identity in \eqref{eq:mt-mode-reorder}; the antipode identity gives the second.
\end{proof}

\begin{proposition}
\label{prop:mt-vec-space-comp}
Via the vacuum-evaluation isomorphism of Prop.~\ref{prop:mt-vac-quotient}, the assignment $\lambda\mapsto F_\lambda$ uniquely characterized by
\begin{equation}
\langle f,F_\lambda(X)\rangle = \lambda(fX\mathbf1),
\qquad f\in\operatorname{Sym}(\mathfrak D_-^*),\ X\in U(\mathfrak D_+)
\label{eq:mt-mode-fn-PBW-pairing}
\end{equation}
defines a vector-space isomorphism preserving spectral degree
\begin{equation}
\mathscr B_\theta^{\mathrm{MT}} \cong \left(
\frac{\widetilde U(\mathcal Y)}
{I_{\mathrm{vac}}+
 \mathfrak g_{\mathrm{out}}(\mathcal Y)\widetilde U(\mathcal Y)}
\right)^\vee
\xrightarrow{\sim}\mathcal R_\theta^{\mathrm{MT}},
\qquad \lambda \longmapsto F_\lambda.
\label{eq:mt-mode-fn-comp}
\end{equation}
Moreover, the reordered map $F_\lambda^R$ defined in \eqref{eq:mt-reorder-map} satisfies
\begin{equation}
\langle f,F_\lambda^R(X)\rangle = \lambda(Xf\mathbf1).
\label{eq:mt-reorder-mode-fn-pairing}
\end{equation}
\end{proposition}

\begin{proof}
Via Prop.~\ref{prop:mt-vac-quotient}, regard $\lambda\in\mathscr B_\theta^{\mathrm{MT}}$ equivalently as the linear functional $\Lambda(P) := \lambda(P\mathbf1)$ on $\widetilde U(\mathcal Y)$ vanishing on $I_{\mathrm{vac}}+\mathfrak g_{\mathrm{out}}(\mathcal Y)\widetilde U(\mathcal Y)$, so $\Lambda(fX) = \lambda(fX\mathbf1)$ and $\Lambda(Xf) = \lambda(Xf\mathbf1)$.

Because both sides of the expected isomorphism are products over spectral degree, suppose first that $\Lambda$ has spectral degree $-q$. Let $X\in U(\mathfrak D_+)$ have bidegree $(p,h)$ with $X\mathbf1\ne0$ (so $h\geq 0$). For basis elements $e_i=\ell_it^{-r_i-1}$ with $\ell_i\in \g^\vee u^{d_i-1}$ ($d_i\geq1$), a nonzero evaluation $\Lambda(e_1\cdots e_nX)\ne0$ requires
\begin{equation*}
\sum_{i=1}^n d_i=q-p,\qquad -h\leq r_i\leq2d_i.
\end{equation*}
The lower bound holds by commuting $e_i$ to $X\mathbf1$; the upper bound holds by commuting it to the far left, since $e_i\in E_\infty\subset\mathfrak g_{\mathrm{out}}(\mathcal Y)$ for $r_i>2d_i$. Since $d_i\geq1$, only finitely many monomials $e_1\cdots e_n$ satisfy these bounds, so \eqref{eq:mt-mode-fn-PBW-pairing} determines a unique element $F_\lambda(X)\in U(\mathfrak D_-)_{[p-q]}$. Applying the reordering identity \eqref{eq:mt-mode-reorder} to $Xf$ gives
\begin{equation*}
\Lambda(Xf)
=\sum_{(X)}
\langle X_{(1)}\triangleright f,F_\lambda(X_{(2)})\rangle
=\langle f,F_\lambda^R(X)\rangle,
\end{equation*}
which proves \eqref{eq:mt-reorder-mode-fn-pairing}.

The image conditions in \eqref{eq:mt-1chart-algs} correspond to annihilators under this pairing: for $i\in\{0,\infty\}$, we have
\begin{equation*}
\left\{Y\in\widehat{U(\mathfrak D_-)}: \langle \operatorname{Sym}(\mathfrak D_-^*)E_i, Y\rangle=0\right\}=\widehat{U(\mathfrak D_-^i)}.
\end{equation*}

Now let $\xi_0\in\mathfrak D_+^0$, $e_0\in E_0$, $\xi_\infty=\xi t^k\in\mathfrak D_+^\infty$ , $e_\infty\in E_\infty$, $f\in\operatorname{Sym}(\mathfrak D_-^*)$, and $X\in U(\mathfrak D_+)$. The defining relations give $fX\xi_0,Xfe_0\in I_{\mathrm{vac}}$ and $e_\infty fX,\bigl(\iota_+(\xi_\infty)-\theta\iota_-(\sigma_{k-1}(\lambda_\xi))\bigr)fX \in \mathfrak g_{\mathrm{out}}(\mathcal Y)\widetilde U(\mathcal Y)$. Since $\Lambda$ vanishes on these subspaces, the defining formulas for $F_\lambda$ and $F_\lambda^R$ imply
$$
\begin{aligned}
\langle f,F_\lambda(X\xi_0)\rangle&=0, \quad \left\langle f,(\xi_\infty\cdot_{\infty,\theta}F_\lambda)(X)\right\rangle=0,\\
\langle fe_0,F_\lambda^R(X)\rangle&=0, \quad \langle fe_\infty,F_\lambda(X)\rangle =\Lambda(e_\infty fX)=0.
\end{aligned}
$$
The first row gives the invariance condition under the actions \eqref{eq:mt-chart-actions}. The second row gives the two image inclusions $\operatorname{im}F_\lambda^R \subseteq\widehat{U(\mathfrak D_-^0)}$ and $\operatorname{im}F_\lambda \subseteq\widehat{U(\mathfrak D_-^\infty)}$. Hence $F_\lambda\in\mathcal R_\theta^{\mathrm{MT}}$.

Conversely, let $F\in(\mathcal R_\theta^{\mathrm{MT}})_{[-q]}$. We first define a map $\Phi_F:U(\mathfrak D_-)^*\otimes U(\mathfrak D_+) \to \kk$ via 
\begin{equation*}
\Phi_F(h\otimes X):=\langle h,F(X)\rangle
\end{equation*}
The vacuum evaluation gives us a map $m_{\mathbf1}:U(\mathfrak D_-)^*\otimes U(\mathfrak D_+)\to \mathcal Y$, defined by $m_{\mathbf1}(h\otimes X):=\iota_-(h)\iota_+(X)\mathbf1$.  For homogeneous $X$, the map $h\mapsto\Phi_F(h\otimes X)$ is continuous because $F(X)\in U(\mathfrak D_-)$, while $h\mapsto m_{\mathbf1}(h\otimes X)$ is continuous by Prop.~\ref{prop:mt-mode-realization}.

To show that $\Phi_F$ descends to a map $\lambda_F:\mathcal{Y} \to \kk$, we only need to show that $\ker m_{\mathbf1} \subset \ker\Phi_F $. The vacuum relations are generated by $E_0$ and $\mathfrak D_+^0$. By the commutation relation \eqref{eq:mt-mode-reorder}, $\ker m_{\mathbf1}$ is the closed span of 
\begin{equation*}
h\otimes X\xi_0,\quad \sum_{(X)}g(X_{(1)}\triangleright b)\otimes X_{(2)},
\end{equation*}
for $h,g\in U(\mathfrak D_-)^*$, $X\in U(\mathfrak D_+)$, $\xi_0\in\mathfrak D_+^0$, $b\in E_0$.  The defining conditions on $F$ give
\begin{equation*}
\Phi_F(h\otimes X\xi_0)=\langle h,F(X\xi_0)\rangle=0,
\end{equation*}
and
\begin{equation*}
\Phi_F\left(\sum_{(X)}g(X_{(1)}\triangleright b)\otimes X_{(2)}\right)=\sum_{(X)}\left\langle(S(X_{(2)})\triangleright g)b,F^R(X_{(1)})\right\rangle =0.
\end{equation*}
For the last equality, each homogeneous term $F^R(X_{(1)})$ belongs
to $U(\mathfrak D_-^0)$, whereas
\begin{equation*}
\langle ab,Y\rangle
=\sum_{(Y)}\langle a,Y_{(1)}\rangle\langle b,Y_{(2)}\rangle=0,
\qquad
a\in U(\mathfrak D_-)^*,\quad b\in E_0,\quad
Y\in U(\mathfrak D_-^0).
\end{equation*}
Thus the continuous functional $\Phi_F$ vanishes on $\ker m_{\mathbf1}$ and defines a unique $\lambda_F\in\mathcal Y^\vee$ of spectral degree $-q$. It satisfies
\begin{equation}
\begin{aligned}
\lambda_F\bigl(\iota_-(h)\iota_+(X)\mathbf1\bigr) &=\langle h,F(X)\rangle,\\
\lambda_F\bigl(\iota_+(X)\iota_-(h)\mathbf1\bigr) &=\left\langle h, \sum_{(X)}p_-\bigl(S(X_{(1)})F(X_{(2)})\bigr) \right\rangle=\langle h,F^R(X)\rangle.
\end{aligned}
\label{eq:mt-conv-mode-pairings}
\end{equation}

It remains to prove $\lambda_F(\mathfrak g_{\mathrm{out}}(\mathcal Y)\mathcal Y)=0$. Using Lem.~\ref{lem:mt-outgoing-gens}, it suffices to prove
\begin{equation*}
\lambda_F(e_\infty fX\mathbf1)=0, \qquad \lambda_F\bigl(\xi_\infty\cdot_\theta(fX\mathbf1)\bigr)=0.
\end{equation*}
Since $fe_\infty\in\operatorname{Sym}(\mathfrak D_-^*)E_\infty$ and $F(X)\in\widehat{U(\mathfrak D_-^\infty)}$, we have $\lambda_F(e_\infty fX\mathbf1)=\langle fe_\infty,F(X)\rangle=0$. We also have $\lambda_F\bigl(\xi_\infty\cdot_\theta(fX\mathbf1)\bigr)
=-\left\langle f,(\xi_\infty\cdot_{\infty,\theta}F)(X)\right\rangle=0$, because $F\in\mathcal R_\theta^{\mathrm{MT}}$. Hence $\lambda_F\in\mathscr B_\theta^{\mathrm{MT}}$.

For $f\in\operatorname{Sym}(\mathfrak D_-^*)$ and $X\in U(\mathfrak D_+)$, the first identity in \eqref{eq:mt-conv-mode-pairings} gives
\begin{equation*}
\langle f,F_{\lambda_F}(X)\rangle
 =\lambda_F(fX\mathbf1)=\langle f,F(X)\rangle.
\end{equation*}
Thus $F_{\lambda_F}=F$. For $f\in\operatorname{Sym}((\mathfrak D_-^0)^*)$ and $X\in U(t^{-1}\dplus[t^{-1}])$, the defining formulas give
\begin{equation*}
\lambda_{F_\lambda}(fX\mathbf1)
 =\langle f,F_\lambda(X)\rangle=\lambda(fX\mathbf1).
\end{equation*}
By \eqref{eq:mt-coord-current-PBW}, $\lambda_{F_\lambda}=\lambda$. The two assignments preserve spectral degree, so their direct product over $q\geq0$ is the topological vector-space isomorphism \eqref{eq:mt-mode-fn-comp}.
\end{proof}

\begin{proposition}
\label{prop:mt-conv-comp}
The linear isomorphism $\lambda\mapsto F_\lambda$ of Prop.~\ref{prop:mt-vec-space-comp} intertwines the mini-twistor convolution product $\star$ on $\mathscr B_\theta^{\mathrm{MT}}$ with the Koppinen product $\diamond$ on $\mathcal R_\theta^{\mathrm{MT}}$:
\begin{equation}
\langle f, F_{\lambda\star\mu}(X)\rangle
=(\lambda\otimes\mu)\Delta(fX\mathbf1)
=\langle f,(F_\lambda\diamond F_\mu)(X)\rangle,
\qquad f\in\operatorname{Sym}(\mathfrak D_-^*),\ X\in U(\mathfrak D_+).
\label{eq:mt-ordered-mode-coprod}
\end{equation}
Consequently, $\mathcal R_\theta^{\mathrm{MT}}$ is a closed unital subalgebra of $(\mathcal K,\diamond)$, and $\lambda\mapsto F_\lambda$ defines an isomorphism of topological algebras $(\mathscr B_\theta^{\mathrm{MT}},\star)\xrightarrow{\sim}(\mathcal R_\theta^{\mathrm{MT}},\diamond)$.
\end{proposition}

\begin{proof}
By continuity in the spectral product topology, it suffices to verify \eqref{eq:mt-ordered-mode-coprod} for homogeneous blocks $\lambda\in\mathscr B_{\theta,[-i]}^{\mathrm{MT}}$, $\mu\in\mathscr B_{\theta,[-j]}^{\mathrm{MT}}$ ($i+j=q$), $f\in\operatorname{Sym}(\mathfrak D_-^*)$, and PBW monomials $X\in U(\mathfrak D_+)$. Write $\Delta_U$ for the primitive coproducts on $U(\mathfrak D_\pm)$; $\Delta$ denotes the vertex coproduct.

The transpose of $m^*$ under $\langle-,-\rangle$ is multiplication in $U(\mathfrak D_-)$, i.e.
\begin{equation}
\langle\Delta(f),Y\otimes Z\rangle=\langle f,YZ\rangle,
\qquad Y,Z\in U(\mathfrak D_-).
\label{eq:mt-coord-coprod-duality}
\end{equation}
We prove by induction on the number of factors in a PBW monomial $X$
that
\begin{equation}
(\lambda\otimes\mu)\Delta(fX\mathbf1)
=
\sum_{(X),(Y)}
\left\langle
f,
Y_{(2)}F_\mu\bigl(X_{(2)}\triangleleft Y_{(1)}\bigr)
\right\rangle,
\qquad
Y:=F_\lambda(X_{(1)}),
\label{eq:mt-conv-induction}
\end{equation}
for every $f\in\operatorname{Sym}(\mathfrak D_-^*)$, where
$\Delta_U(X)=\sum_{(X)}X_{(1)}\otimes X_{(2)}$ and
$\Delta_U(Y)=\sum_{(Y)}Y_{(1)}\otimes Y_{(2)}$.
For $X=1$, $1\triangleleft Z=\epsilon(Z)1$, and
\eqref{eq:mt-coord-coprod-duality} gives
\eqref{eq:mt-conv-induction}.

For $\xi\in\dplus$, and write $M_g\xi=(\operatorname{Ad}_{g^{-1}}\xi)_+ =\sum_i a_i(g)\xi_i$ with $a_i\in A$, $\xi_i\in\dplus$. For $y_j\in\dminus$ and $r_j\in\mathbb Z$ $(1\leq j\leq n)$, put $Z_t:=\operatorname{sym}(y_1t^{r_1}\cdots y_nt^{r_n})\in U(\mathfrak D_-)$, and $Z_u:=\operatorname{sym}(y_1\cdots y_n)\in U(\dminus)$. We have
\begin{equation}
\begin{aligned}
&(\Delta\Lcur_\xi)_{(k)}=(\Lcur_\xi)_{(k)}\otimes1 +\sum_{i,s}\iota_-(\sigma_s(a_i))  \otimes(\Lcur_{\xi_i})_{(k-s-1)},\\
&\sum_{i,s}\langle\sigma_s(a_i),Z_t\rangle\,\xi_it^{k-s-1} =t^{k+\sum_jr_j} \sum_i\langle a_i,Z_u\rangle\xi_i =(\xi t^k) \triangleleft Z_t.
\end{aligned}
\label{eq:mt-gen-coprod-action}
\end{equation}
Since the elements $Z_t$ span $U(\mathfrak D_-)$, \eqref{eq:mt-coord-coprod-duality} and \eqref{eq:mt-gen-coprod-action} give \eqref{eq:mt-conv-induction} for $X=\xi t^k$. By linearity, this holds for every $X\in\mathfrak D_+$.

Assume \eqref{eq:mt-conv-induction} for $X$ and let $X'\in\mathfrak D_+$. For $h\in U(\mathfrak D_-)^*$ and $Z\in U(\mathfrak D_-)$, we have 
\begin{equation}
\iota_+(X')\iota_-(h)=\iota_-(h)\iota_+(X')+\iota_-(X'\triangleright h),\quad \langle X'\triangleright h,Z\rangle=-\langle h,p_-(X' Z)\rangle.
\label{eq:mt-gen-mode-reorder}
\end{equation}
The PBW decomposition and
$X\triangleleft Z=p_+(S(Z)X)$ give the following identities
(see Lem.~\ref{lem:kop-current-word}):
\begin{equation}
\begin{aligned}
X'Z
&=p_-(X' Z)+\sum_{(Z)}Z_{(2)}(X'\triangleleft Z_{(1)}),\\
S(Z)X'
&=-S\bigl(p_-(X' Z)\bigr)+\sum_{(Z)}(X'\triangleleft Z_{(1)})S(Z_{(2)}),\\
(X' X)\triangleleft Z
&=\sum_{(Z)}(X'\triangleleft Z_{(1)}) (X\triangleleft Z_{(2)})-X\triangleleft p_-(X' Z).
\end{aligned}
\label{eq:mt-action-induction}
\end{equation}
For the primitive coproduct,
\begin{equation*}
\Delta_U(X'X)=\sum_{(X)}X'X_{(1)}\otimes X_{(2)}+\sum_{(X)}X_{(1)}\otimes X'X_{(2)}.
\end{equation*}
The induction hypothesis combined with  \eqref{eq:mt-gen-coprod-action}, \eqref{eq:mt-gen-mode-reorder} and \eqref{eq:mt-action-induction} give \eqref{eq:mt-conv-induction} for $X'X$. Thus \eqref{eq:mt-conv-induction} holds for every PBW monomial $X$, and its right-hand side equals $\langle f,(F_\lambda\diamond F_\mu)(X)\rangle$ by definition.

Summing over $i+j=q$ extends the identity to full blocks. Since $\lambda\mapsto F_\lambda$ is a linear isomorphism onto $\mathcal R_\theta^{\mathrm{MT}}$ and $\mathscr B_\theta^{\mathrm{MT}}$ is closed under $\star$ (Prop.~\ref{prop:block-ord-descent}), the relation $F_\lambda\diamond F_\mu = F_{\lambda\star\mu}$ implies that $\mathcal R_\theta^{\mathrm{MT}}$ is a closed unital subalgebra of $(\mathcal K,\diamond)$ and $\lambda\mapsto F_\lambda$ is an isomorphism of topological algebras, completing the proof of Thm.~\ref{thm:mt-kop-comp}.
\end{proof}

\subsection{A filtration on blocks and its classical limit}
\label{sec:mt-classical-limit}

We first recall the completed associated graded construction for filtered complete algebras. Let $B = \prod_{n\geq0}B_n$ be the degree completion of an $\mathbb N$-graded algebra $B^\circ = \bigoplus_{n\geq0}B_n$, equipped with the inverse-limit product topology. Suppose that $B$ is endowed with an increasing multiplicative filtration $P_\bullet$ compatible with the grading:
\begin{equation*}
P_r B = \prod_{n\geq0} P_r B_n, \qquad P_{-1}B = 0,
\end{equation*}
whose restriction to each homogeneous subspace $B_n$ is finite and exhaustive ($P_{N(n)}B_n = B_n$).

For a general element $(b_n)_{n\geq0}\in B$, the filtration degree of $b_n$ may grow arbitrarily with $n$. The ordinary associated graded algebra $\operatorname{gr}_P B = \bigoplus_{r\geq0} P_r B / P_{r-1} B$ captures only elements with a uniform filtration bound across all degrees. In the category of complete topological algebras, the natural associated graded object is instead the \emph{completed associated graded algebra} (cf.~\cite{QuillenRationalHomotopy}):
\begin{equation*}
\widehat{\operatorname{gr}}_P B
:= \varprojlim_N \operatorname{gr}_P\left( B\big/\prod_{n\geq N}B_n \right)
= \prod_{n\geq0}\operatorname{gr}_P B_n.
\end{equation*}
Equivalently, $\widehat{\operatorname{gr}}_P B$ is the degree completion of the ordinary associated graded algebra $\operatorname{gr}_P B^\circ = \bigoplus_{n\geq0}\operatorname{gr}_P B_n$ of the dense graded subalgebra.

For minitwistor blocks, we apply this construction to the spectral degree completion $\mathscr B_\theta^{\mathrm{MT}} = \prod_{q\geq0}\mathscr B_{\theta,[-q]}^{\mathrm{MT}}$. By the PBW decomposition \eqref{eq:mt-coord-current-PBW}, every state in $\mathcal Y$ is a linear combination of vectors $fX\mathbf1$ with $f\in\operatorname{Sym}((\mathfrak D_-^0)^*)$ and $X\in U(t^{-1}\dplus[t^{-1}])$. We filter blocks by the polynomial degree in the coordinate variables $f$:
\begin{equation}
P_r\mathscr B_{\theta,[-q]}^{\mathrm{MT}}
:= \{\lambda\in\mathscr B_{\theta,[-q]}^{\mathrm{MT}} : \lambda(fX\mathbf1)=0 \text{ for all } f \in \operatorname{Sym}^{>r}((\mathfrak D_-^0)^*)\},
\;
P_r\mathscr B_\theta^{\mathrm{MT}} := \prod_{q\geq0} P_r\mathscr B_{\theta,[-q]}^{\mathrm{MT}},
\label{eq:mt-block-filtration}
\end{equation}
with $P_{-1}=0$. Since each coordinate mode $\ell t^k$ has strictly positive spectral degree, $P_q\mathscr B_{\theta,[-q]}^{\mathrm{MT}} = \mathscr B_{\theta,[-q]}^{\mathrm{MT}}$, so the filtration is finite and exhaustive on each spectral component, yielding the completed associated graded algebra $\widehat{\operatorname{gr}}_P\mathscr B_\theta^{\mathrm{MT}} = \prod_{q\geq0}\operatorname{gr}_P\mathscr B_{\theta,[-q]}^{\mathrm{MT}}$.

To describe the classical limit, consider the sheaf of Lie algebras on $\mathbb{MT}$:
\begin{equation*}
\mathscr G_{\mathrm{MT}}
:= (\g\otimes\cO_{\mathbb{MT}}) \ltimes (\g^\vee\otimes p^*\cO_{\mathbb P^1}(-4)),
\end{equation*}
where the abelian second summand is equipped with the coadjoint action of the first. The double-residue pairing identifies $\mathfrak D_-^*$ with $\Gamma(\widetilde U_{01},\g^\vee\otimes p^*\cO_{\mathbb P^1}(-4))$. With $E_i = (\mathfrak D_-^i)^\perp$, the Lie algebras of sections on the overlap and the two charts are
\begin{equation*}
\mathscr G_{01} := \Gamma(\widetilde U_{01},\mathscr G_{\mathrm{MT}}) = \mathfrak D_+\ltimes\mathfrak D_-^*,
\qquad
\mathscr G_i := \Gamma(\widetilde U_i,\mathscr G_{\mathrm{MT}}) = \mathfrak D_+^i\ltimes E_i
\quad (i=0,\infty).
\end{equation*}
These are precisely the relations among the symbols of the generator modes, with $\mathfrak D_+$ of $P$-degree zero and $\mathfrak D_-^*$ of $P$-degree one.

\begin{theorem}
\label{thm:mt-classical-limit}
The filtration \eqref{eq:mt-block-filtration} is multiplicative and almost commutative:
\begin{equation*}
P_r\star P_s\subseteq P_{r+s}, \qquad [P_r,P_s]_\star\subseteq P_{r+s-1}.
\end{equation*}
Moreover, there is an isomorphism of complete commutative algebras
\begin{equation}
\widehat{\operatorname{gr}}_P\mathscr B_\theta^{\mathrm{MT}}
\cong
\bigl(U(\mathscr G_{01})^\vee\bigr)^{
\mathscr G_\infty\oplus\mathscr G_0}
\cong
\left(
\frac{U(\mathscr G_{01})}
{\mathscr G_\infty U(\mathscr G_{01})
+U(\mathscr G_{01})\mathscr G_0}
\right)^\vee.
\label{eq:mt-double-invariants}
\end{equation}
where $\vee$ is the full algebraic dual with convolution product for the enveloping coproduct. The isomorphism preserves spectral degree and identifies $P$-degree with the degree contributed by $\mathfrak D_-^*$.
\end{theorem}

Before proving this, we first establish a normal form for the coinvariants. Define the Lie subalgebras
\begin{equation*}
\mathfrak a := t^{-1}\g[u,t^{-1}]\cap\mathfrak D_+^\infty,
\qquad
\mathfrak b := \bigoplus_{m\geq1}\bigoplus_{j=1}^{2m-1}\g u^mt^{-j}.
\end{equation*}
Then $t^{-1}\g[u,t^{-1}]=\mathfrak a\oplus\mathfrak b$ as vector spaces, $\mathfrak D_+^\infty=\g\ltimes\mathfrak a$, and $\mathfrak b \cong \mathfrak D_+/(\mathfrak D_+^0+\mathfrak D_+^\infty)$ as bigraded $\g$-modules. Because $\mathfrak D_-^*/(E_0+E_\infty) \cong \bigoplus_{d\geq1}\bigoplus_{r=0}^{2d} (\g^\vee u^{d-1})t^{-r-1}$, we have
\begin{equation*}
\mathcal Y/E_\infty\mathcal Y
\cong \operatorname{Sym}\!\left(\mathfrak D_-^*/(E_0+E_\infty)\right)\otimes U(t^{-1}\g[u,t^{-1}]),
\end{equation*}

Since $p_-(XU(\mathfrak D_-^\infty)) \subseteq U(\mathfrak D_-^\infty)$ for $X\in\mathfrak D_+^\infty$, the action $X\triangleright f$ on $f\in E_\infty$ pairs to zero with $U(\mathfrak D_-^\infty)$, so $[\mathfrak D_+^\infty,E_\infty]\mathcal Y \subseteq E_\infty\mathcal Y$. Since $\lambda_\xi\in A$ have regular OPEs with generators in $A$, $[(\lambda_\xi)_{(k-1)},e_\infty]=0$ for every $e_\infty\in E_\infty$. Hence $\mathcal Y/E_\infty\mathcal Y$ is a $\mathfrak D_+^\infty$-module for $\cdot_\theta$, and Lem.~\ref{lem:mt-outgoing-gens} identifies the full coinvariants with $\mathcal Y/\mathfrak g_{\mathrm{out}}(\mathcal Y)\mathcal Y \cong ((\mathcal Y/E_\infty\mathcal Y)_\mathfrak a)_\g$.

Here and below, $\mathfrak a$-coinvariants are taken with respect to $\cdot_\theta$. The $\g$-action is unchanged, since $\lambda_\xi\in A_{[0]}=\kk$ and $\lambda_\xi(e)=0$ for $\xi\in\g$.

\begin{lemma}
\label{lem:mt-normal-form}
The map
\begin{equation*}
\operatorname{Sym}\!\left(\mathfrak D_-^*/(E_0+E_\infty)\right)\otimes U(\mathfrak b)
\xrightarrow{\sim} \left(\mathcal Y/E_\infty\mathcal Y\right)_{\mathfrak a},
\qquad f\otimes X\longmapsto[fX\mathbf1],
\end{equation*}
is an isomorphism of bigraded $\g$-modules. Moreover, it is a filtered isomorphism for the descending filtration by the number of $\mathfrak D_-^*$ factors, with $\mathfrak D_+$ of filtration degree zero.
\end{lemma}

\begin{proof}
Put $I:=\operatorname{Sym}^{\geq1}\!\left(\mathfrak D_-^*/(E_0+E_\infty)\right)$, where we regard $\operatorname{Sym}(\mathfrak D_-^*/(E_0+E_\infty))$ as the subalgebra of $\operatorname{Sym}(\mathfrak D_-^*)$ generated by $\alpha u^{d-1}t^{-r-1}$ with $\alpha \in\g^\vee$, $d\geq1$, and $0\leq r\leq2d$. By the reordering \eqref{eq:mt-mode-reorder}, the action $\cdot_\theta$ gives an isomorphism of left $U(\mathfrak a)$-modules
$$
U(\mathfrak a)\otimes
\operatorname{Sym}\!\left(\mathfrak D_-^*/(E_0+E_\infty)\right)
\otimes U(\mathfrak b)
\xrightarrow{\sim}\mathcal Y/E_\infty\mathcal Y,
\qquad X'\otimes f\otimes X\longmapsto X'\cdot_\theta(fX\mathbf1).
$$

Taking $\mathfrak a$-coinvariants gives
$$
(\mathcal Y/E_\infty\mathcal Y)_{\mathfrak a}
\cong
\operatorname{Sym}\!\left(\mathfrak D_-^*/(E_0+E_\infty)\right)
\otimes U(\mathfrak b),
$$
and the resulting map is $f\otimes X\mapsto[fX\mathbf1]$. For $X'\in\mathfrak a$, \eqref{eq:mt-coord-action} gives $\langle X'\triangleright f,1\rangle
=\langle f,p_-(S(X'))\rangle=0$, so \eqref{eq:mt-transpose-mod-alg} implies $X'\triangleright I^p\subseteq \overline{\operatorname{Sym}^{\geq p}(\mathfrak D_-^*)}$, where the closure is taken in each bidegree of $U(\mathfrak D_-)^*$. Moreover, $\lambda_\xi(e)=0$, so $\sigma_{k-1}(\lambda_\xi)\in \overline{\operatorname{Sym}^{\geq1}(\mathfrak D_-^*)}$. Hence, by induction on PBW degree in the current factors using the two identities in \eqref{eq:mt-mode-reorder}, the formula defining $\cdot_\theta$, and the vacuum relations, the above isomorphism and its inverse preserve the filtrations by $I^p$, and the isomorphism is filtered. Finally, multiplication and the actions $\triangleright$ and $\cdot_\theta$ preserve the bigrading and are $\g$-equivariant.
\end{proof}

\begin{proof}[Proof of Thm.~\ref{thm:mt-classical-limit}]
We denote the space of coinvariants by $ \mathscr{Q}_\theta^{\mathrm{MT}}:=\mathcal Y/\mathfrak g_{\mathrm{out}}(\mathcal Y)\mathcal Y$. We define a filtration $F$ on $\mathcal{Y}$ by
\begin{equation*}
F^p\mathcal Y:= \operatorname{span}\left\{ fX\mathbf1: f\in\operatorname{Sym}^{\geq p}((\mathfrak D_-^0)^*), X\in U(t^{-1}\dplus[t^{-1}])\right\}.
\end{equation*}
Let $F^p\mathscr{Q}_\theta^{\mathrm{MT}}$ be the image of $F^p\mathcal Y$, and set
\begin{equation*}
F^p(\mathcal Y\otimes\mathcal Y)
:=\sum_{i+j=p}F^i\mathcal Y\otimes F^j\mathcal Y,
\qquad
F^p(\mathscr{Q}_\theta^{\mathrm{MT}}\otimes\mathscr{Q}_\theta^{\mathrm{MT}}):=\sum_{i+j=p}F^i\mathscr{Q}_\theta^{\mathrm{MT}}\otimes F^j\mathscr{Q}_\theta^{\mathrm{MT}}.
\end{equation*}
Since $\mathscr B_{\theta,[-q]}^{\mathrm{MT}}=(\mathscr{Q}_{\theta,[q]}^{\mathrm{MT}})^\vee$,
\begin{equation*}
P_r\mathscr B_{\theta,[-q]}^{\mathrm{MT}}
=\{\lambda\in (\mathscr{Q}_{\theta,[q]}^{\mathrm{MT}})^\vee:\lambda|_{(F^{r+1}\mathscr{Q}_\theta^{\mathrm{MT}})_{[q]}}=0\}
\cong\bigl(\mathscr{Q}_{\theta,[q]}^{\mathrm{MT}}/(F^{r+1}\mathscr{Q}_\theta^{\mathrm{MT}})_{[q]}\bigr)^\vee.
\end{equation*}
For $f\in\mathfrak D_-^*$ and $X\in\mathfrak D_+$, the coproduct formulas give
\begin{equation}
\begin{aligned}
\Delta(\iota_-(f)\mathbf1)
-\bigl(\iota_-(f)\mathbf1\otimes\mathbf1
+\mathbf1\otimes\iota_-(f)\mathbf1\bigr)
&\in F^2(\mathcal Y\otimes\mathcal Y),\\
\Delta(\iota_+(X)\mathbf1)
-\bigl(\iota_+(X)\mathbf1\otimes\mathbf1
+\mathbf1\otimes\iota_+(X)\mathbf1\bigr)
&\in F^1(\mathcal Y\otimes\mathcal Y).
\end{aligned}
\label{eq:mt-filtered-coprod-gens}
\end{equation}
Consequently, the coproduct induced on $\mathscr{Q}_\theta^{\mathrm{MT}}$ satisfies
\begin{equation*}
\Delta(F^p\mathscr{Q}_\theta^{\mathrm{MT}})\subseteq F^p(\mathscr{Q}_\theta^{\mathrm{MT}}\otimes\mathscr{Q}_\theta^{\mathrm{MT}}),
\quad
(\Delta-\Delta^{\mathrm{op}})(F^p\mathscr{Q}_\theta^{\mathrm{MT}})\subseteq F^{p+1}(\mathscr{Q}_\theta^{\mathrm{MT}}\otimes\mathscr{Q}_\theta^{\mathrm{MT}}).
\end{equation*}
Taking annihilators gives
\begin{equation*}
P_r\star P_s\subseteq P_{r+s},
\qquad
[P_r,P_s]_\star\subseteq P_{r+s-1}.
\end{equation*}

Since $\lambda_\xi(e)=0$, the reordering identities \eqref{eq:mt-mode-reorder} give $ (\lambda_\xi)_{(k-1)}F^p\mathcal Y
\subseteq F^{p+1}\mathcal Y$, for  $\xi t^k\in\mathfrak D_+^\infty$. Thus $(\Lcur_\xi)_{(k)}-\theta(\lambda_\xi)_{(k-1)}$
and $(\Lcur_\xi)_{(k)}$ induce the same operator on $\operatorname{gr}_F\mathcal Y$. Since $\mathfrak D_+^0\cap\mathfrak D_+^\infty=\g$, PBW gives a bigraded vector-space isomorphism
\begin{equation*}
\frac{U(\mathscr G_{01})}
{\mathscr G_\infty U(\mathscr G_{01})+U(\mathscr G_{01})\mathscr G_0}
\cong
\left(
\operatorname{Sym}\!\left(\mathfrak D_-^*/(E_0+E_\infty)\right)
\otimes U(\mathfrak b)
\right)_\g.
\end{equation*}
Under this isomorphism, $[fX]$ corresponds to $[f\otimes X]$. The identification $\mathscr{Q}_\theta^{\mathrm{MT}}\cong((\mathcal Y/E_\infty\mathcal Y)_{\mathfrak a})_\g$ and Lem.~\ref{lem:mt-normal-form} give the filtered isomorphism
\begin{equation*}
\left(
\operatorname{Sym}\!\left(\mathfrak D_-^*/(E_0+E_\infty)\right)
\otimes U(\mathfrak b)
\right)_\g
\xrightarrow{\sim}\mathscr{Q}_\theta^{\mathrm{MT}},
\qquad
[f\otimes X]\longmapsto[fX\mathbf1].
\end{equation*}
The $\g$-action preserves the polynomial degree. Hence, for $f\in\operatorname{Sym}^p\!\left(\mathfrak D_-^*/(E_0+E_\infty)\right)$ and $X\in U(\mathfrak b)$, the associated graded of the composition of the above two isomorphisms is
\begin{equation}
\begin{aligned}
\frac{U(\mathscr G_{01})}
{\mathscr G_\infty U(\mathscr G_{01})+U(\mathscr G_{01})\mathscr G_0}
&\xrightarrow{\sim}\operatorname{gr}_F \mathscr{Q}_\theta^{\mathrm{MT}},\\
[fX]&\longmapsto[fX\mathbf1]+F^{p+1}\mathscr{Q}_\theta^{\mathrm{MT}}.
\end{aligned}
\label{eq:mt-coinv-comp}
\end{equation}
Every element of $\mathscr G_0\cup\mathscr G_\infty$ is primitive, so $\mathscr G_\infty U(\mathscr G_{01})+U(\mathscr G_{01})\mathscr G_0$
is a coideal contained in the kernel of the counit. For $f\in\mathfrak D_-^*$ and $X\in\mathfrak D_+$, after projection to $\mathscr{Q}_\theta^{\mathrm{MT}}$,
\eqref{eq:mt-filtered-coprod-gens} implies that
$[\iota_-(f)\mathbf1]+F^2\mathscr{Q}_\theta^{\mathrm{MT}}\in F^1\mathscr{Q}_\theta^{\mathrm{MT}}/F^2\mathscr{Q}_\theta^{\mathrm{MT}}$ and
$[\iota_+(X)\mathbf1]+F^1\mathscr{Q}_\theta^{\mathrm{MT}}\in F^0\mathscr{Q}_\theta^{\mathrm{MT}}/F^1\mathscr{Q}_\theta^{\mathrm{MT}}$ are primitive. Since
$\Delta(a_{(n)}v)=\Delta(a)_{(n)}\Delta(v)$, induction on PBW length shows that
\eqref{eq:mt-coinv-comp} intertwines the coproducts. Hence it is an isomorphism of cocommutative coalgebras. For every spectral degree $q$,
\begin{equation*}
\operatorname{gr}_P\mathscr B_{\theta,[-q]}^{\mathrm{MT}}
\cong\bigl((\operatorname{gr}_F\mathscr{Q}_\theta^{\mathrm{MT}})_{[q]}\bigr)^\vee
\cong
\left(
\frac{U(\mathscr G_{01})}
{\mathscr G_\infty U(\mathscr G_{01})+U(\mathscr G_{01})\mathscr G_0}
\right)_{[q]}^\vee.
\end{equation*}
Finally,
\begin{equation*}
\left(
\frac{U(\mathscr G_{01})}
{\mathscr G_\infty U(\mathscr G_{01})+U(\mathscr G_{01})\mathscr G_0}
\right)^\vee
=\bigl(U(\mathscr G_{01})^\vee\bigr)^{\mathscr G_\infty\oplus\mathscr G_0}.
\end{equation*}
Taking the product over the spectral degrees proves
\eqref{eq:mt-double-invariants}.
\end{proof}

The above description is in direct analogy with the result in \cite{CostelloPaquetteCelestial}: it is the algebra of formal functions on the overlap, invariant under the gauge transformations on $U_0$ and $U_\infty$. The algebra also has a natural cohomological description. The \v{C}ech complex for $\{\widetilde U_0,\widetilde U_\infty\}$ and Serre duality yield
\begin{equation*}
H^0(\mathbb{MT},\mathscr G_{\mathrm{MT}}) = \g,
\qquad
H^1(\mathbb{MT},\mathscr G_{\mathrm{MT}}) \cong \frac{\mathfrak D_+}{\mathfrak D_+^0+\mathfrak D_+^\infty} \oplus \frac{\mathfrak D_-^*}{E_0+E_\infty}
\end{equation*}
as spectrally graded $\g$-modules.

\begin{corollary}
\label{cor:mt-cohomology}
There is an isomorphism of completed commutative algebras
\begin{equation*}
\widehat{\operatorname{gr}}_P\mathscr B_\theta^{\mathrm{MT}}
\cong \left(\operatorname{Sym}\!\left(H^1(\mathbb{MT},\mathscr G_{\mathrm{MT}})\right)^\vee\right)^\g,
\end{equation*}
endowed with convolution for the symmetric coalgebra.
\end{corollary}

\begin{proof}
By Lem.~\ref{lem:mt-normal-form}, the coalgebra in \eqref{eq:mt-double-invariants} is $(\operatorname{Sym}(\mathfrak D_-^*/(E_0+E_\infty))\otimes U(\mathfrak b))_\g$. PBW symmetrization $\operatorname{Sym}(\mathfrak b)\cong U(\mathfrak b)$ and the \v{C}ech identification of $\mathfrak b$ and $\mathfrak D_-^*/(E_0+E_\infty)$ with the two summands of $H^1(\mathbb{MT},\mathscr G_{\mathrm{MT}})$ identify this quotient coalgebra with $\operatorname{Sym}(H^1(\mathbb{MT},\mathscr G_{\mathrm{MT}}))_\g$. Taking the full dual in each spectral degree yields the isomorphism.
\end{proof}

\section*{Acknowledgements}
The author thanks Kevin Costello and Wenjun Niu for helpful discussions and valuable comments. This research was supported by the Center of Mathematical Sciences and Applications at Harvard University.

\appendix
\section{The Koppinen smash product}
\label{app:kop}

All tensor products and linear maps in this appendix are over $\kk$.
Let $A$ be a unital associative algebra with multiplication $\mu_A$, and let $C$ be a coassociative counital coalgebra with coproduct $\Delta_C$ and counit $\epsilon_C$, with no finite-dimensionality assumptions.

\begin{definition}
\label{def:kop-entwining}
An entwining structure consists of $A,C$ and a linear map
$\psi:C\otimes A\to A\otimes C$ satisfying
\begin{equation*}
\begin{aligned}
\psi(\id_C\otimes\mu_A)
&=(\mu_A\otimes\id_C)(\id_A\otimes\psi)(\psi\otimes\id_A),
&\psi(c\otimes1_A)&=1_A\otimes c,\\
(\id_A\otimes\Delta_C)\psi
&=(\psi\otimes\id_C)(\id_C\otimes\psi)(\Delta_C\otimes\id_A),
\qquad&
(\id_A\otimes\epsilon_C)\psi(c\otimes a)&=\epsilon_C(c)a.
\end{aligned}
\end{equation*}
We denote these data by $(A,C,\psi)$ and write
$\Delta_C(c)=\sum_{(c)}c_{(1)}\otimes c_{(2)}$ and
$\psi(c\otimes a)=\sum_\alpha a_\alpha\otimes c^\alpha$.
\end{definition}

\begin{proposition}
\label{prop:kop-assoc}
For an entwining structure $(A,C,\psi)$, the product
\begin{equation}
(F\diamond G)(c)
=\sum_{(c),\alpha}F(c_{(2)})_\alpha\,
G(c_{(1)}^\alpha),
\qquad F,G\in\operatorname{Hom}_{\kk}(C,A),
\label{eq:kop-entwined-prod}
\end{equation}
is associative, with unit $c\mapsto\epsilon_C(c)1_A$
\cite[Sec.~1]{BrzezinskiEntwiningCohomology}, where $\psi(c_{(1)}\otimes F(c_{(2)}))=\sum_\alpha F(c_{(2)})_\alpha\otimes c_{(1)}^\alpha$. We call this algebra the Koppinen smash product associated to $(A,C,\psi)$.
\end{proposition}

\begin{example}
\label{ex:kop-doi}
Let $H$ be a bialgebra, let $A$ be a right $H$-comodule algebra with unital coaction $\rho_A(a)=\sum a_{(0)}\otimes a_{(1)}$, and let $C$ be a right $H$-module coalgebra with action $\triangleleft$:
\begin{equation}
\Delta_C(c\triangleleft h)
=\sum_{(c),(h)}(c_{(1)}\triangleleft h_{(1)})
\otimes(c_{(2)}\triangleleft h_{(2)}),
\qquad
\epsilon_C(c\triangleleft h)=\epsilon_C(c)\epsilon_H(h).
\label{eq:kop-mod-coalgebra}
\end{equation}
These data yield the entwining structure \cite[Sec.~1]{SchauenburgDoiKoppinen}
\begin{equation*}
\psi(c\otimes a)
=\sum a_{(0)}\otimes(c\triangleleft a_{(1)}),
\end{equation*}
under which the Koppinen product \eqref{eq:kop-entwined-prod} becomes
\begin{equation*}
(F\diamond G)(c)
=\sum_{(c)}F(c_{(2)})_{(0)}
G\bigl(c_{(1)}\triangleleft F(c_{(2)})_{(1)}\bigr).
\end{equation*}
\end{example}

\begin{proposition}
\label{prop:kop-matched-pair}
Let $\mathfrak D=\mathfrak D_+\bowtie\mathfrak D_-$ be a matched pair of Lie algebras, so that $\mathfrak D=\mathfrak D_+\oplus\mathfrak D_-$ as vector spaces. Set $H=A=U(\mathfrak D_-)$ and $C=U(\mathfrak D_+)$, with their standard cocommutative coproducts $\Delta$, counits $\epsilon$, and antipodes $S$, and with $H$ coacting on $A=H$ via $\Delta$.
Let $p_+:U(\mathfrak D)\to U(\mathfrak D_+)$ be the PBW coalgebra projection characterized by $p_+(XY)=X\epsilon(Y)$ for $X\in U(\mathfrak D_+)$ and $Y\in U(\mathfrak D_-)$. The formula
\begin{equation}
X\triangleleft Y=p_+(S(Y)X)
\label{eq:kop-env-action}
\end{equation}
makes $C$ a right $H$-module coalgebra. Together with the coaction $\Delta$ on $A$, it gives the entwining
\begin{equation}
\psi(X\otimes Y)
=\sum_{(Y)}Y_{(1)}\otimes(X\triangleleft Y_{(2)})
=\sum_{(Y)}Y_{(2)}\otimes(X\triangleleft Y_{(1)}),
\label{eq:kop-env-entwining}
\end{equation}
and the associative Koppinen product on $\operatorname{Hom}_{\kk}(U(\mathfrak D_+),U(\mathfrak D_-))$:
\begin{equation}
(F\diamond G)(X)
=\sum_{(X),(Y)}Y_{(2)}G(X_{(2)}\triangleleft Y_{(1)}),
\qquad Y=F(X_{(1)}).
\label{eq:kop-env-prod}
\end{equation}
\end{proposition}

\begin{proof}
Since $\ker p_+=U(\mathfrak D)\mathfrak D_-$ is a left ideal,
\begin{equation*}
(X\triangleleft Y)\triangleleft Y'
=p_+\bigl(S(Y')S(Y)X\bigr)
=X\triangleleft(YY'),
\qquad X\triangleleft1=X.
\end{equation*}
Since $p_+$ and $S$ are coalgebra maps for the cocommutative coproduct, applying $\Delta$ and $\epsilon$ to \eqref{eq:kop-env-action} gives \eqref{eq:kop-mod-coalgebra}. Thus Ex.~\ref{ex:kop-doi} applies. Cocommutativity of $U(\mathfrak D_-)$ gives the second expression in \eqref{eq:kop-env-entwining}; substituting it into \eqref{eq:kop-entwined-prod} and using cocommutativity of $U(\mathfrak D_+)$ gives \eqref{eq:kop-env-prod}.
\end{proof}

\begin{lemma}[PBW change-of-basis isomorphism]
\label{lem:kop-opp-PBW}
For the opposite PBW order $U(\mathfrak D)\cong H\otimes C$ ($H=U(\mathfrak D_-)$, $C=U(\mathfrak D_+)$) with projection $\bar p_+(YX)=\epsilon(Y)X$, the action $X\triangleleft_R Y:=\bar p_+(XY)$ induces the associative product $(F\diamond_R G)(X):=\sum_{(X),(Y)}Y_{(2)}G(X_{(2)}\triangleleft_R Y_{(1)})$ on $\operatorname{Hom}(C,H)$, with $Y=F(X_{(1)})$.
The assignment $F\mapsto F^R$, characterized together with its inverse by
\begin{equation*}
F^R(X):=\sum_{(X)}p_-\bigl(S(X_{(1)})F(X_{(2)})\bigr),
\qquad
F(X)=\sum_{(X)}p_-\bigl(X_{(1)}F^R(X_{(2)})\bigr),
\end{equation*}
defines an algebra isomorphism
\begin{equation*}
(-)^R:(\operatorname{Hom}(C,H),\diamond)\xrightarrow{\sim}(\operatorname{Hom}(C,H),\diamond_R),
\qquad
(F\diamond G)^R=F^R\diamond_RG^R.
\end{equation*}
\end{lemma}

\begin{proof}
Associativity of $\diamond_R$ follows as in Prop.~\ref{prop:kop-matched-pair}, since $\ker\bar p_+=H^+U(\mathfrak D)$ is a right ideal. The antipode identities show that the two displayed formulas are mutually inverse. Expanding $XY$ in the two PBW factorizations,
\begin{equation*}
XY = \sum_{(X),(W)}W_{(2)}(X_{(2)}\triangleleft W_{(1)}) = \sum_{(X),(Y)}p_-(X_{(1)}Y_{(1)})(X_{(2)}\triangleleft_RY_{(2)})
\end{equation*}
(with $W=p_-(X_{(1)}Y)$), and substituting into \eqref{eq:kop-env-prod} yields $(F\diamond G)^R=F^R\diamond_R G^R$.
\end{proof}

The next lemma computes the matched-pair action when a Lie generator is added to a word.

\begin{lemma}
\label{lem:kop-current-word}
In the notation of Prop.~\ref{prop:kop-matched-pair}, for $\xi\in\mathfrak D_+$, $X\in U(\mathfrak D_+)$, and $Y\in U(\mathfrak D_-)$ one has
\begin{equation}
\xi Y=p_-(\xi Y)
+\sum_{(Y)}Y_{(2)}(\xi\triangleleft Y_{(1)}),
\label{eq:kop-PBW-reorder}
\end{equation}
and consequently
\begin{equation}
(\xi X)\triangleleft Y
=\sum_{(Y)}(\xi\triangleleft Y_{(1)})(X\triangleleft Y_{(2)})
-X\triangleleft p_-(\xi Y).
\label{eq:kop-word-recursion}
\end{equation}
\end{lemma}

\begin{proof}
Adjoint reordering in $U(\mathfrak D)$ gives $\xi Y=\sum_{(Y)}Y_{(1)}\ad_{S(Y_{(2)})}\xi$. Applying $p_+$ to $\ad_{S(Y)}\xi=\sum_{(Y)}S(Y_{(1)})\xi Y_{(2)}$ yields $p_+(S(Y)\xi)=\xi\triangleleft Y\in\mathfrak D_+$, as the rightmost factor is evaluated by $\epsilon$. Splitting each adjoint term into its negative and positive Lie parts gives $\xi Y = p_-(\xi Y) + \sum_{(Y)}Y_{(1)}(\xi\triangleleft Y_{(2)})=p_-(\xi Y)+\sum_{(Y)}Y_{(2)}(\xi\triangleleft Y_{(1)})$ by cocommutativity, proving \eqref{eq:kop-PBW-reorder}. Applying $S$ and multiplying on the right by $X$ gives $S(Y)\xi X = -S(p_-(\xi Y))X + \sum_{(Y)}(\xi\triangleleft Y_{(1)})S(Y_{(2)})X$; applying $p_+$ yields \eqref{eq:kop-word-recursion}.
\end{proof}

\end{document}